\documentclass[11pt]{amsart}

\usepackage{graphicx}

\usepackage{amsmath}
\usepackage{amssymb}
\usepackage{amsfonts}
\usepackage{amsopn}
\usepackage{ifthen}
\usepackage{algorithm}

\newtheorem{lemma}{Lemma} % this has to be removed for SIAM style
\newtheorem{Prop}{Proposition}

\usepackage[numbers,sort&compress]{natbib}

\newcommand{\EF}[1]{}

\newcommand{\deleted}[1]{}
\newcommand{\replaced}[2]{#1}

\newfont{\NUMBERS}{msbm8 scaled\magstep1}

\newcommand{\REAL}{\mbox{\NUMBERS R}}

\newcommand{\Argmin}{\operatornamewithlimits{argmin\vphantom{q}}}

\newcommand{\Vect}[2][]
{
  \ifthenelse{\equal{#1}{}}
  {\underline{#2}}
  {{#2}_{#1}}
}
\newcommand{\Matr}[2][]
{
  \ifthenelse{\equal{#1}{}}
  {\boldsymbol{#2}}
  {{#2}_{#1}}
}
\newcommand{\Span}{\ensuremath{\operatorname{span}}}
\newcommand{\Var}{\ensuremath{\operatorname{var}}}
\newcommand{\Err}{E}
\newcommand{\Wass}[1]{W_{#1}}
\newcommand{\ErrWass}{\Err_{\Wass{1}}}
\newcommand{\ErrTdens}{\Err_{\OptTdens}}
\newcommand{\RelDiff}{\rho}
\newcommand{\Tendsto}{\rightarrow}
\newcommand{\dx}{\, dx}

\DeclareMathOperator{\Div}{\nabla\cdot}
\DeclareMathOperator{\Grad}{\nabla}

\newcommand{\Dt}[1]{\partial_{t} #1}
\newcommand{\Lyap}{\mathcal{L}}
\newcommand{\SLyap}{\mathcal{S}}

\newcommand{\Ene}{\mathcal{E}_{\Forcing}}

\newcommand{\Wmass}{\mathcal{M}}
\newcommand{\InfSupLyap}{\Upsilon}

\newcommand{\Dim}{d}
\newcommand{\tstep}{k}
\newcommand{\tstepp}{k+1}
\newcommand{\Deltat}[1][]{\ifthenelse{\equal{#1}{}}{\Delta t_{\tstep}}{\Delta t_{#1}}}
\newcommand{\nlit}{m}
\newcommand{\nlitp}{m+1}
\newcommand{\PolySymb}{\mathcal{P}}
\newcommand{\PC}[1]{\ensuremath{\PolySymb_{#1}}}
\newcommand{\PONE}{\ensuremath{\PolySymb_{1}}}

\newcommand{\PZERO}{\ensuremath{\PolySymb_{0}}}
\newcommand{\MeshPar}{h}
\newcommand{\Cell}[1][]{\ifthenelse{\equal{#1}{}}{T}{T_{#1}}}
\newcommand{\Tsymb}{\mathcal{T}}
\newcommand{\Triang}[1][]
{
  \ifthenelse{\equal{#1}{}}
  {\Tsymb}
  {\Tsymb_{#1}}
}
\newcommand{\Domain}{\Omega}

\newcommand{\Ftest}{\varphi}

\newcommand{\PotOp}{\mathcal{U}_{\Forcing}}

\newcommand{\Lspacechar}{L}
\newcommand{\Lspace}[1]{\Lspacechar^{#1}}
\newcommand{\Lplus}[1]{\Lspace{#1}_{+}}

\newcommand{\Vof}[2]{[#1]^{#2}}

\newcommand{\Hspacechar}{H}
\newcommand{\Wspacechar}{W}
\newcommand{\Sob}[2][]{
  \ifthenelse{\equal{#2}{}}
  {\Hspacechar^{#1}}
  {\Wspacechar^{#1,#2}}  
}

\newcommand{\Hspace}[2][]{
  \ifthenelse{\equal{#2}{}}
  {\Hspacechar^{#1}}
  {\Hspacechar^{#1}_{#2}}  
}

\newcommand{\HolderExp}{\delta}
\newcommand{\Holder}[1][]{
  \ifthenelse{\equal{#1}{}}
  {\Cachar^{\HolderExp}}
  {\Cachar^{#1}}  
}

\newcommand{\Lip}[1]{\text{Lip}_{#1}}
\newcommand{\Cachar}{\mathcal{C}}
\newcommand{\Ca}[1][]{
  \ifthenelse{\equal{#1}{}}
  {\Cachar^{\delta}}
  {\Cachar^{\delta}(#1)}
}
\newcommand{\VspaceSymb}{\mathcal{V}}
\newcommand{\Vspace}[1][]{
  \ifthenelse{\equal{#1}{}}
  {\VspaceSymb_{\MeshPar}}
  {\VspaceSymb_{\MeshPar}(#1)}
}
\newcommand{\Vdim}{N}
\newcommand{\VbaseSymb}{\varphi}
\newcommand{\Vbase}[1][]{
  \ifthenelse{\equal{#1}{}}
  {\VbaseSymb}
  {\VbaseSymb_{#1}}
}
\newcommand{\WspaceSymb}{\mathcal{W}}
\newcommand{\Wspace}[1][]{
  \ifthenelse{\equal{#1}{}}
  {\WspaceSymb_{\MeshPar}}
  {\WspaceSymb_{\MeshPar}(#1)}
}
\newcommand{\Wdim}{M}
\newcommand{\WbaseSymb}{\psi}
\newcommand{\Wbase}[1][]{
  \ifthenelse{\equal{#1}{}}
  {\WbaseSymb}
  {\WbaseSymb_{#1}}
}
\newcommand{\Cont}[1][]{
  \ifthenelse{\equal{#1}{}}
  {\Cachar} 
  {\Cachar^{#1}}
}

\newcommand{\Opt}[1]{#1^*}

\newcommand{\Tdens}{\mu}
\newcommand{\TdensH}{\Tdens_{\MeshPar}}
\newcommand{\OptTdensH}{\Opt{\TdensH}}
\newcommand{\OptTdens}{\Opt{\Tdens}}
\newcommand{\Pot}{u}
\newcommand{\PotH}{\Pot_{\MeshPar}}
\newcommand{\OptPot}{\Opt{\Pot}}
\newcommand{\OptPotH}{\Opt{\PotH}}
\newcommand{\OptTime}{\Opt{t}}
\newcommand{\Vel}{v}
\newcommand{\OptVel}{\Opt{\Vel}}
\newcommand{\Field}{w}%\xi} MARIO: usiamo lettere greche per scalari
\newcommand{\Vectorfield}{Y}
\newcommand{\Map}{T}
\newcommand{\OptMap}{\Opt{\Map}}
\newcommand{\OptMapH}{\OptMap_{\MeshPar}}
\newcommand{\Plan}{\gamma}
\newcommand{\OptPlan}{\Opt{\Plan}}
\newcommand{\OptPlanH}{\OptPlan_{\MeshPar}}
\newcommand{\SnkhPar}{\varepsilon}
\newcommand{\SnkhOptPlanH}{\OptPlan_{\MeshPar,\SnkhPar}}

\newcommand{\Fsource}{f}
\newcommand{\Forcing}{\Fsource}        
\newcommand{\Source}{\Forcing^+}            
\newcommand{\Sink}{\Forcing^-}
\newcommand{\Fcont}{\Fsource_{1}}
\newcommand{\Fcost}{\Fsource_{2}}
\newcommand{\SourceH}{\underline{\Forcing}^+}
\newcommand{\SinkH}{\underline{\Forcing}^-}            
\newcommand{\DiscreteSink}{t}
\newcommand{\DiscreteSource}{s}
\newcommand{\DimSource}{N_{+}}            
\newcommand{\DimSink}{N_{-}}

\newcommand{\Dirac}{\delta}

\newcommand{\SuppMinus}{Q^-}
\newcommand{\SuppPlus}{Q^+}
\newcommand{\SuppCenter}{Q^c}
\newcommand{\SuppOT}{Q^{\Tdens}}
\newcommand{\SuppF}{Q^{\Fsource}}
\newcommand{\Supp}{\mbox{supp}}
\newcommand{\TolPCG}{\tau_{\mbox{{\scriptsize CG}}}}
\newcommand{\TolPic}{\tau_{\mbox{{\scriptsize NL}}}}
\newcommand{\ItPic}{m_{\mbox{{\scriptsize MAX}}}}
\newcommand{\TolTime}{\tau_{\mbox{{\scriptsize T}}}}
\newcommand{\LCF}{Lyapunov-candidate functional}

\newcommand{\MKEQS}{MK equations}
\newcommand{\OTP}{OTP}

\newcommand{\DMK}{DMK}
\newcommand{\OTD}{OT\  density}

\usepackage{cleveref}
\crefname{Prop}{Proposition}{Propositions}

\title[Numerical solution of dynamic MK equations]{
  Numerical Solution of Monge-Kantorovich
  Equations via a dynamic formulation}

\author{ENRICO FACCA}
\address{Department of Mathematics, University of Padua, 
 Padova, Italy
}
\email{facca@math.unipd.it}

\author{SARA DANERI}
\address{
  Department of Mathematics, University of Erlangen-N\"urnberg,
  Erlangen, Germany
}
\email{daneri@math.fau.de}

\author{FRANCO CARDIN}
\author{MARIO PUTTI}
\address{Department of Mathematics, University of Padua,
 Padova, Italy
}
\email{\{cardin,putti\}@math.unipd.it}

\begin{document}
% \headers{Numerical solution of dynamic \MKEQS}
% {Enrico Facca, Sara Daneri, Franco Cardin, and Mario Putti}

\maketitle

\begin{abstract}
  We extend our previous work on a biologically inspired dynamic
  Monge-Kantorovich model~\citep{Facca-et-al:2018} and propose it as an
  effective tool for the numerical solution of the $\Lspace{1}$-PDE
  based optimal transportation model.  Starting from the conjecture
  that the dynamic model is time-asymptotically equivalent to the
  Monge-Kantorovich equations governing $\Lspace{1}$ optimal
  transport, we experimentally analyze a simple yet effective
  numerical approach for the quantitative solution of these equations.

  The novel contributions in this paper are twofold. First, we
  introduce a new \LCF\ that better adheres to the dynamics of our
  proposed model.  It is shown that the Lie derivative of the new
  \LCF\ is strictly negative and, more remarkably, the \OTD\ is the
  unique minimizer for this new \LCF, providing further support to the
  conjecture of asymptotic equivalence of our dynamic model with the
  Monge-Kantorovich equations.  Second, we describe and test different
  numerical approaches for the solution of our problem. The ordinary
  differential equation for the transport density is projected into a
  piecewise constant or linear finite dimensional space defined on a
  triangulation of the domain.  The elliptic equation is discretized
  using a linear Galerkin finite element method defined on uniformly
  refined triangles.  The ensuing nonlinear differential-algebraic
  equation is discretized by means of a first order Euler method
  (forward or backward) and a simple Picard iteration is used to
  resolve the nonlinearity. The use of two discretization levels is
  dictated by the need to avoid oscillations on the potential
  gradients that prevent convergence of the scheme.

  We study the experimental convergence rate of the proposed solution
  approaches and discuss limitations and advantages of these
  formulations.  An extensive set of test cases, including problems
  that admit an explicit solution to the Monge-Kantorovich equations
  are appropriately designed to verify and test the expected numerical
  properties of the solution methods. Finally, a comparison with
  literature methods is performed and the ensuing transport maps are
  compared.  The results show that optimal convergence toward the
  asymptotic equilibrium point is achieved for sufficiently regular
  forcing function, and that the proposed method is accurate, robust,
  and computationally efficient.
\end{abstract}

% \keywords{
%   Monge-Kantorovich equations, Optimal Transport, Numerical Solution
% }

% \subjclass[2000]{
% 35M20 % PDES None of the above, but in this section
% 65M60 % Finite elements, Rayleigh-Ritz and Galerkin methods, finite
%       % methods
% 65M12 % Stability and convergence of numerical methods
% }

\section{Introduction}

We are interested in finding the numerical solution to the following
nonlinear differential problem. Given a domain
$\Domain\subset\REAL^{\Dim}$, two positive functions $\Source$ and
$\Sink$ belonging to $\Lspace{1}(\Domain)$ and such that
$\int_{\Domain}\Source\dx=\int_{\Domain}\Sink\dx$, find the pair
$(\Tdens,\Pot):[0,+\infty[\times\Domain\mapsto\REAL^{+}\times\REAL$
that satisfies:
\begin{subequations}
  \label{eq:sys-intro}
  \begin{align}
   &
   -\Div\left(\Tdens(t,x)\Grad\Pot(t,x)\right) 
   = 
   \Source(x)-\Sink(x)
   =
   \Fsource(x) 
   \label{eq:sys-intro-div}\\ 
   & \Dt{\Tdens}(t,x) = \Tdens(t,x)\left(|\Grad\Pot(t,x)|-1\right)
   \label{eq:sys-intro-dyn}\\
   & \Tdens(0,x) =\Tdens_0(x)> 0
   \label{eq:bound_cond_d} 	
 \end{align}
\end{subequations}
complemented by homogeneous Neumann boundary conditions
for~\cref{eq:sys-intro-div}. Here, $\Dt{}$ indicates partial
differentiation with respect to time, $\Grad =\Grad_x$ indicates the
spatial gradient operator, and $\Div$ the divergence operator.  This
problem, proposed initially in~\citet{Facca-et-al:2018}, is a
generalization to the continuous setting of the discrete model
developed by~\citet{tero:model} for the simulation of the dynamics of
\emph{Physarum Polycephalum}, a slime mold with exceptional
optimization abilities~\citep{nakagaki:maze}. This latter model was
analyzed in~\citet{bonifaci:physarum}, where its equivalence to an
optimal transportation problem on a graph was shown.  In analogy to
the discrete setting,~\citet{Facca-et-al:2018} conjecture that, at
infinite times, system~\cref{eq:sys-intro} is equivalent to the
PDE-based formulation of the Monge-Kantorovich (MK) optimal
transportation problem with cost equal to the Euclidean distance as
given in~\citet{EvansGangbo99}. This latter problem reads: find a
positive function $\OptTdens$ in $\Lspace{1}(\Domain)$ and a potential
$\OptPot\in\Lip{1}(\Domain)$, with $\Lip{1}(\Domain)$ the space of
Lipschitz continuous functions with unit constant, such that:
\begin{subequations}
  \label{eq:MK-problem}
  \begin{align}
    &-\Div(\OptTdens(x)
    \Grad\OptPot(x))=\Fsource(x)  \label{eq:MK-elliptic} \\ 
    &|\Grad\OptPot(x)| \leq 1  \qquad
      \qquad \forall\; x \in \Domain  \label{eq:MK-grad-omega} \\
    &|\Grad\OptPot(x)| =1   \qquad \
      \text{ a.e. where } \OptTdens(x)>0   \label{eq:MK-grad-supp}
  \end{align}
\end{subequations}
The function $\OptTdens$, called the Optimal Transport (OT) density,
is uniquely defined by $\Forcing$~\citep{Feldman-McCann:2002}, and we
will use the notation $\OptTdens(\Forcing)$ when the space dependence
is not needed explicitly.  The function $\OptPot$ is called the
transport or Kantorovich potential~\citep{Villani:2009}.  We term this
problem the $\Lspace{1}$-\MKEQS\ (or simply \MKEQS) to distinguish it
from the $\Lspace{2}$-MK problem, characterized by a quadratic distance
cost function, and its famous fluid-dynamic formulation given
by~\citet{Benamou:2000,Benamou:2002}. Intuitively, the \OTD\ describes
``how much'' mass ``flows'' through each point of the domain in the
optimal reallocation of $\Source$ into $\Sink$. Indeed, an equivalent
formulation of the \MKEQS, called Beckmann
Problem~\citep{Santambrogio:2015}, states that the vector field
$\OptVel=-\OptTdens \Grad \OptPot$ solves the following problem:
\begin{equation} 
  \label{eq:beckmann}
  \min_{\Vel\in \Vof{\Lspace{1}(\Domain)}{\Dim}}
  \left\{
    \int_{\Domain}{
      |\Vel|
    }\dx
    \ : \ 
    \Div\Vel = \Forcing
  \right\}
\end{equation}
where the divergence is taken in the sense of distributions.

In~\citet{Facca-et-al:2018} local in time existence and uniqueness of
the solution pair $(\Tdens(t),\Pot(t))$ of~\cref{eq:sys-intro} was
proved under the assumption of $\Forcing\in\Lspace{\infty}(\Domain)$
and $\Tdens_0\in\Holder[](\Domain)$. The main difficulty in obtaining
existence and uniqueness of the solution to~\cref{eq:sys-intro} at
large times is the absence of a uniform upper bound for
$|\Grad\Pot(t)|$ or $|\Tdens(t)\Grad\Pot(t)|$.  However, several
numerical experiments in~\citet{Facca-et-al:2018} support the
conjecture of the convergence of $\Tdens$ in~\cref{eq:sys-intro}
toward the solution of the \MKEQS. Moreover, the numerical
approximation of $\Tdens(t)$ and $\Pot(t)$ is experimentally well
defined with $|\Grad\Pot(t)|$ always fulfilling the constraints of the
\MKEQS\ to be less than or equal to one when $t \Tendsto +\infty$.
This suggests that the dynamic optimal transport problem can be an
effective strategy for the numerical solution of the
$\Lspace{1}$-\MKEQS.  Unlike the $\Lspace{2}$ case, for which the
fluid-dynamic formulation of~\citet{Benamou:2000,Benamou:2002} allows
for efficient numerical solution, discretization of the $\Lspace{1}$
formulation treated in this study is much more complicated.
Algorithms based on nonlinear minimization~\citep{prigozhin} or on the
solution of the highly nonlinear Monge-Ampere equation on the product
space are often used typically coupled to some
regularization~\citep{Delzanno:2010,Delzanno:2011}.
In~\citet{Benamou-Carlier:2015,Bartels:2017kj} the
Beckmann Problem~\cref{eq:beckmann} is solved via augmented Lagrangian
methods and $\Hspace[div]{}$ discretizations of the relevant vector
fields.  Numerical methods based on entropy regularization of the
linear programming problem arising from Kantorovich-relaxation have
been introduced in~\citet{Cuturi:2013,Benamou:2015hr}.  These
techniques require the discretization of the problem in the product
space defined by the transported measures $\Source$ and $\Source$, and
thus scale quadratically with the number of unknowns, although they
may possess good parallelization properties.
Finally, we mention efficient approaches discussed
in~\citet{Li-et-al:2018,Jacobs:2018ub}  
based on the Primal-Dual Hybrid Gradient (PDHG)
algorithm~\citep{Chambolle:2010hk}. 

In this paper we propose the numerical solution of
the \MKEQS\ via the discretization of the dynamic model 
as an efficient and robust approach that does not require the
introduction of additional regularizing parameters.  Standard Galerkin
finite elements and Euler time-stepping can be combined with
efficient numerical linear algebra algorithms to produce effective
solution strategies exploiting also the dynamics of the process.  For
example, in~\citet{Bergamaschi-et-al:2018} during the time-stepping
procedure spectral information are collected and used to devise
efficient preconditioners for the conjugate gradient solver.

This paper is formed by two separate parts both supporting the
conjecture of the equivalence between the dynamic \MKEQS\ and the
$\Lspace{1}$-\MKEQS.  The first part introduces a new \LCF\ $\SLyap$
formed by the sum of an energy functional $\Ene$ and a mass functional
$\Wmass$ given by:
\begin{gather}
  \label{eq:slyap-def}
  \SLyap(\Tdens)
  := \Ene(\Tdens) + \Wmass(\Tdens)
  \\
  \label{eq:slyap-ene-wmass}
  \Ene(\Tdens)
  :=
  \sup_{\varphi \in \Lip{}(\Domain)}
  \left\{
    \int_{\Domain} \left(
      \Forcing \Ftest
      -
      \Tdens
      \frac{
        | \Grad \Ftest|^2
      }{
        2
      }
    \right)
    \dx
  \right\}
  \quad
  \Wmass(\Tdens)
  :=
  \frac{1}{2}
  \int_{\Domain}{
    \Tdens\dx
  }
\end{gather}
The energy functional $\Ene(\Tdens)$ is written above in a
variational form, which,
under the assumption $\Tdens \in \Holder[](\Domain)$ and
$\Forcing\in \Lspace{\infty}(\Domain)$, is equivalent to
\begin{equation*}
  \Ene(\Tdens)=
  \frac{1}{2}
  \int_{\Domain}{
    \Tdens |\Grad \PotOp(\Tdens)|^2 
  }\dx,
\end{equation*}
where $\PotOp(\Tdens)$ identifies the solution of
\cref{eq:sys-intro-div} given $\Tdens$.
Using the same hypothesis adopted in~\citep{Facca-et-al:2018} to show
existence and uniqueness of the solution pair $(\Tdens(t),\Pot(t))$
for small times, we prove that $\SLyap$ is strictly decreasing along
$\Tdens$-trajectories of \cref{eq:sys-intro}. More remarkably, we can
show that the \OTD\ $\OptTdens_{\Forcing}$ is the unique minimizer of
the functional $\SLyap$, and that the minimum equals the Wasserstein
distance between $\Source$ and $\Sink$ with cost equal to Euclidean
distance (denoted with $\Wass{1}$).  This result gives further
evidence in support of the conjecture that $\OptTdens_{\Forcing}$ is the
unique attractor of the dynamics on $\Tdens$ of~\cref{eq:sys-intro}.
Unfortunately, since
we are still not able to provides a uniform bound 
on $|\Grad \Pot(t)|$, global existence results seem far 
to be reached and the conjecture in~\citet{Facca-et-al:2018}
remains open.

The second part of the paper reports an extensive experimental
analysis of the numerical solution of the dynamic \MKEQS, with the
twofold ambition of i) supplying additional support of the equivalence
conjecture introduced in~\citet{Facca-et-al:2018}, and ii)
corroborating the thesis that this dynamic MK model provides an ideal
setting for the numerical solution of the $\Lspace{1}$-MK
equations. To this aim, we derive and test several numerical
approaches for the solution of~\cref{eq:sys-intro}. All the considered
methods couple together simple and cost-effective low order ($\PONE$
or triangular $\PZERO$) Galerkin finite element spaces with Euler
(forward or backward) scheme for the time-discretization of the
ensuing Differential Algebraic (DAE) system of
equations~\citep{QuarteroniValli94}.  Successive (Picard) iterations
are used when necessary to resolve the nonlinearities.  The expected
convergence of the different approaches are tested at large simulation
times against the closed form solution proposed by~\citep{buttazzo}
for sufficiently regular forcing functions.  We also verify the
convergence toward steady-state for increasingly refined grids and the
behavior of the proposed Lyapunov-candidate function.

Next, we extend the comparison already presented
in~\citet{Facca-et-al:2018} of our model results against those
reported in~\citet{prigozhin}.  These tests consider a sequence of
spatially refined grids where we look at monotonicity of the solution
and convergence of the Lyapunov-candidate functions toward a common
value.  The numerical results show that the use of one single grid may
promote the emergence of oscillations in the numerical gradient field.
These oscillations are amplified by the companion ODE solver,
eventually preventing the long-time convergence of the schemes. A
monotone solution is obtained by discretizing the transport potential
on a triangulation that is uniformly refined ($\Triang[\MeshPar/2]$)
with respect to the triangulation $\Triang[\MeshPar]$, where the
gradient field and the transport density are defined, similarly to
what happens with the inf-sup stable mixed finite elements methods for
the solution of Stokes equation~\citep{Boffi:2013}.  In the case of a
mesh aligned with the support of the transport density, optimal grid
convergence for smooth ($\Cont^1$) forcing functions is obtained using
using $\PONE(\Triang[\MeshPar/2])$ to discretize the transport
potential and either the $\PONE(\Triang[\MeshPar])$ or
$\PZERO(\Triang[\MeshPar])$ for the discretization of the transport
density and the gradient of the transport potential.  For piecewise
continuous forcing function, the loss of convergence seems to affect
more the $\PONE-\PONE$ strategy than $\PONE-\PZERO$, which seems to
be more robust.  When the grid is not aligned with the support of the
transport density, additional errors due to geometrical approximations
are introduced and optimal first order convergence convergence is
lost.  Simple grid refinement strategies can be easily employed to
solve this problem, as shown in~\citet{prigozhin}. However,
we work on fixed grids since we are interested in exploring the
convergence properties of the basic methods.  In the case of spatially
heterogeneous domain, the comparison against published numerical
results is qualitatively coherent, but no quantitative result is
obviously possible. However, convergence of the Lyapunov-candidate
functions toward the equilibrium point is verified.

The last section of the paper presents the numerical evaluation of the
$\Lspace{1}$-OT Map described in~\citet{EvansGangbo99} from the the
approximated solution $(\OptTdensH,\OptPotH)$ of the \MKEQS\ obtained
with the \DMK\ approach.  The approximate OT maps are compared with
those obtained with the Sinkhorn algorithm with entropic
regularization by computing barycentric maps as described
in~\citet{Perrot-et-al:2016} and implemented in the package
\verb|Pot|~\citep{flamary2017pot}.  The numerical comparison on a test
case with singular optimal sets shows the accuracy, efficiency and
robustness of the proposed approach.

\section{The \LCF\  $\SLyap$}
In~\citet{Facca-et-al:2018} the authors proposed a $\Lyap(\Tdens)$
given by the product of $\Ene$ and $\Wmass$. Here the product is
replaced by the sum, and we analyze here the behavior of of this new
functional $\SLyap$ along the $\Tdens(t)$-trajectory given by the
solution of~\cref{eq:sys-intro}.  We have the following proposition:
\begin{Prop}
\label{prop-lie-der}
  Given $\bar{t}>0$ such
  that~\cref{eq:sys-intro} admits a solution pair 
  $\left(\Tdens(t),\Pot(t)\right)$ with
  $\Cont[1]$-regularity in time for all $t\in[0,\bar{t}[$,
  then $\SLyap(\Tdens(t))$ is strictly decreasing in time and 
  its time derivative is given by:
  \begin{equation}
    \label{eq:lie-der}
    \frac{d}{dt}
    \SLyap(\Tdens(t))
    =
    -
    \frac{1}{2}
    \int_{\Domain}{
      \Tdens(t)
      \left(
        |\Grad \PotOp( \Tdens(t) )| - 1
      \right)^2
      \left(
        |\Grad \PotOp( \Tdens(t) )| + 1
      \right)
    }\dx.
  \end{equation}
\end{Prop}
\begin{proof}
  That hypothesis of the~\namecref{prop-lie-der} are fulfilled under the
  regularity assumptions $\Tdens_0\in \Holder[](\Domain)$ and
  $\Forcing \in \Lspace{\infty}(\Domain)$ used
  in~\citet{Facca-et-al:2018}. 

  The proof starts by computing the Lie-derivative of $\Ene$:
  \begin{equation*}
    \frac{d \Ene(\Tdens(t))}{dt}
    =
    \frac{1}{2}
    \int_{\Omega}{
      \left(
        \Dt{\Tdens}(t)
        |\Grad \Pot(t)|^2
        +
        2
        \Tdens(t)
        \Grad \Dt{\Pot}(t)
        \Grad \Pot(t)
      \right)
    }\dx
  \end{equation*}
  Differentiating in time the weak form of
  equation~\cref{eq:sys-intro-div}, we obtain that $\Dt{\Pot}(t)$
  solves the following problem:
  \begin{equation*}
%    \label{eq:dt-pot}
    \begin{aligned} 
      \int_{\Omega}{
        \Tdens(t) \Grad \Dt{\Pot}(t)
        \cdot
        \Grad
        \Ftest}\dx
      &=
      -\int_{\Omega}{
        \Dt{\Tdens}(t)
        \Grad \Pot(t)
        \cdot
        \Grad
        \Ftest}
      \dx 
      \quad 
      \forall
      \Ftest\in  \Sob[1]{}(\Omega)
    \end{aligned}
  \end{equation*}
  Substitute $\Ftest = \Pot(t)$  we
  obtain
  % the following equality
  % % 
  % \begin{displaymath}
  %   \int_{\Omega}{
  %     \Tdens(t)
  %     \Grad \Dt{\Pot}(t)
  %     \cdot 
  %     \Grad 
  %     \Pot(t)
  %   }
  %   \dx
  %   =
  %   -\int_{\Omega}{
  %     \Dt{\Tdens}(t)
  %     |\Grad \Pot(t)|^2}
  %   \dx
  % \end{displaymath}
  % % 
  % Thus 
  %
  \begin{equation*}
    \frac{d \Ene(\Tdens(t))}{dt}
    =
    -
    \frac{1}{2}
    \int_{\Omega}{
      \Dt{\Tdens}(t)
      |\Grad \Pot(t)|^2
    }\dx
  \end{equation*}
  from which~\cref{eq:lie-der} follows.

  For any $\Tdens_0>0$ we have that 
  $\Tdens(t)\geq e^{-t}\min_{x\in\Domain}\Tdens_0(x)>0$
  for $t\in[0,\bar{t}[$. Hence, all
  terms contained in~\cref{eq:lie-der} are strictly positive, and thus
  the time derivative is strictly negative.
\end{proof}

Looking at~\cref{eq:lie-der} we note that
$\frac{d}{dt} \SLyap(\Tdens(t))$ is equal to zero only if
$|\Grad\PotOp(\Tdens(t))|=1$ within the support of $\Tdens(t)$, which
is one of the constraints of the \MKEQS.  Thus the \OTD\ $\OptTdens$
becomes a natural candidate for the minimizer of $\SLyap$.
To verify this claim, we need the following duality lemma, whose proof
can by found in~\citet{shape}:
\begin{lemma}
  \label{lemma-energy-duality}
  Consider $\Tdens\in\Lplus{1}(\Domain)$,
  $\Forcing\in \Lspace{1}(\Domain)$ with zero mean, then the following
  equalities hold
  \begin{equation}
    \label{eq:energy-duality}
    \begin{aligned}
      \Ene(\Tdens)
      &=
      \!\sup_{\varphi \in \Lip{}(\Domain)}
      \left\{
        \int_{\Domain} \left(
          \Forcing \Ftest-\Tdens\frac{|\Grad\Ftest|^2}{2}
        \right)\dx
      \right\} \\
      &=\inf_{\Field \in \Vof{\Lspace{2}_{\Tdens}(\Domain)}{\Dim}}
      \left\{
        \int_{\Domain}\frac{|\Field|^2}{2} \Tdens\dx
        \, : 
        -\Div(\Tdens\Field) = \Forcing
      \right\}
    \end{aligned}
  \end{equation}
  where $\Lspace{2}_{\Tdens}(\Domain)$ indicate the space of
  real-valued functions on $\Domain$, square-integrable with respect
  to the measure $\Tdens\dx$.
\end{lemma}

We can now state the following~\namecref{prop-min-lyap}:
\begin{Prop}
  \label{prop-min-lyap}
  Given $\Forcing=\Source-\Sink \in \Lspace{1}(\Domain)$ with zero
  mean, then the \OTD\ $\OptTdens(\Forcing)$ is a minimizer for
  $\SLyap$ with value equal to the $\Wass{1}$-Wasserstein distance
  between $\Source$ and $\Sink$.
\end{Prop}
\begin{proof}
  \label{proof-min-lyap}
  This proof is based on the equivalence
  between the minimization of $\SLyap$ and the \emph{Beckmann Problem}
  in~\cref{eq:beckmann}.  Using~\cref{lemma-energy-duality},
  $\forall \Tdens \in \Lplus{1}(\Domain)$ the following equalities can
  be written:
  \begin{gather*}
    \SLyap(\Tdens)
    =
    \inf_{\Field \in \Vof{\Lspace{2}_{\Tdens}(\Domain)}{\Dim}}
    \left\{
      \InfSupLyap(\Tdens,\Field)
      \; : \;
      -\Div(\Tdens\Field) = \Forcing
    \right\}
    \\
    \InfSupLyap(\Tdens,\Field):=
    \frac{1}{2}
    \int_{\Domain}{
      |\Field|^{2}
      \Tdens
    }\dx
    +
    \frac{1}{2} 
    \int_{\Domain}{
      \Tdens
    }\dx
  \end{gather*}
  For any $\Tdens\in\Lplus{1}(\Domain)$ and for any
  $\Field\in(\Lspace{2}_{\Tdens}(\Domain))^{\Dim}$, a straight forward
  application of Young inequality yields:
  \begin{equation*}
%    \label{eq:young}
    \int_{\Domain}{|\Field|\Tdens}\dx
    \leq
    \frac{1}{2}\int_{\Domain}{|\Field|^2\Tdens}\dx
    +
    \frac{1}{2}\int_{\Domain}{\Tdens}\dx
    =
    \InfSupLyap(\Tdens,\Field)
    \quad 
    \forall \Field \in (\Lspace{2}_{\Tdens}(\Domain))^{\Dim}
  \end{equation*}
  By taking the infimum on
  $\Field\in(\Lspace{2}_{\Tdens}(\Domain))^{\Dim}$ with
  $-\Div(\Tdens\Field)=\Forcing$ in the last inequality we obtain
  \begin{equation*}
    \inf_{\Field \in (\Lspace{2}_{\Tdens}(\Domain))^{\Dim}}
    \left\{\int_{\Domain}{|\Field|\Tdens}\dx
      \; : \;
      -\Div(\Tdens\Field) =\Forcing
    \right\}
    \leq
    \SLyap(\Tdens)
    \quad \forall \Tdens \in \Lplus{1}(\Domain)
  \end{equation*} 
  Since $\OptVel=-\OptTdens \Grad \OptPot$ solves the Beckmann
  problem, we can write:
  \begin{align*}    
    \int_{\Domain}{\Tdens^*}\dx
    =
    &\inf_{\Vel\in\Vof{\Lspace{1}(\Domain)}{\Dim}}
    \left\{\int_{\Domain}{|\Vel|}\dx
      \; : \;
      \Div\Vel=\Forcing
    \right\}
    \\
    \leq
    & \inf_{\Tdens,\Field }
      \left\{
      \int_{\Domain}{|\Field|\Tdens}\dx
      \; : \;
      % \begin{aligned} 
      %   &(\Tdens,\Field) \in 
      %   \Lplus{1}(\Domain)\times\Vof{\Lspace{2}_{\Tdens}(\Domain)}{\Dim}\\
      %   &  MARIO: l'abbiamo gia' detto all'inizio
      -\Div(\Tdens\Field) =\Forcing
      % \end{aligned}
    \right\}
    \leq \SLyap(\Tdens)
  \end{align*}
  which holds for any $\Tdens\in\Lplus{1}(\Domain)$. Since
  \begin{equation*}
    \SLyap(\Tdens^*)
    =
    \int_{\Domain}{\Tdens^*}\dx
  \end{equation*}
  we have that:
  \begin{equation*}
    \SLyap(\Tdens^*)
    \leq
    \inf_{\Tdens  \in \Lplus{1}(\Domain)}\SLyap(\Tdens),
  \end{equation*}
  showing that all the above inequalities are equalities, proving that
  that the \OTP\ is a minimum for $\SLyap$.  If there were another
  minimum $\Tilde{\Tdens}\ne\OptTdens$ for $\SLyap$ we would get a
  contradiction to the result shown in~\citet{Feldman-McCann:2002} on the
  uniqueness of \OTD\ when $\Forcing \in \Lspace{1}(\Domain)$.  Since
  the integral of the \OTD\ is equal to the $\Wass{1}$-distance between
  $\Source$ and $\Sink$~\citep{EvansGangbo99} we obtain also the
  second statement of the Proposition, thus concluding the proof.
\end{proof}

\section{Numerical discretization}
\label{sec:discr}
We start this section by stressing the fact that our aim is to show
the effectiveness of simple discretization methods for the solution
of~\cref{eq:sys-intro}. Obvious improvements in both computational
efficiency and accuracy can be obtained by using more advanced
approaches, such as, e.g., higher order approximations, Newton method,
automatic mesh refinement, etc. However, our starting point is to show
that even the simple methods presented here form an efficient
and robust framework for the solution of the
$\Lspace{1}$-\MKEQS\ using the proposed dynamic setting.

\subsection{Projection spaces}
\label{sec:spatial_discr}
Our numerical approach at the solution of~\cref{eq:sys-intro} is based
on the method of lines. Spatial discretization is achieved by
projecting the weak formulation of the system of equations onto a pair
of finite dimensional spaces $(\Vspace,\Wspace)$.  We denote with
$\Triang[\MeshPar](\Domain)$ a regular triangulation of the (assumed
polygonal) domain $\Domain$, characterized by $n$ nodes and $m$ cells,
where $\MeshPar$ indicates the characteristic length of the
elements. We denote with
$\PZERO(\Triang[\MeshPar](\Domain))=
\Span\{\Wbase[1](x),\ldots,\Wbase[\Wdim](x)\}$ the space of
element-wise constant functions on $\Triang[\MeshPar](\Domain)$, i.e.,
$\Wbase[i](x)$ is the characteristic function of cell $\Cell[i]$.  The
space
$\PONE(\Triang[\MeshPar](\Domain))=
\Span\{\Vbase[1](x),\ldots,\Vbase[\Vdim](x)\}$ is the space of
continuous linear Lagrangian basis functions defined on
$\Triang[\MeshPar](\Domain)$.  We consider two different choices of
the space $\Vspace$ used in the projection of the elliptic
equation~\cref{eq:sys-intro-div}, namely
$\Vspace=\PC{1,\MeshPar}=\PONE(\Triang[\MeshPar](\Domain))$ and
$\Vspace=\PC{1,\MeshPar/2}=\PONE(\Triang[\MeshPar/2](\Domain))$.  Here
$\Triang[\MeshPar/2](\Domain)$ is the triangulation generated by
conformally refining each cell $\Cell[k]\in\Triang[\MeshPar](\Domain)$
(i.e. each element $\Cell[k]$ is divided in $2^{\Dim}$ sub-elements
having as nodes the gravity centers of the $2^{\Dim-1}$-faces
contained of $\Cell[k]$).  Again we consider different choices of
spaces also for the projection of the dynamic
equation~\cref{eq:sys-intro-dyn} by using alternatively
$\Wspace=\PC{1,\MeshPar}$ and
$\Wspace=\PC{0,\MeshPar}=\PZERO(\Triang[\MeshPar](\Domain))$, when the
projection is done on the same mesh used for the elliptic equations,
or $\Wspace=\PC{1,\MeshPar/2}$ $\Wspace=\PC{0,\MeshPar}$, when we use
the sub-grid.

Following this approach and separating the
temporal and spatial variables, the discrete potential $\PotH(t,x)$
and diffusion coefficient $\TdensH(t,x)$ are written as:
\begin{equation*}
  \PotH(t,x)=\sum_{i=1}^{\Vdim} \Pot_{i}(t)\Vbase[i](x) \quad
  \Vbase[i]\in \Vspace 
  \qquad 
  \TdensH (t,x)= \sum_{k=1}^{\Wdim}\Tdens_{k}(t) \Wbase[k](x) \quad
  \Wbase[k]\in \Wspace
\end{equation*}
where $\Vdim$ and $\Wdim$ are the dimensions of $\Vspace$ and
$\Wspace$, respectively. 
The finite element discretization yields the following
problem: for $t\geq0$ find
$(\PotH(t,\cdot),\TdensH(t,\cdot))\in \Vspace\times \Wspace$ such that
\begin{subequations}
  \label{eq:fem}
  \begin{align}
    % &a_{\TdensH(\PotH,\Vbase[j])}
     % =
    &\int _{\Domain}\TdensH\Grad \PotH\cdot\Grad\Vbase[j]\dx
     =(\Fsource,\Vbase[j])= \int_{\Domain}\Fsource\Vbase[j]\dx 
    &\quad j=1,\ldots,\Vdim,
      \label{eq:fem_elliptic}\\
    & \int _{\Domain}\Dt{\TdensH}\Wbase[l]\dx = 
      \int_{\Domain}(|\TdensH\ \Grad \PotH|-\TdensH)\Wbase[l] \dx
    & l=1,\ldots,\Wdim,
      \label{eq:fem_ode} \\
    &\int_{\Domain}\TdensH(0,\cdot) \Wbase[j]\dx 
       = \int_{\Domain}\Tdens_0 \Wbase[l]\dx 
    &l=1,\ldots,\Wdim,
  \end{align}
\end{subequations}
where we add to~\cref{eq:fem_elliptic} the zero-mean constraint
$\int_{\Domain}\PotH\dx = 0$ to enforce well-posedness.  In matrix
form, indicating with $\Vect{\Pot}(t)=\left\{\Pot_i(t)\right\}$,
$i=1,\ldots,\Vdim$, and $\Vect{\Tdens}(t)=\left\{\Tdens_k(t)\right\}$,
$k=1,\ldots,\Wdim$, the vectors that describe the time evolution of
the projected system, we can write the following index-1 nonlinear
system of differential algebraic equations (DAE):
\begin{subequations}
  \label{eq:DAE}
  \begin{align}
    & \Matr{A}[\Vect{\Tdens}(t)]\ \Vect{\Pot}(t)=\Vect{b},
      %\qquad\Vect{a}\cdot\Vect{\Pot}(t)=0
      \\
    \label{eq:DAE_ODE}
    & \Matr{M}\ \Dt{\Vect{\Tdens}}(t)=
      \Matr{B}(\Vect{\Pot}(t))\ \Vect{\Tdens}(t),
      \qquad\Matr{M}\ \Vect{\Tdens}(0)=\Vect{\Tdens_0}.
  \end{align}
\end{subequations}
The $\Vdim\times\Vdim$ stiffness matrix
$\Matr{A}[\Vect{\Tdens}(t)]$ is given by:
\begin{equation*}
  \Matr[ij]{A}[\Vect{\Tdens}(t)]=
  \sum_{k=1}^{\Wdim} \Tdens_{k}(t) 
  \int _{\Domain}\Wbase[k]\Grad\Vbase[i]\cdot\Grad\Vbase[j]\dx.
\end{equation*}
The components of the $\Vdim$-dimensional source vector $\Vect{b}$ are
$\Vect[i]{b}=\int_{\Domain}\Forcing\; \Vbase[i]\dx$.  The
$\Wdim\times\Wdim$ mass matrix $\Matr{M}$ is expressed by:
\begin{equation*}
  \Matr[k,l]{M}=\int_{\Domain}\Wbase[k]\Wbase[l]\dx.
\end{equation*}
The $\Wdim\times\Wdim$ matrix $\Matr{B}$ has the same structure of
$\Matr{M}$ and is defined as
\begin{equation*}
  \Matr[k,l]{B}[\Vect{\Pot}(t)]=
  \int_{\Domain}{
    \left(|\sum_{i=1}^{\Vdim}\Pot_i(t)\Grad\Vbase[i]|-1\right)
    \Wbase[k]\Wbase[l]\dx}
\end{equation*}
and the $\Wdim$-dimensional vector $\Vect{\Tdens}_0$ contains the
projected initial condition
$\Vect[l]{(\Tdens_0)}=\int_{\Domain}\Tdens_0\Wbase[l]\dx$.

\subsection{Time discretization}

In order to solve the DAE~\cref{eq:DAE} %or~\cref{eq:DAE_p0}
we \replaced{define}{operate} a discretization in time using either a
forward or a backward Euler scheme. Denoting with $\Deltat$ the
time-step size so that $t_{\tstepp}=t_{\tstep}+\Deltat$ and
$(\Vect{\Pot}^{\tstep},\Vect{\Tdens}^{\tstep})=
\left(\Vect{\Pot}(t_{\tstep}),\Vect{\Tdens}(t_{\tstep})\right)$, the
approximate solution at time $t_{\tstep}$ can be written as
$\PotH^{\tstep}(x)=\sum_{i}^{\Vdim}\Vect[i]{\Pot}^{\tstep}\Vbase[i](x)$
and
$\TdensH^{\tstep}(x)=\sum_{l=1}^{\Wdim}\Vect[l]{\Tdens}^{\tstep}\Wbase[l](x)$.
%\subsubsection{Case $\TdensH(t,\cdot)\in\PC{1,\MeshPar}$}
%In this case, 
The forward Euler scheme is:
\begin{align*}
  &\Matr{A}[\Vect{\Tdens}^{\tstep}]\ \Vect{\Pot}^{\tstep}=\Vect{b},\qquad
    %\Vect{a}\cdot \Vect{\Pot}^{\tstep}=0
    \nonumber\\%[-0.5em]
%  \label{p1ee}\\[-0.5em]
&\Vect{\Tdens}^{\tstepp}=
    (I+\Deltat \Matr{M}^{-1}\Matr{B}[\Vect{\Pot}^{\tstep}])
    \Vect{\Tdens}^{\tstep},\qquad
    \Vect{\Tdens}^{0}= \Matr{M}^{-1} \Vect{\Tdens_0}\nonumber
\end{align*}
When backward Euler is employed, the time-stepping scheme
becomes:
\begin{align*}
  &\Matr{A}[\Vect{\Tdens}^{\tstepp}]\Vect{\Pot}^{\tstepp}=\Vect{b}
  %,\qquad \Vect{a}\cdot \Vect{\Pot}^{\tstepp}=0
  \nonumber\\%[-0.5em]
%  \label{p1ei}\\[-0.5em]
  &\Matr{M}\Vect{\Tdens}^{\tstepp}=\Matr{M}\Vect{\Tdens}^{\tstep}+
    \Deltat\Matr{B}[\Vect{\Pot}^{\tstepp}]\Vect{\Tdens}^{\tstepp},\qquad
    \Vect{\Tdens}^{0} = \Matr{M}^{-1} \Vect{\Tdens_0}\nonumber
\end{align*}
and the nonlinearity is resolved by means of the following successive
(Picard) iteration, starting from
$\Vect{\Tdens}^{0,\tstepp}=\Vect{\Tdens}^{\tstep}$:
\begin{equation*}
  \mbox{for } \nlit=0,1,2,\ldots
  \quad
  \left\{
    \begin{aligned}
      &\Matr{A}[\Vect{\Tdens}^{\nlit,\tstepp}]
      \Vect{\Pot}^{\nlit,\tstepp}=\Vect{b},\qquad
      %\Vect{a}\cdot \Vect{\Pot}^{\nlit,\tstepp}=0
      \nonumber\\[-0.5em]
      \label{eq:p1picard}\\[-0.5em]
      &\Vect{\Tdens}^{\nlitp,\tstepp}=(\Matr{M}-\Deltat
      \Matr{B}[\Vect{\Pot}^{\nlit,\tstepp}])^{-1}\ 
      \left(\Matr{M} \Vect{\Tdens}^{\tstep}\right)\nonumber
    \end{aligned}
  \right.,
\end{equation*}
iterated until the relative difference is smaller than the prefixed
tolerance $\TolPic$:
\begin{equation*}
%  \label{reldiff}
  \RelDiff(\TdensH^{\nlitp,\tstepp},\TdensH^{\nlit,\tstepp})=
    \frac{\|\TdensH^{\nlitp,\tstepp}-\TdensH^{\nlit,\tstepp}\|_{L^2(\Domain)}} 
         {\|\TdensH^{\nlit,\tstepp}\|_{L^2(\Domain)}}\le\TolPic,
\end{equation*}
or the number of Picard iterations $m$ reaches a prefixed maximum
$\ItPic$.  Note that when we consider $\Wspace=\PC{0,\MeshPar}$, the
matrices $\Matr{M}$ and $\Matr{B}$ are diagonal and thus trivially
invertible.

We consider that time-equilibrium has been reached when the relative
variation in $\TdensH$ ($\Var(\TdensH)$) is smaller than $\TolTime$,
i.e.,
\begin{equation*}
  \Var(\TdensH):=
  \RelDiff(\TdensH^{\tstepp},\TdensH^{\tstep})/
  \Deltat<\TolTime.
\end{equation*}
We indicate with $\OptTime$ the time when equilibrium is numerically
reached and with $\OptTdensH$ the corresponding $\TdensH^{\tstep}$.

\subsection{Solution of the linear system}
\label{sec:pcg-solution}
At each time step or each Picard iteration a linear system involving
the large, sparse, symmetric, and semi-positive matrix $\Matr{A}$ must
be solved (the linear system involving $\Matr{M}$ and $\Matr{B}$ is
diagonal or can be made diagonal with mass lumping).  We use a
Preconditioned Conjugate Gradient (PCG) method iterated until the
relative 2-norm of the residual is smaller than the tolerance $\TolPCG$.

The singularity arising from the pure Neumann boundary conditions is
addressed by maintaining the solution orthogonal to the null space
$\Span\{\Vect{1}\}$ of $\Matr{A}$. This is simply obtained as
suggested in~\citet{Bochev:2005} by starting from an initial solution
that is orthogonal to $\Vect{1}$, and when the rounding error in the
matrix-vector multiplication adds non-zero kernel components to the
current iterate vectors, orthogonalizing with respect to the vector
$\Vect{1}$.  In addition, we employ the strategy developed
in~\citet{Bergamaschi-et-al:2018} to correct for the ``near
singularity'' of the stiffness matrix as time advances. In fact, the
system dynamics drives the transport density $\TdensH$ toward zero in
large portions of the domain $\Domain$, progressively loosing
coercivity of the discrete bilinear form.  However, starting from
$\Tdens_0>0$, for a number of initial time steps we have that
$\TdensH^{\tstep}>0$. Hence, essentially we are solving a sequence of
slightly varying coercive linear systems.  Then, at each system
solution we collect spectral information on the preconditioned matrix
to update the previously calculated incomplete Choleski
($IC(\tau)$) preconditioner and enforce orthogonality with respect to
the ``near null space'' of $\Matr{A}$.  In this situation, direct
solvers are not viable and fail to reach a solution in a reasonable
amount of time.

\section{Numerical experiments}
The numerical schemes described in~\cref{sec:discr} are numerically
tested on three test-cases.  The first test compares the large-time
numerical solution against the closed form solution proposed
by~\citep{buttazzo} for given forcing functions.  We verify the
convergence toward steady-state for increasingly refined grids and
ascertain the order of accuracy of the proposed schemes.  The second
test-case is taken from~\citet{prigozhin} and is used to analyze
experimentally the stability of the proposed spatial discretizations.
In the last test-case we consider the reallocation of mass from a
centrally located support towards four disjoint sets with the aim of
verifying the ability of the proposed dynamic formulation to
approximate singular sets. In this test we also build the OT map from
the approximate transport density using the procedure described
in~\citet{EvansGangbo99}.  The resulting map is compared with the maps
computed by means of the Sinkhorn algorithm with entropic
regularization as described in~\citet{Perrot-et-al:2016}.

\subsection{Test~Case~1: comparison with closed-form
  solutions}
\label{sec:num_exp}
\begin{figure}
  \centerline{
    \includegraphics[width=0.3\textwidth]{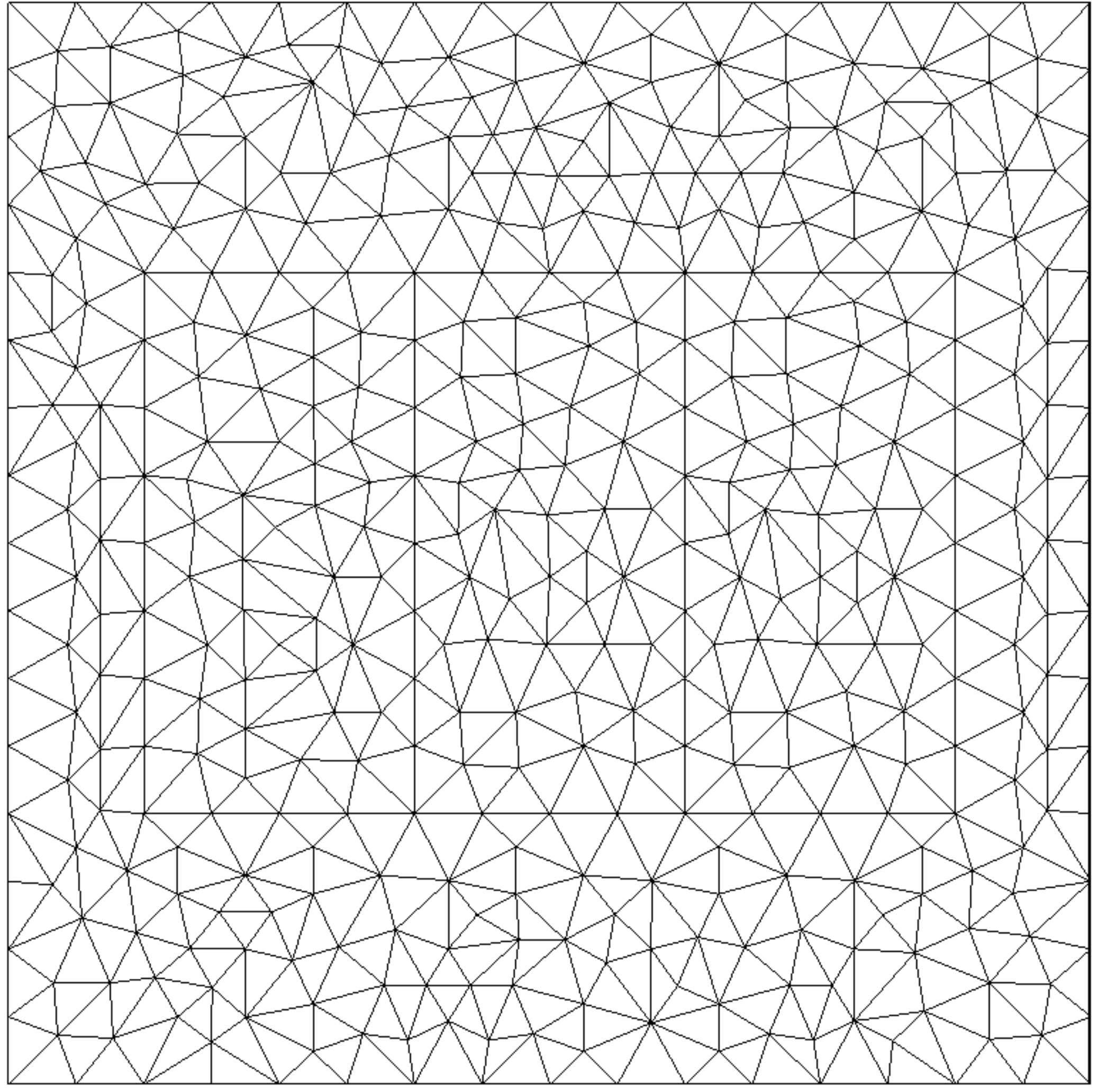}\
    \includegraphics[width=0.3\textwidth]{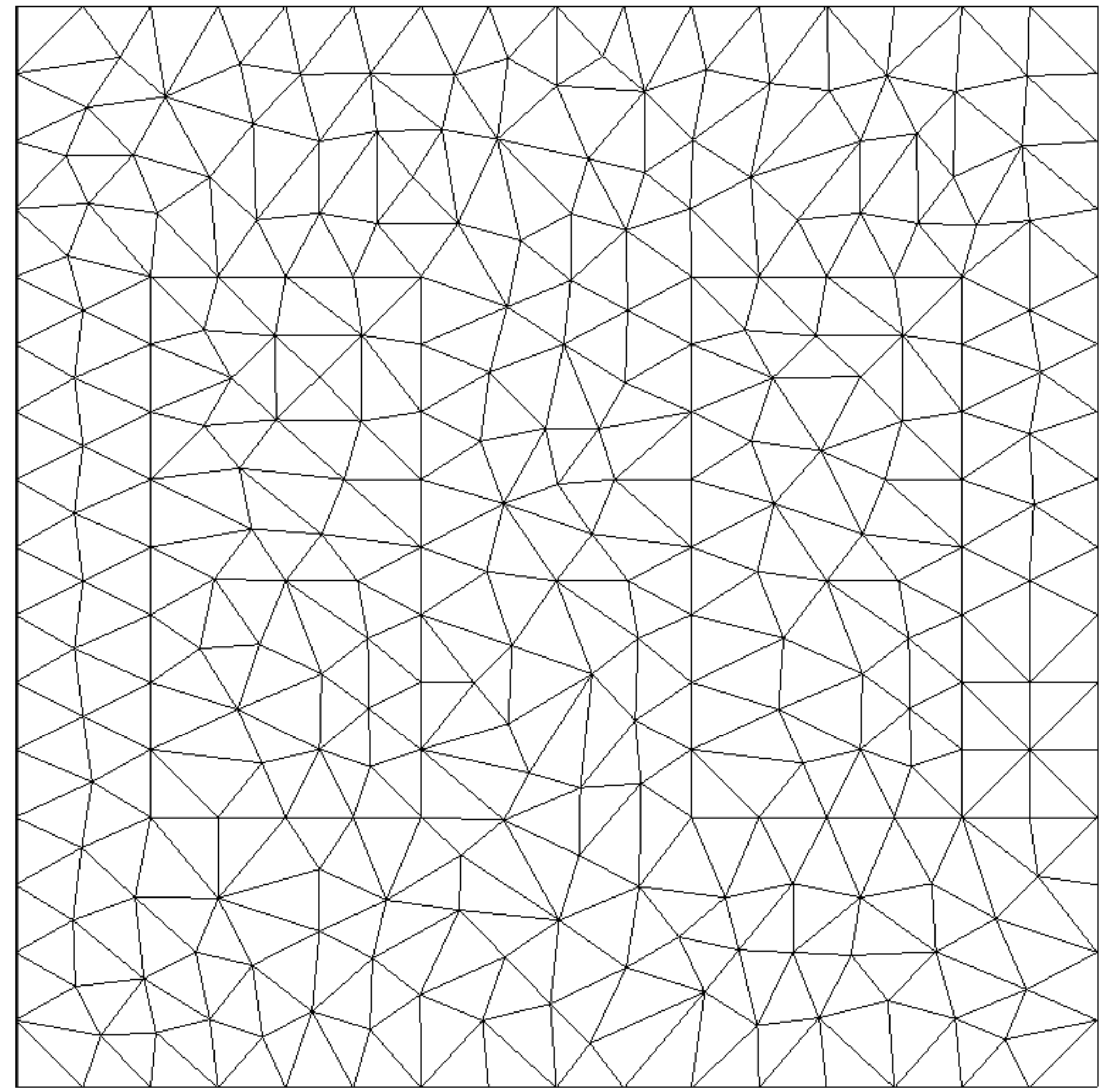}
  }
  \caption
  { Domain $\Domain$ and supports $\SuppPlus$, $\SuppCenter$, and
    $\SuppMinus$ for Test~Case~1, together with the unrefined initial
    meshes. From the left to the right: Mesh~1 (constrained Delaunay,
    438 nodes and 810 elements), and Mesh~2 (constrained Delaunay, 297
    nodes and 528 elements). The edges of Mesh~1 are aligned with the
    supports of $\Fsource$ and $\OptTdens(\Fsource)$. Mesh~2 is
    aligned only with the supports of $\Fsource$.}
  \label{fig:meshes}
\end{figure}

In this first set of tests we consider a square domain in $\REAL^2$,
$\Domain=[0,1]\times[0,1]$, and a zero-mean forcing function
$\Fsource$ supported in two rectangles $\SuppPlus$ and $\SuppMinus$
contained in $\Domain$, where $\Fsource$ assumes opposite signs
(\cref{fig:meshes}). The different supports are given by:
\begin{align*}
  \SuppPlus
  =
  \left[\frac{1}{8},\frac{3}{8} \right] \times 
  \left[\frac{1}{4},\dfrac{3}{4}\right]
  \quad
  \SuppMinus
  =
  \left[\frac{5}{8},\frac{7}{8} \right] \times 
  \left[\frac{1}{4},\dfrac{3}{4}\right]
\end{align*}
To test our numerical schemes we set up two problems that
differ from each other by the specific choice of $\Fsource$.
The first  
test considers a continuous forcing function $\Fcont$ with opposite
sign in $\SuppPlus$ and $\SuppMinus$, while in the second case 
a piecewise constant function $\Fcost$ is used. Their expression is
given by:
\begin{gather*}
%  \label{eq:f_cont}
  \Fcont(x,y)\ ,\ \Fcost(x,y)=
  \left\{
    \begin{aligned}
      &2\sin\left(4\pi x\!-\!\frac{\pi}{2}\right)
      \sin\left(2\pi y \!-\!\frac{\pi}{2}\right)
      \ , \ 
      &&\quad   2
      \quad \mbox{ in }\SuppPlus\\
      -&2\sin
      \left(
        4\pi x\!-\!\frac{5}{2}
      \right)
      \sin
      \left(
        2\pi y\! - \!
        \frac{\pi}{2}
      \right)
      \ , \
      &&-2
      \quad\mbox{ in }\SuppMinus\\
      & 0,0 \quad \mbox{elsewhere}
    \end{aligned}
  \right.
\end{gather*}
From~\citet{buttazzo} we derive explicit formulas for the \OTD\
$\OptTdens(\Fcont)$ and $\OptTdens(\Fcost)$ together with their
support given by $\SuppOT=\SuppPlus\cup\SuppMinus\cup\SuppCenter$
with $\SuppCenter=[3/8,5/8]\times[1/4,3/4]$.  With this explicit
solution, we can verify the experimental convergence rates at
large times for the different proposed schemes.  We use two
different initial triangulation settings (Mesh~1 and Mesh~2,
see \cref{fig:meshes}), each uniformly refined four times to yield
four refinement levels.
Both meshes are constrained to be aligned with the exact
supports of $\Fsource^+$ and $\Fsource^-$, so that the condition
$\sum_{i}\int_{\Omega}\Fsource(x)\Vbase[i]{}\dx=0$ can be imposed
exactly, and, at each level, have approximately the same number of
nodes and elements. 
Mesh~1 (\cref{fig:meshes}, left) is a constrained Delaunay
triangulations with edges aligned with the boundary of
$\SuppOT$. Mesh~2 is also a constrained Delaunay triangulation but is
not aligned with $\SuppOT$ in the area between $\SuppPlus$ and
$\SuppMinus$.  In the latter case, we expect convergence to be
influenced also by the geometric error in approximating the 
boundaries of the support $\SuppOT$ of $\OptTdens$.  Sensitivity
to initial conditions is tested by employing the following different
initial data $\Tdens_0^{(i)}$:
\begin{gather}
  \Tdens^{(1)}_0\equiv 1;
  \Tdens^{(2)}_0(x,y) = 0.1+4|x-0.5,y-0.5\|^2;\nonumber\\[-0.75em]
  \label{eq:initial-conditions}\\[-0.75em]
  \Tdens^{(3)}_0(x,y)=3+2\sin(8\pi x)\sin(8\pi y).\nonumber
\end{gather}
Note that in these tests we do not focus on computational speed, but
only on the numerical behavior of the schemes.  Thus we do not limit
the minimum time step size and the maximum number of iterations (in
both time-stepping and the PCG algorithm used to solve the linear
system of algebraic equations), and use tight tolerances to determine
when time equilibrium is reached and termination of linear and
nonlinear iteration: $\TolPic=10^{-11}$ and $\TolTime=5\times10^{-9}$,
$\TolPCG=10^{-13}$.  In the simulations presented here we adopt both
for the forward and backward Euler time-stepping and vary the time
step size by setting
$\Deltat[\tstepp]=\min(1.05\times\Deltat[\tstep],
\Deltat[\mbox{{\scriptsize max}}])$, where
$\Deltat[\mbox{{\scriptsize max}}]=0.5$.  Preliminary experiments are
used to calibrate this strategy so that it ensures the stability of
the forward Euler scheme, or equivalently, the convergence of the
Picard iteration.  Convergence as $\MeshPar\Tendsto0$ is explored by
looking also at the time behavior of the $\Lspace{2}(\Domain)$ relative
$\Tdens$-error defined as:
\begin{equation*}
  \ErrTdens(\Tdens):=
  \|\Tdens-\OptTdens(\Fsource)\|_{\Lspace{2}(\Domain)}
  \ / \
  \|\OptTdens(\Fsource)\|_{\Lspace{2}(\Domain)}.
\end{equation*}

\paragraph{Convergence toward steady-state equilibrium}

\begin{figure}
  \centerline{
    \includegraphics[width=\textwidth]{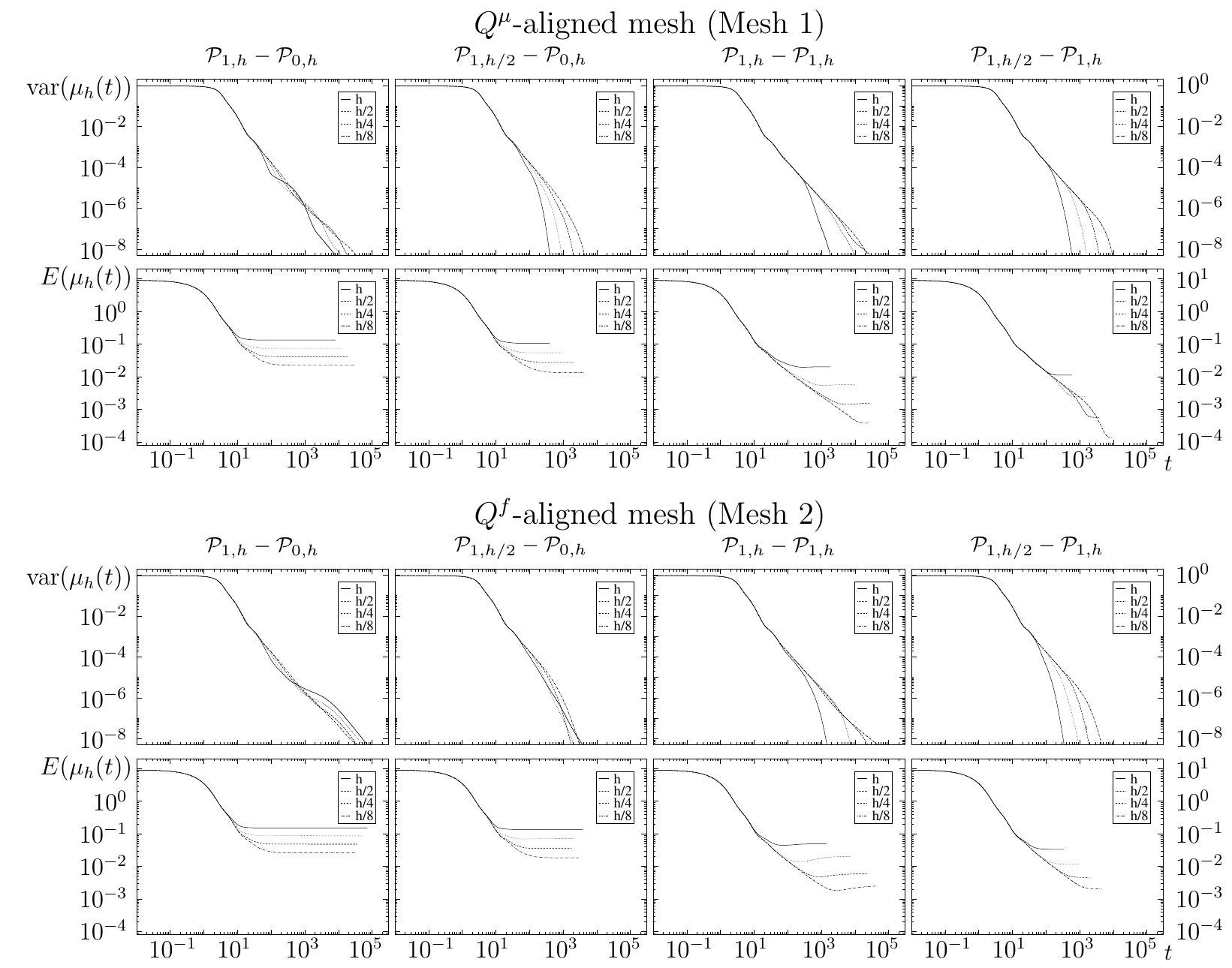}
  }
  \caption { Convergence toward equilibrium in the case of continuous
    forcing ($\Fsource=\Fcont$). The log-log plots of $\Var(\TdensH(t,\cdot))$
    and $\ErrTdens(\TdensH(t,\cdot))$ vs. time are reported for Mesh~1
    ($\SuppOT$-aligned, top block) and for Mesh~2 ($\SuppF$-aligned,
    bottom block).  The columns refer from left to right to the
    results obtained with $\PC{1,\MeshPar}-\PC{0,\MeshPar}$,
    $\PC{1,\MeshPar/2}-\PC{0,\MeshPar}$,
    $\PC{1,\MeshPar}-\PC{1,\MeshPar}$,
    $\PC{1,\MeshPar/2}-\PC{1,\MeshPar}$, respectively.  }
  \label{fig:fcost_p1p0}
\end{figure}

\Cref{fig:fcost_p1p0} reports the log-log scale plots of
$\Var(\TdensH(t))$ and $\ErrTdens(\TdensH(t))$ vs. time, calculated for the
two mesh families in the case of continuous forcing function $\Fcont$.
Each curve in each sub-plot corresponds to a different mesh level.
The columns are related to different
combinations of spatial discretizations. Only results of the Explicit
Euler time-stepping scheme are shown, the results of the Implicit
Euler method being identical. The first set of plots (first two rows)
are relative to the $\SuppOT$-aligned mesh set, while the lower set
reports the results for the $\SuppF$-aligned meshes.

The results show a steady convergence toward the equilibrium point
$\OptTdensH$. The $\TdensH$ variation, $\Var(\TdensH(t))$, displays a
monotone behavior for all schemes, with an expected geometric
convergence rate toward steady-state, as evidenced by the slope of the
rectilinear portions of the curves that coincides for all mesh levels
and types.  At increasing refinement levels the convergence curves
have a common initial behavior for all schemes but start to diverge
approximately when the corresponding spatial accuracy limit is
attained. Accuracy saturation in the error plots ($\ErrTdens(\TdensH(t))$
vs. $t$) occurs at the same time at which $\Var(\TdensH(t))$ start
diverging.  More uncertain profiles are obtained when spatial
discretization is performed on the same mesh for the pair
$(\TdensH,\PotH)$ for both $\PONE-\PZERO$ and $\PONE-\PONE$
discretization spaces. The reason for the loss of regularity is to be
attributed to oscillations in the cell gradients that cause amplified
oscillations in the corresponding transport density. Spatial averaging
of the gradient magnitudes, leading to the
$\Triang[\MeshPar]-\Triang[\MeshPar/2]$ formulation, shows a much
smoother behavior with a faster convergence towards equilibrium. We
postpone a more detailed discussion of this phenomenon to
\cref{sec:confr_prigozhin}, where a more challenging test case is
approached.

Looking at the bottom half of \cref{fig:fcost_p1p0}, we see the effect
of using meshes that are not aligned with the support of the optimal
transport density. Because of the discontinuity in $\TdensH$ occurring
across the boundary of $\SuppOT$, convergence is limited by the
geometric convergence of the triangular shapes towards this boundary,
and the global attainable accuracy is bounded by this error.  We
observe a consistent behavior of the error for both mesh-types at
different $\MeshPar$ levels. The accuracy levels at which the error
saturates decrease consistently with the expected order of spatial
convergence of the different schemes, when the geometric error is
negligible. This is clearly observable by looking at the plots of
$\Var(\TdensH(t))$ for the $\PC{1,\MeshPar/2}-\PC{0,\MeshPar}$, and
the $\PC{1,\MeshPar/2}-\PC{1,\MeshPar}$, cases, where the optimal
second order convergence of the latter approach is observable from the
fact that difference in the attained accuracy levels are doubled with
respect to the first order approach. Higher order methods display
higher accuracy, but the geometric error prevents the realization of
optimal convergence rates.

\paragraph{Convergence of the spatial discretization}

\begin{figure}
  \includegraphics[width=\textwidth]{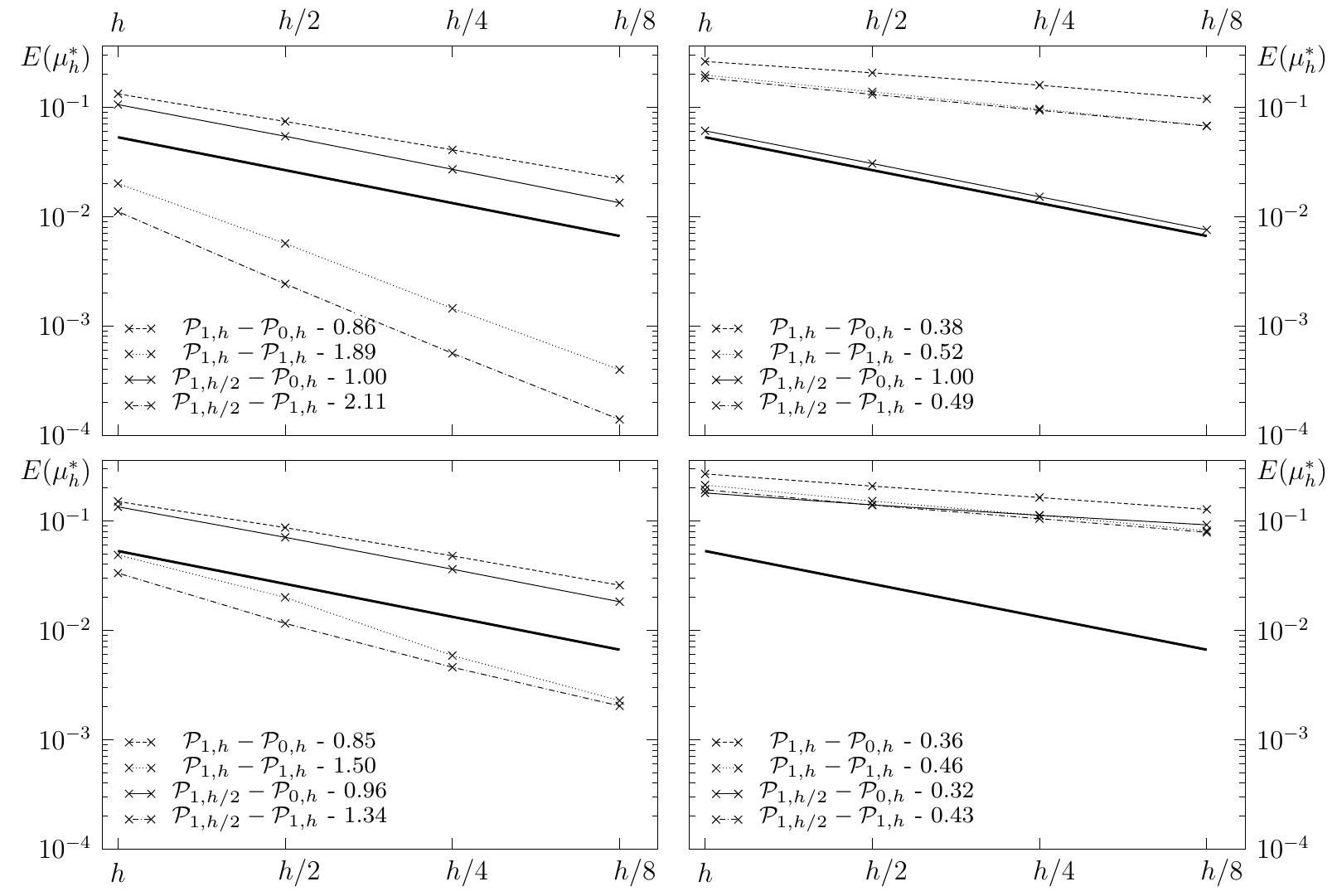}
  \caption{Behavior of $\ErrTdens(\OptTdensH)$ vs. $\MeshPar$ for the
    different discretization methods. The results for the continuous
    forcing function $\Fcont$ are shown in the left column, while the
    right column reports the results for $\Fcost$. The top row is
    relative to the Mesh-1 sequence (aligned with $\SuppOT$), while
    the bottom row corresponds to the Mesh-2 sequence (aligned only
    with $\SuppF$). For visual reference, the first order convergence
    line is also plotted with a thick solid trait. The average
    experimental convergence rates are reported in the legends of each
    plot next to the discretization method.}
  \label{fig:convergence}
\end{figure}

\begin{figure}
  \centerline
  {
    {\includegraphics[width=0.4\textwidth]{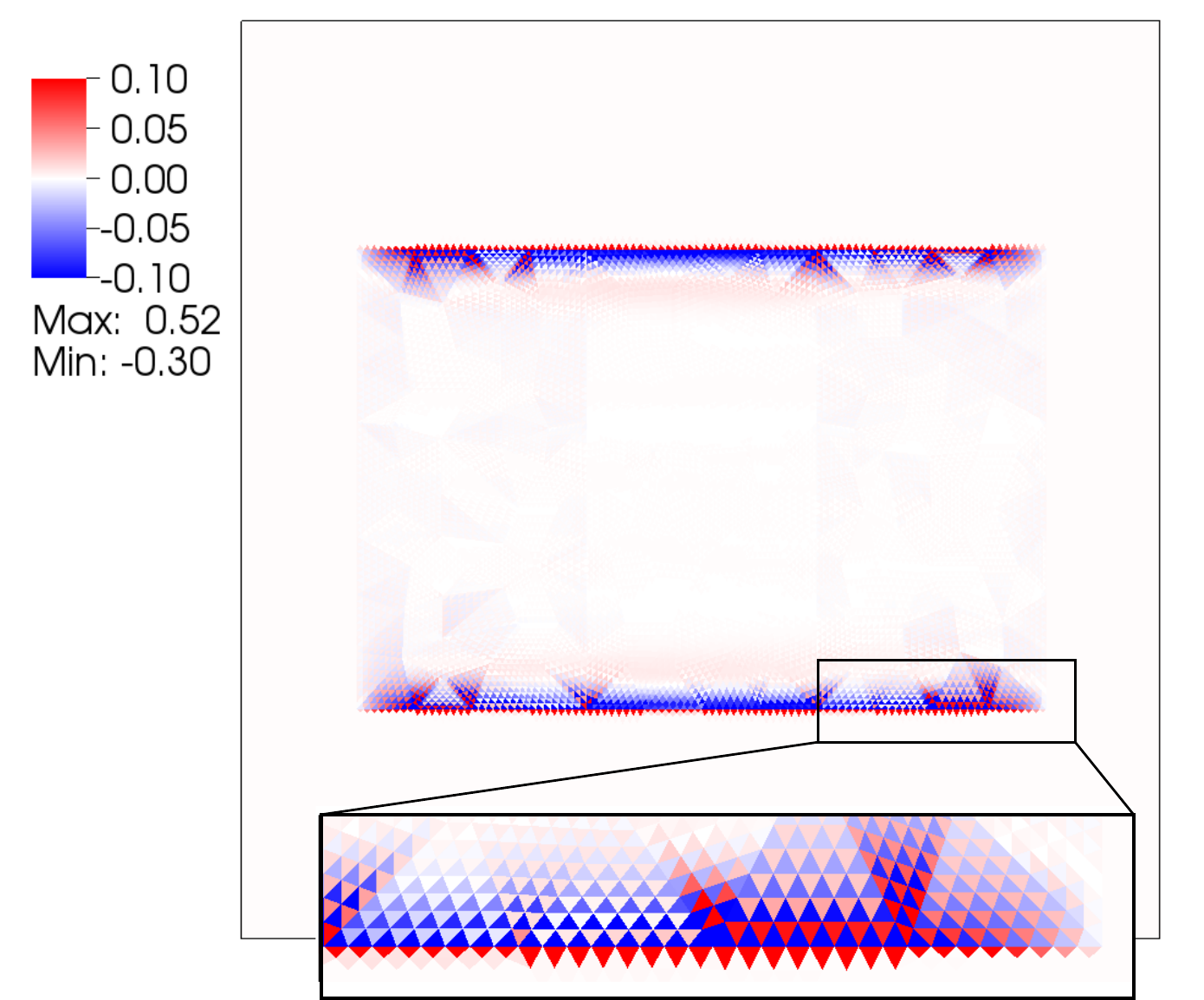}}
    \qquad
    {\includegraphics[width=0.4\textwidth]{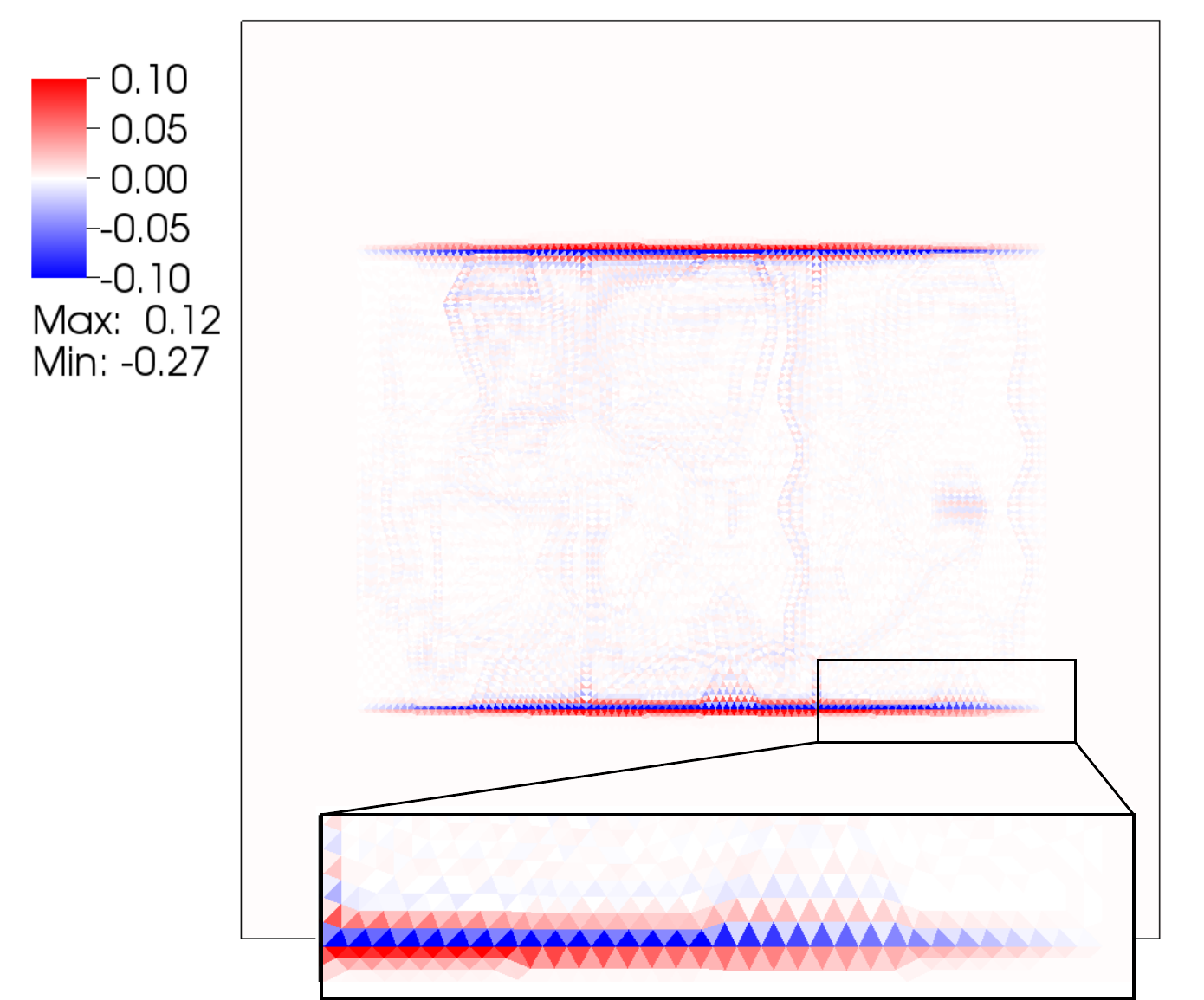}}
  }
  \centerline
  {
    {\includegraphics[width=0.4\textwidth]{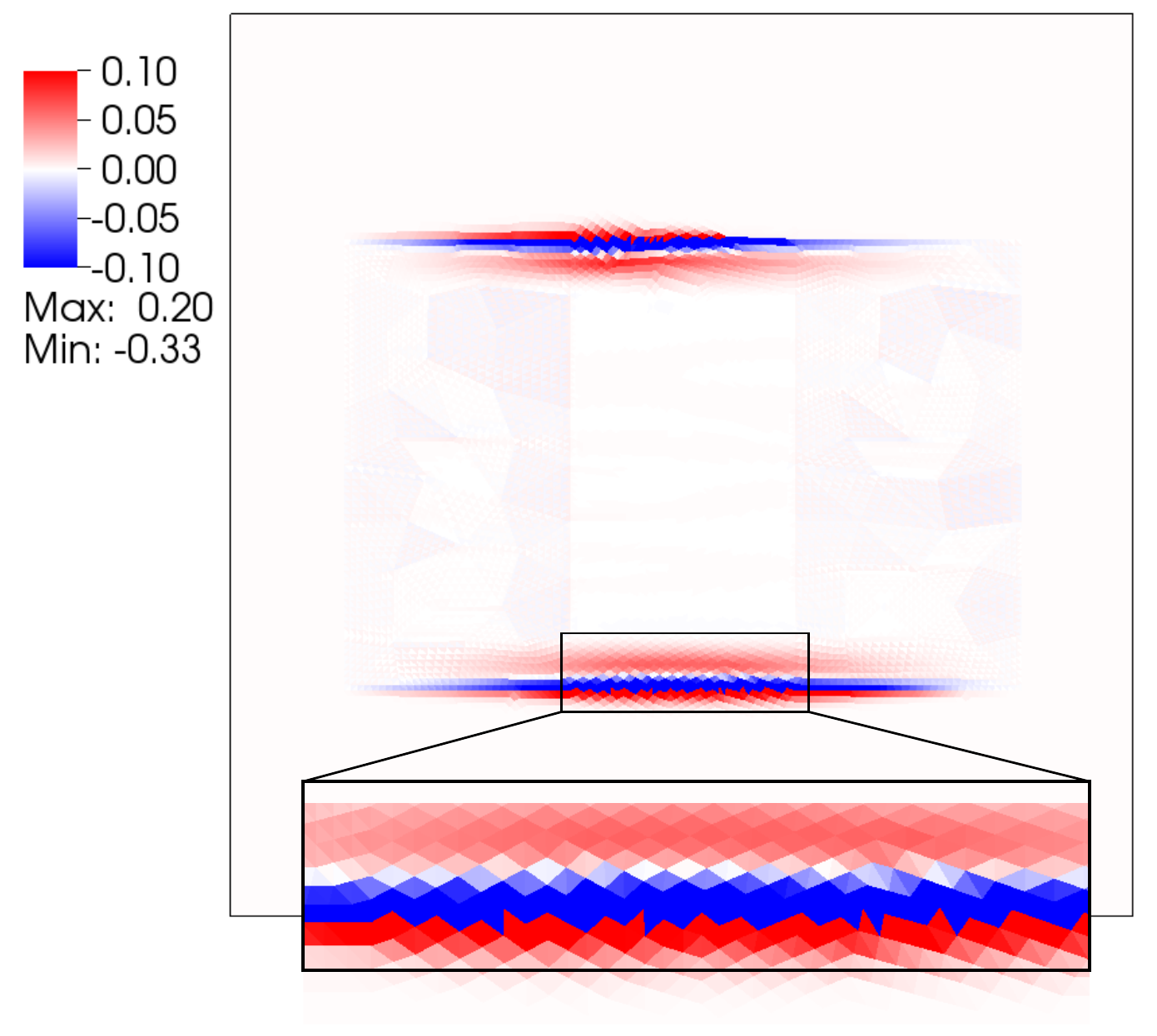}}
    \qquad
    {\includegraphics[width=0.4\textwidth]{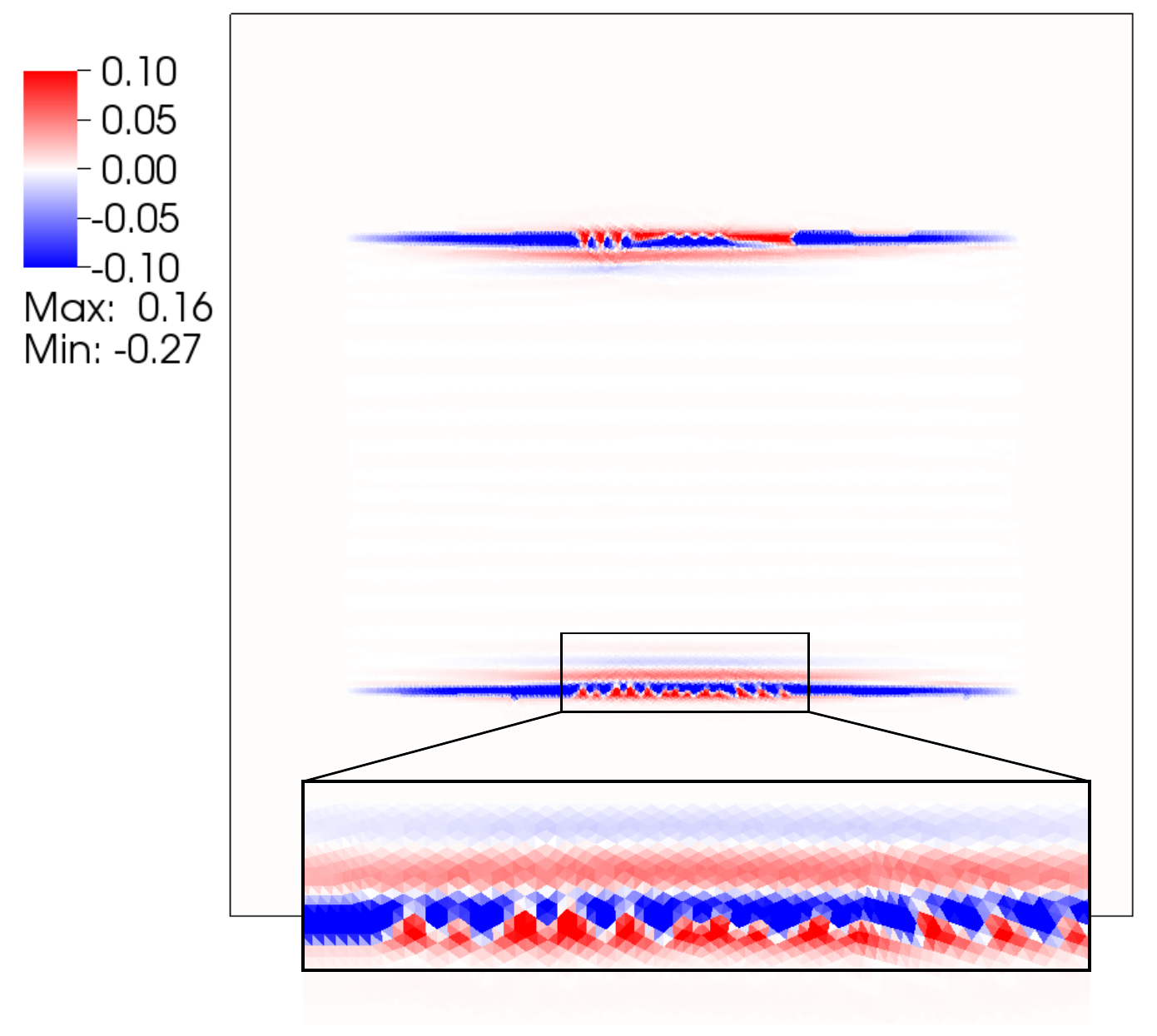}}
  }
  \caption{Spatial distribution of the error
    $\OptTdensH-\OptTdens(\Fcost)$ at steady state for the piecewise
    constant forcing function $\Fcost$.  The upper row reports the
    results on the finest level of Mesh-1 obtained with the
    $\PC{1,\MeshPar}-\PC{0,\MeshPar}$, (left) and
    $\PC{1,\MeshPar}-\PC{1,\MeshPar}$, (right).  The lower row shows
    the results on the finest level of Mesh-2 from the
    $\PC{1,\MeshPar/2}-\PC{0,\MeshPar}$, (left) and
    $\PC{1,\MeshPar/2}-\PC{1,\MeshPar}$, (right) approaches.  }
  \label{fig:err_f_cost}
\end{figure}

We would like to recall that continuity of the transport density
$\OptTdens(\Fsource)$ when the forcing term $\Fsource$ is continuous
was proved in $\REAL^2$ in~\citet{fragala} under some assumptions on
$\Fsource$. However, except for partial regularity results along
transport rays~\citep{buttazzo}, the general case seems to be an open
question.  In our test cases, for both $\Fcont$ and $\Fcost$ forcings,
strong variations in $\TdensH$ are present in a direction orthogonal
to the boundary of $\SuppOT$ in the central portion of the domain
(outside $\SuppF$). Because of these variations, which in the
discontinuous forcing case are actual $\OptTdens$-discontinuities, we
expect a loss of convergence in the FE solution. We note, however,
that convergence towards steady-state is not influenced by spatial
errors, as shown in the previous discussion.

The experimental convergence profiles for the different methods are
reported in \cref{fig:convergence}. The column on the left groups the
results relative to the more regular case of continuous forcing
function $\Fcont$. The right column reports the results obtained for
piecewise constant forcing $\Fcost$. The top and bottom rows identify
the mesh sequences aligned with the boundary of $\SuppOT$ or with the
boundary of $\SuppF$, respectively.

From the two plots on the left, we can argue that: i) all methods
attain optimal convergence when the Mesh-1 sequence is used; ii) the
$\Triang[\MeshPar]-\Triang[\MeshPar/2]$ combination is characterized
by a smoother behavior; iii) the use of the Mesh-2 sequence, which we
recall is aligned only with the boundaries of $\SuppF$ and not those
of $\SuppOT$, triggers the emergence of geometrical errors that cause
a sizeable reduction on the convergence rates of both $\PONE-\PONE$
and $\PONE-\PZERO$ schemes.

As expected, the results for the discontinuous forcing function (right
column) are characterized by an important loss of convergence rate for
all schemes, except the $\PC{1,\MeshPar/2}-\PC{0,\MeshPar}$ in
combination with the $\SuppOT$-aligned meshes. The use of a
$\Triang[\MeshPar]-\Triang[\MeshPar/2]$ combination seems to be more
robust.  This is confirmed by the spatial distribution of the error
$\OptTdensH-\OptTdens(\Fcost)$ shown in \cref{fig:err_f_cost}. In this
figure we report the results obtained with the
$\PC{1,\MeshPar}-\PC{0,\MeshPar}$ (upper left panel) and the
$\PC{1,\MeshPar}-\PC{1,\MeshPar}$ (upper right panel) for Mesh~1, and
the $\PC{1,\MeshPar/2}-\PC{0,\MeshPar}$ (lower left panel) the
$\PC{1,\MeshPar/2}-\PC{1,\MeshPar}$ (lower right panel) for Mesh~2.
The plots suggest that the $\PONE-\PONE$ approach localizes the error
on the north and south boundaries of $\SuppOT$, where the jump in
$\OptTdens$ is concentrated. The $\PONE-\PZERO$ approach, on the other
hand, displays an additional small but non negligible error on the
support of the forcing function $\SuppF$.  The zooms on the pictures
show clear oscillations for the one-mesh methods (upper row) in both
directions orthogonal and parallel to the $\TdensH$-discontinuity.  On
the contrary, the methods based on two-meshes (lower row) exhibit a
monotone error behavior along the boundary of $\SuppOT$, but the
mis-alignment of the triangle edges causes an increased error as
compared to the Mesh-1 results. The error slightly oscillates in the
direction normal to the $\TdensH$-jump due to the gradient
reconstruction. It is evident that the smoothing due to the averaging
of the gradient magnitude on the larger triangles helps in reducing
overall oscillations. This will become more evident when we will
discuss \deleted{ these oscillations in connection with a more
  challenging test case } in \cref{sec:confr_prigozhin}.

\paragraph{Implicit Euler and convergence of the Picard scheme}

In the case of implicit Euler time-stepping, the nonlinear system is
solved by Picard iteration as described in~\cref{eq:p1picard}.
Unfortunately, the lack of a uniform bound on
$|\Grad \Pot(t)|\, \forall t \geq 0 $ prevents the theoretical
derivation of an estimate of the contraction factor.  Experimentally,
all numerical experiments displayed a number of iterations of the
Picard scheme increasing linearly with the time step size $\Deltat$,
suggesting a fixed rate of contraction. This was evaluated by
computing the relative $\TdensH$-variation:
\begin{equation*}%  \label{eq:cost_pic}
  C(\tstep):=
  \frac{
    \|\TdensH^{\nlit^*,\tstep}-\TdensH^{\nlit^*-1,\tstep}\|_{L^2(\Domain)}
  }
  {
    \|\TdensH^{\nlit^*-1,\tstep}-\TdensH^{\nlit^*-2,\tstep}\|_{L^2(\Domain)}
  }
\end{equation*}
where $\nlit^*$ is the Picard iteration number at convergence.
Independently of the spatial discretization method, preliminary
numerical experiments, not reported here, showed that
$C(\tstep)\approx \Deltat$, suggesting that $\Deltat$ can be used as a
proxy to control the time-step evolution in this case.  Values
$\Deltat[\tstep] \ge 1$ caused non-convergence of the Picard iteration,
thus we impose an upper limit of
$\Deltat[\mbox{{\scriptsize max}}]=0.5$. This choice offered a good
trade-off between minimizing the number of Picard iterations and
maximizing the time-step size. At the same time, convergence of the
Picard scheme was achieved with an acceptably small number of Picard
iterations, averaging between 2 and 8 depending on the simulation.
Because of the exponential decay of the solution in time,
as predicted by the mild solution of~\cref{eq:sys-intro-dyn}, 
the time step size was incremented at every step by a factor 1.05.

\paragraph{Dynamics of $\SLyap\left(\Tdens(t)\right)$,
  $\Wass{1}$-distance and computational cost}

\begin{figure}
  \begin{minipage}{0.48\textwidth}
    \raggedleft\includegraphics[width=0.95\textwidth]
    {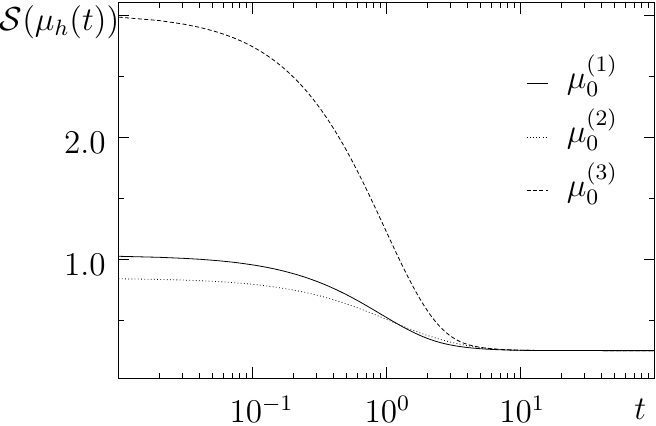}
    \\
    \raggedleft\includegraphics[width=1.01\textwidth]
    {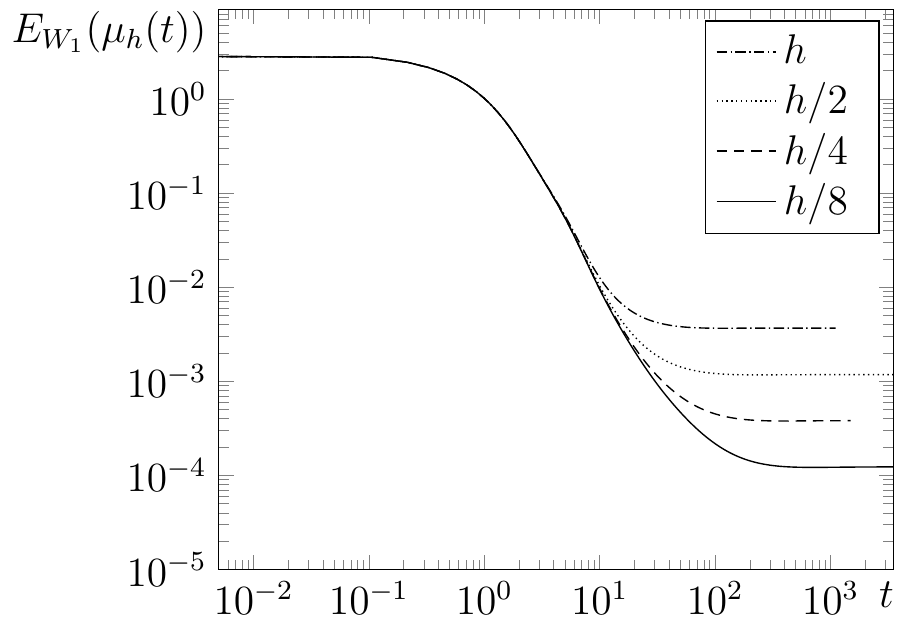}
  \end{minipage}
  \begin{minipage}{0.48\textwidth}
    \centerline{
      \includegraphics[width=0.9\textwidth]
    {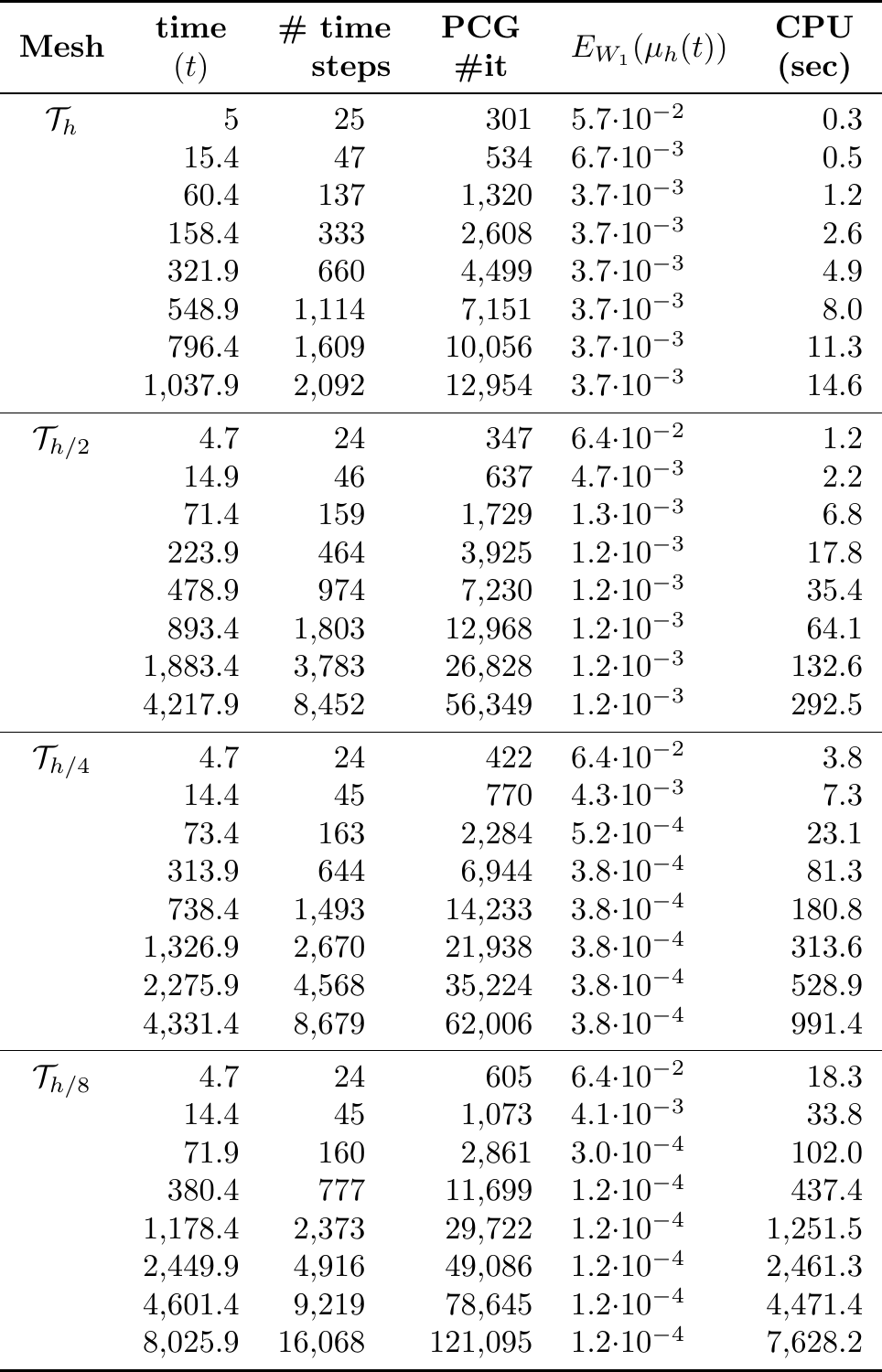}
  }
  \end{minipage} 
  \caption{Test~Case~2: numerical statistics for
    $\PC{1,\MeshPar/2}-\PC{0,\MeshPar}$ on Mesh 2 ($\SuppF$ aligned).
    Top left panel: time behavior of the \LCF\ $\SLyap(\TdensH(t))$
    for three the initial data $\Tdens_0$
    in~\cref{eq:initial-conditions}.  Bottom left panel: time behavior
    of $\ErrWass(\TdensH(t))$ for four refinement levels starting with
    $\Tdens_0=1$.  Right panel: table with simulation statistics,
    including \# of time steps, time $(t)$, number of PCG iterations,
    $\ErrWass(t)$, and CPU time (in seconds). Corresondingly,
    $\Var(\TdensH(t))$ varies in the range $10^{-1}$ and $10^{-8}$.  }
  \label{fig:s_lyap}
\end{figure}

In this paragraph we report numerical evidence of the statements
in~\cref{prop-lie-der,prop-min-lyap}. We also include a discussion on
computational cost to show the effectiveness of the proposed approach,
although the employed numerical techniques are not optimized.  In
fact, a number of cost-saving strategies can be envisaged, including
using coarse-mesh solutions to extrapolate initial guesses of
$\Tdens_0$, re-use of stiffness and preconditioning matrices, use of a
Newton-Raphson strategy to improve stability and allow for larger
time-step sizes together with an inexact Krylov linear solver, etc.
On the other hand, in this work we are interested in showing that,
although far from optimal, our approach is potentially very effective
and competitive with literature approaches in the solution of the
Monge-Kantorovich equations.

\Cref{fig:s_lyap} (top-left panel) reports the time behavior of
$\SLyap(\TdensH(t))$ for the different initial conditions described
in~\cref{eq:initial-conditions} using the finest mesh of set 2
of the $\PC{1,\MeshPar/2}-\PC{0,\MeshPar}$ method.
The results for other methods and mesh sets are practically
indistinguishable, and are not reported here.  We see that $\SLyap$
decreases monotonically and always attains the same minimum value in
time independently of the initial conditions. After $t\approx100$ the
value of $\SLyap(\TdensH(t))$ becomes approximately stationary, up to
machine precision.

The differences among mesh levels emerge after scaling the
value of $\SLyap(\TdensH(t))$ with its assumed asymptotic
value. According to \cref{prop-min-lyap} this value is the minimum of
the $\SLyap$ and is equal to the $\Wass{1}$-distance between $\Source$
and $\Sink$. For Test~Case~2 this value is given by $0.125$, (equal to
the integral of the \OTD), allowing us to compute the relative error as
\begin{equation*}
  \ErrWass(\Tdens):=\frac{\SLyap(\Tdens)-0.125}{0.125}.
\end{equation*}
The bottom left panel in~\cref{fig:s_lyap} reports the time-evolution
of $\ErrWass(\TdensH(t))$ for four mesh refinements. Similarly to the
behavior of $\Var(\TdensH(t)$ in~\cref{fig:fcost_p1p0},
$\ErrWass(\TdensH(t))$ shares the same profile for all mesh levels
until $t\approx 10$. At this time the graphs start separating and
converge to their corresponding asymptotic values that scale
approximately linearly with $\MeshPar$.
This implies that the
stop-tolerance $\TolTime$ used to identify steady state can be relaxed
depending on the sought accuracy.  In fact, two distinct phases can be
identified. The first initial phase displays profiles of
$\Var(\TdensH^{\tstep})$, $\ErrTdens(\TdensH^{\tstep})$, and
$\ErrWass(\TdensH(t))$ that are superimposed and independent of the
mesh level.  This phase is characterized by strong variations of
$\TdensH$, and consequently, by higher number of PCG iterations.
After this initial phase, $\TdensH$ varies more slowly and stabilizes
within $\SuppOT$ to its final value which depends upon the actual mesh
size. At the same time, in $\Omega\setminus\SuppOT$, the decay
continues towards zero.  This phase is characterized by larger
time-step sizes and faster PCG convergence, but much slower
convergence of $\TdensH$ to its asymptotic value, so that only
marginal accuracy gains require large computational efforts.  The use
of increasingly refined meshes should be able to exploit the iterative
process of the \DMK\ approach with consistent reduction of the
computational cost.

These results suggest that the proposed approach can be very efficient
in evaluating $\Wass{1}$ distances.  This statement is corroborated by
the computational statistics collected in the table shown in the right
panel of~\cref{fig:s_lyap}.  For the four mesh levels, we show
simulation time ($t$), cumulated number of time steps ($\# it$),
number of linear (PCG) iterations, the value of
$\ErrWass(\TdensH(t))$, and CPU time in seconds.  The data reported
are collected during the simulation at each change in order of
magnitude of $\Var(\TdensH(t))$ in the range $10^{-1}$ to
$10^{-8}$. The runs are conducted on a 3.4GHz Intel-I7 (1-core)
computer.  The table shows that $\MeshPar$ determines the practical
bound of achievable accuracy in the evaluation of the $\Wass{1}$
distance. For example, looking at the results for
$\Triang{\MeshPar/4}$, it evidently useless go beyond $t=405$, at
which time the accuracy in the $\Wass{1}$ distance is already
$4.6\cdot 10^{-5}$, not far from the highest achievable error accuracy
of $4\cdot 10^{-5}$. Note that with $\Triang{\MeshPar}$ at $t=60.4$
which achieve an accuracy of $3.7\cdot 10^{-3}$ with a mere 1.2
seconds of CPU time. Most of this error is probably due to the
geometric error of having a mesh not aligned with the support of the
\OTD.

\subsection{Test~Case~2: comparison with literature 
  and stability of the spatial discretization}
\label{sec:confr_prigozhin}

\begin{figure}
  \centerline
  {
    \includegraphics[width=0.45\textwidth]{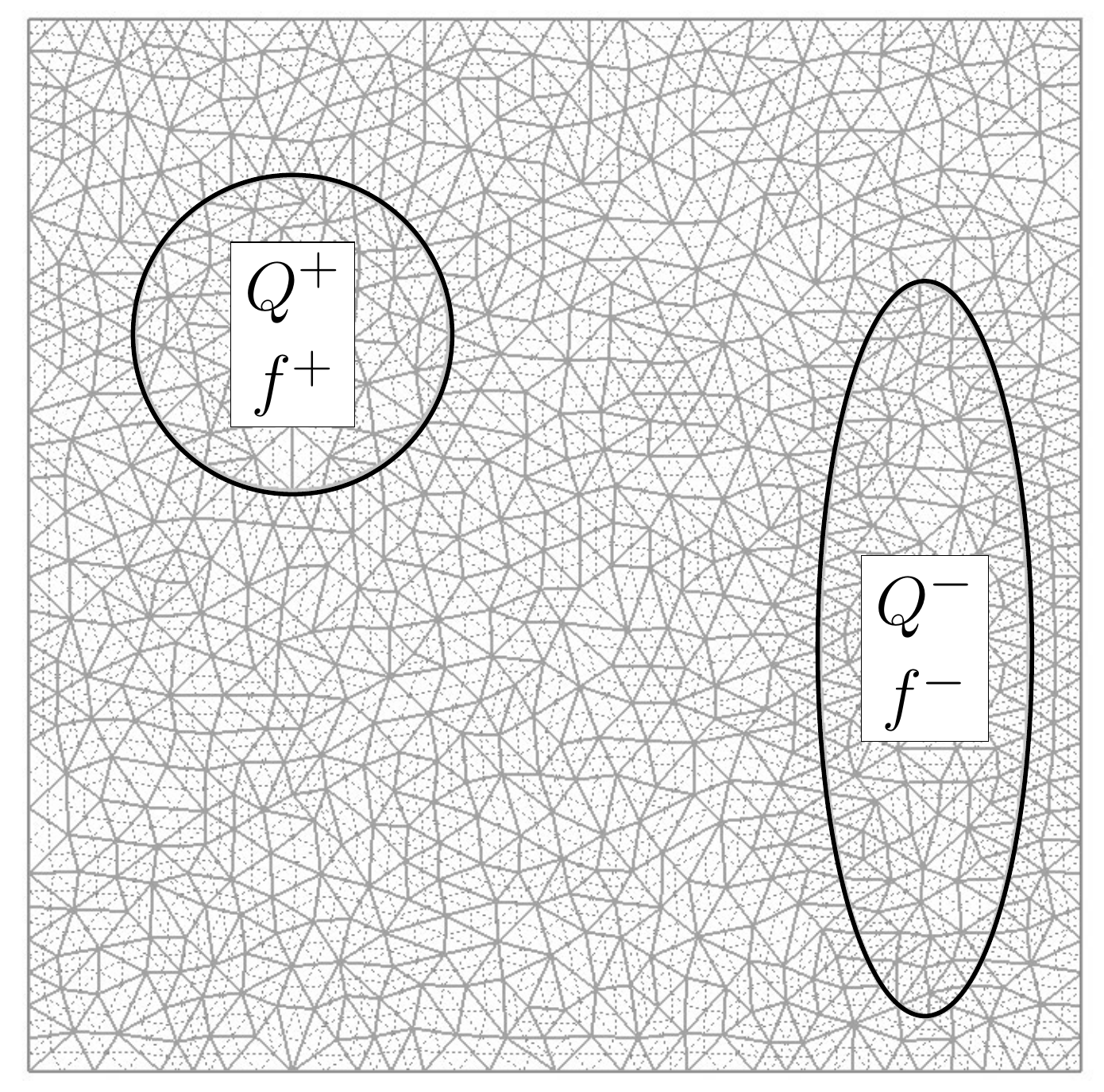}
    \includegraphics[width=0.53\textwidth]{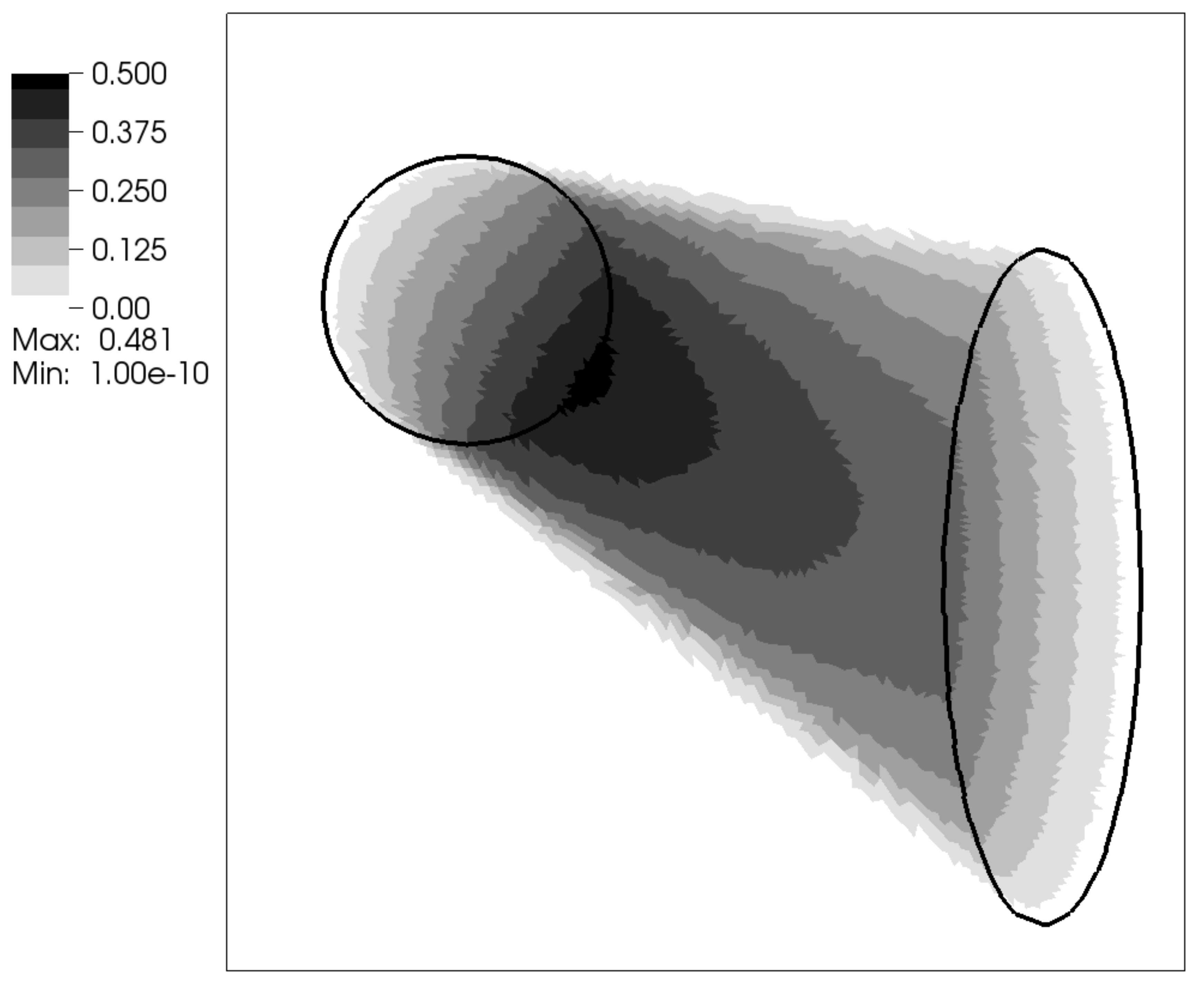}
  }
  \caption{Domain and supports of the forcing function used in the
    discretization of the \MKEQS~\cref{eq:MK-problem} for the
    solution of Example 1 of~\citet{prigozhin}.  The two triangulations
    $\Triang[\MeshPar]$ and $\Triang[\MeshPar/2]$ are shown with
    blue and dashed black lines, respectively.}
  \label{fig:prig-domain}
\end{figure}

In this section we address a test case proposed by~\citet{prigozhin}.
The problem considers the transport of a uniform density supported on
a circle towards a disjoint ellipsis.  \Cref{fig:prig-domain} (left)
shows the domain $\Domain$ where problem~\cref{eq:MK-problem} is
defined and the supports $\SuppPlus$ and $\SuppMinus$ of the forcing
term $\Fsource=\Fsource^+-\Fsource^-$, with $\Fsource^+(x)=2$ for
$x\in\SuppPlus$ and zero otherwise, and $\Fsource^-(y)$, appropriately
rescaled for $y\in\SuppMinus$ to ensure mass balance.  The coarse
initial mesh is also shown in light blue lines, and its uniform
refinement is shown in thin dashed lines. This mesh, characterized by
820 nodes and 1531 triangles, is a constrained Delaunay triangulation
that follows the boundaries of both $\SuppPlus$ and $\SuppMinus$. The
same Figure shows in the right panel the time-converged spatial
distribution of the transport density numerically evaluated with the
most stable discretization method,
$\PC{1,\MeshPar/2}-\PC{0,\MeshPar}$, on the finest mesh. The spatial
distribution of $\TdensH$ is in good agreement with the results
obtained by~\citet{prigozhin}, achieving its maximum value (0.482) on
the boundary of the left circle, and its minimum value $10^{-10}$ set
by the prescribed lower bound as discussed in \cref{sec:pcg-solution}.
This solution is used in the following digression as a reference
solution.  We would like to note that a similar test case was already
proposed in~\citet{Facca-et-al:2018} to to test the conjecture that
the solution of the dynamic MK problem \cref{eq:MK-problem} converges
at infinite time towards the solution of the static \MKEQS.  In this
section we re-use this example to experimentally discuss the need to
use different FEM spaces for the discretization of the transport
density and of the transport potential.

\begin{figure}
  \centerline{
    {\includegraphics[width=0.3\textwidth]{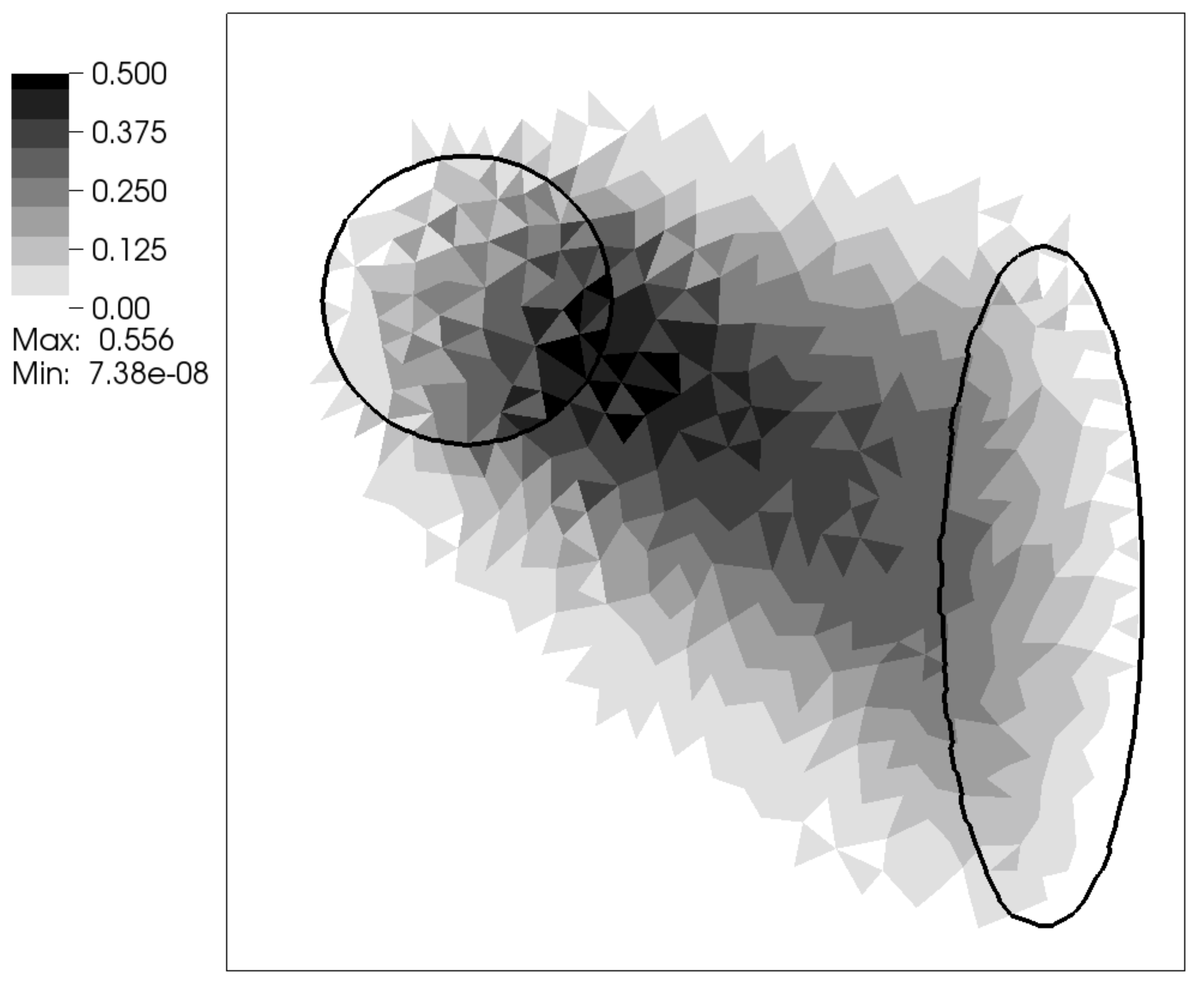}}
    \quad
    {\includegraphics[width=0.3\textwidth]{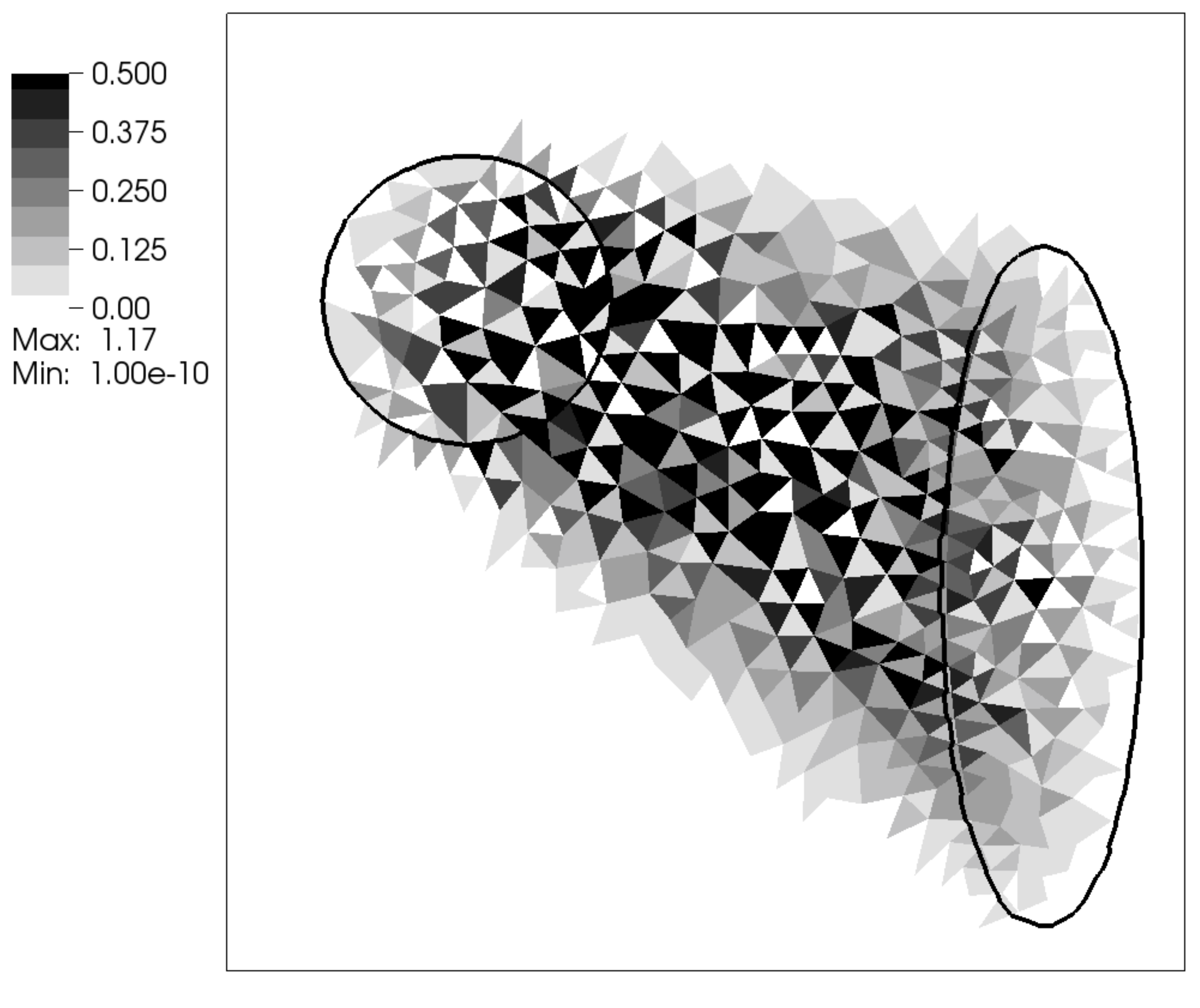}}
    \quad
    {\includegraphics[width=0.3\textwidth]{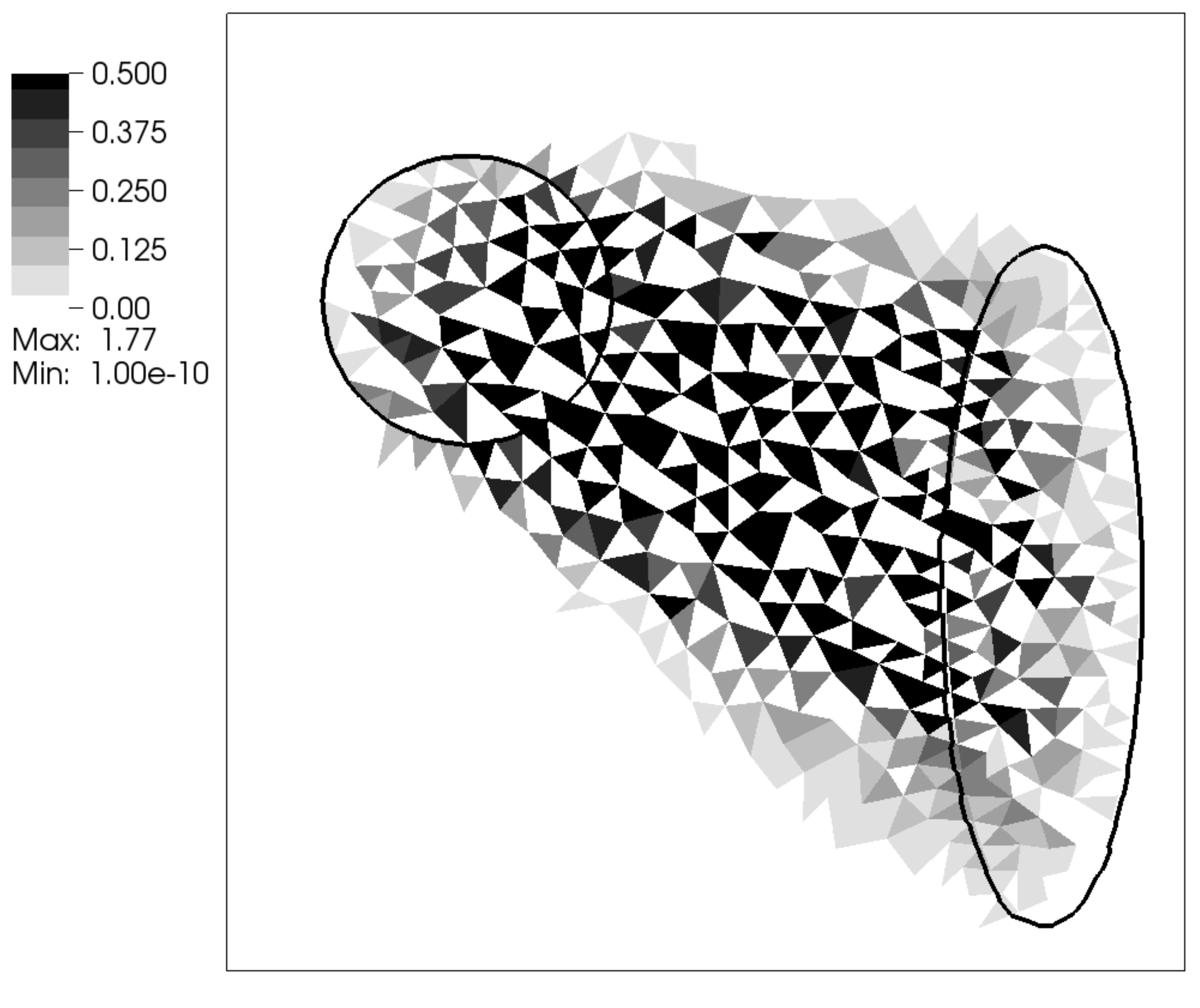}}
  }
  \centerline{
    {\includegraphics[width=0.3\textwidth]{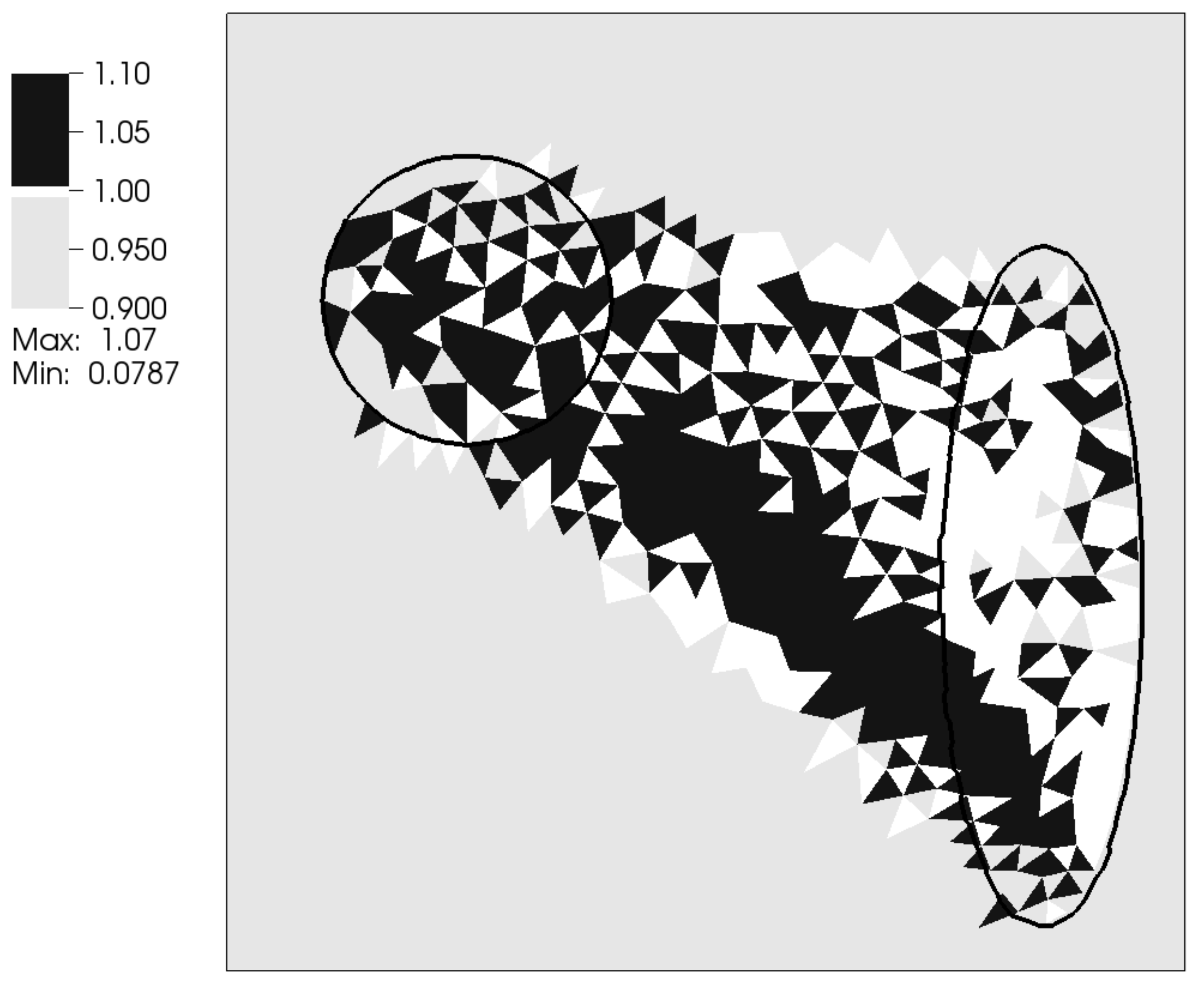}}
    \quad
    {\includegraphics[width=0.3\textwidth]{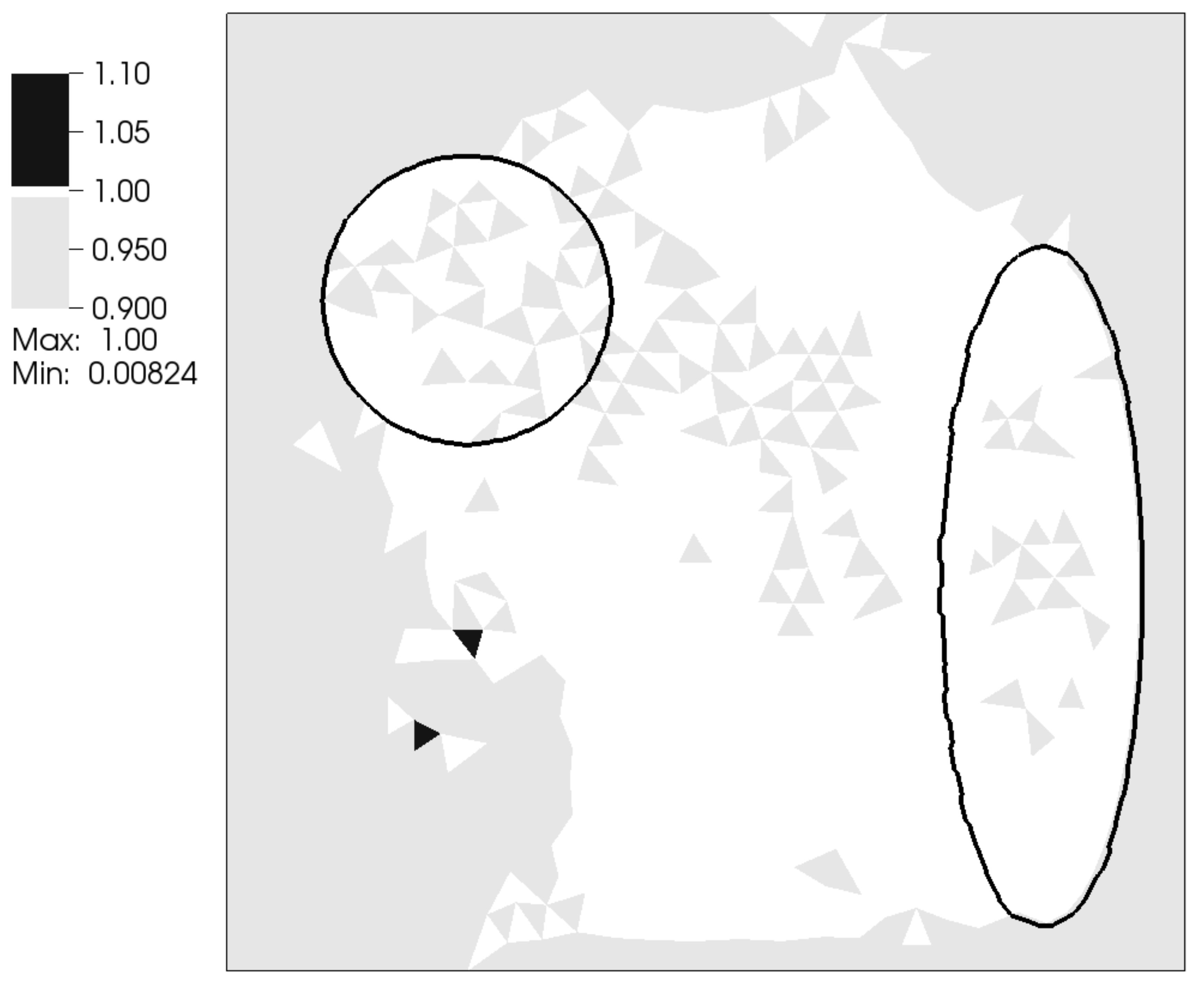}}
    \quad
    {\includegraphics[width=0.3\textwidth]{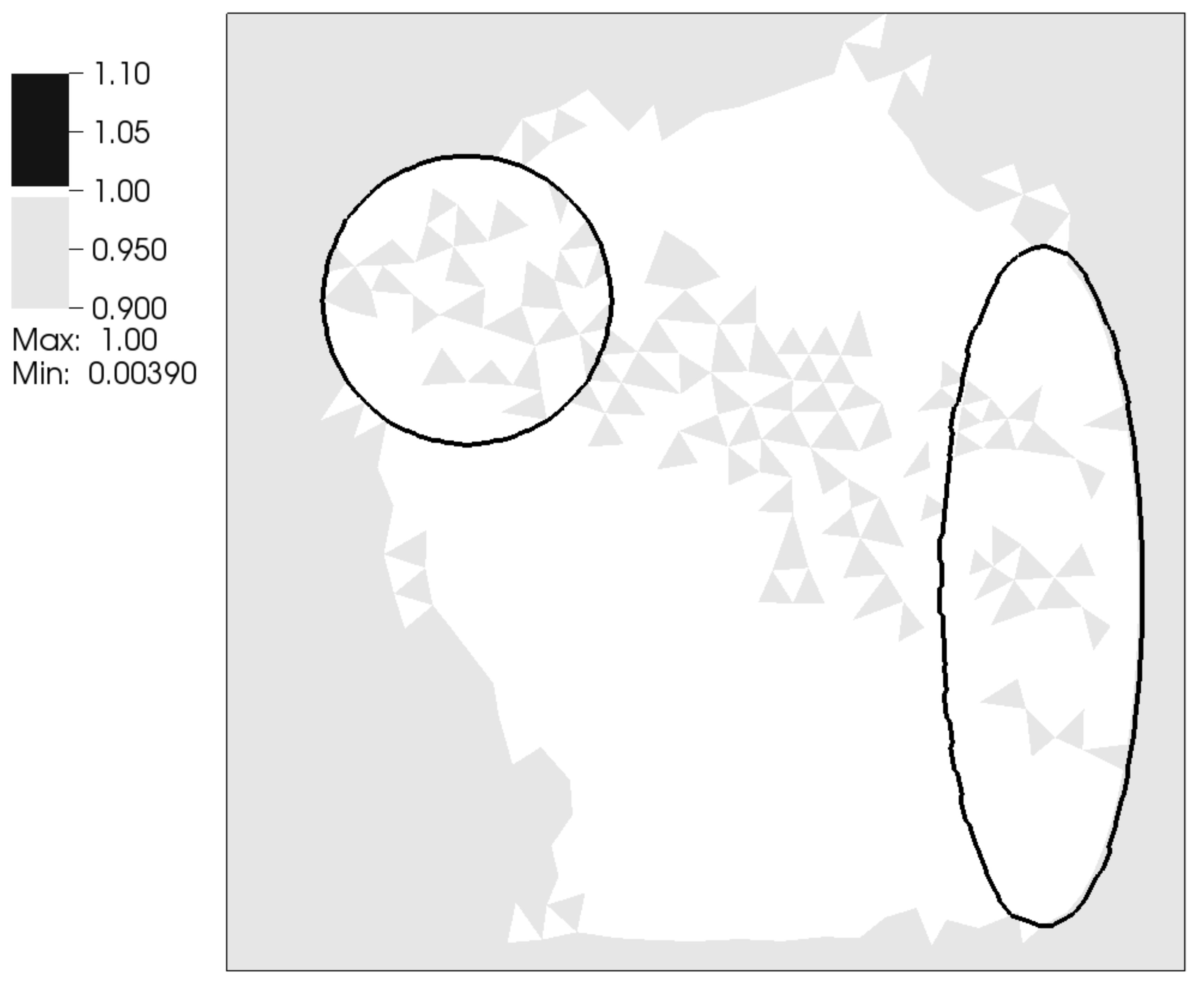}}
  }
  \caption{ Solution of Test~Case~2 at $t_1=2.73\times 10^{2}$ (left),
    $t_2=1.36\times 10^{3} $(center) , and
    $t_3=2.5\times 10^{5}$(right), using the
    $\PC{1,\MeshPar}-\PC{0,\MeshPar}$ approach. Top row: spatial
    distribution of $\TdensH\in\PC{0,\MeshPar}$. Bottom row: spatial
    distribution of $|\Grad\PotH|\in\PC{0,\MeshPar}$ as calculated
    from $\PotH\in\PC{1,\MeshPar}$.}
  \label{fig:prig_p0}
\end{figure}

We start this discussion by presenting the results obtained using the
$\PC{1,\MeshPar}-\PC{0,\MeshPar}$ approach on the coarsest grid and
look at three different times during the evolution. The times are
selected so that $\Var(\TdensH^{\tstep})$ reaches the values
$10^{-3}$, $10^{-4}$, $5\times 10^{-8}$, namely
$t_1=2.73\times 10^{2}$, $t_2=1.36\times 10^{3} $, and
$t_3=2.5\times 10^{5}$, the latter time corresponding to the
time-converged solution.  We plot in \cref{fig:prig_p0} both $\TdensH$
(upper panels) and $|\Grad\PotH|$ (lower panels).

At the first sampled time the solution clearly resembles the reference
solution shown in \cref{fig:prig-domain} (right), although at a much
coarser resolution. The corresponding gradient (shown in the second
row) displays some slight but acceptable overshoots in a region that
resembles $\SuppOT$. Already at this early time, which occurs after
1630 time steps, some oscillations are visible.  At time $t_2$ these
oscillations are much more pronounced with a checkerboard pattern that
suggests an intrinsic instability of the scheme. We should note that
the color scale in the plots are limited from above and from below by
suitable values to emphasize the oscillations.  The maximum and
minimum values for both $\TdensH$ and $|\Grad\PotH|$ are reported
right below each legend.  We observe that there are no overshoots in
$|\Grad\PotH|$, which at the final time is never greater than
one. Still, checkerboard-like fluctuations are visible, causing the
dynamic equation to drive $\TdensH$ to zero quickly thus determining a
drastic deterioration of the solution accuracy.

\begin{figure}
  \centerline{
    {\includegraphics[width=0.3\textwidth]{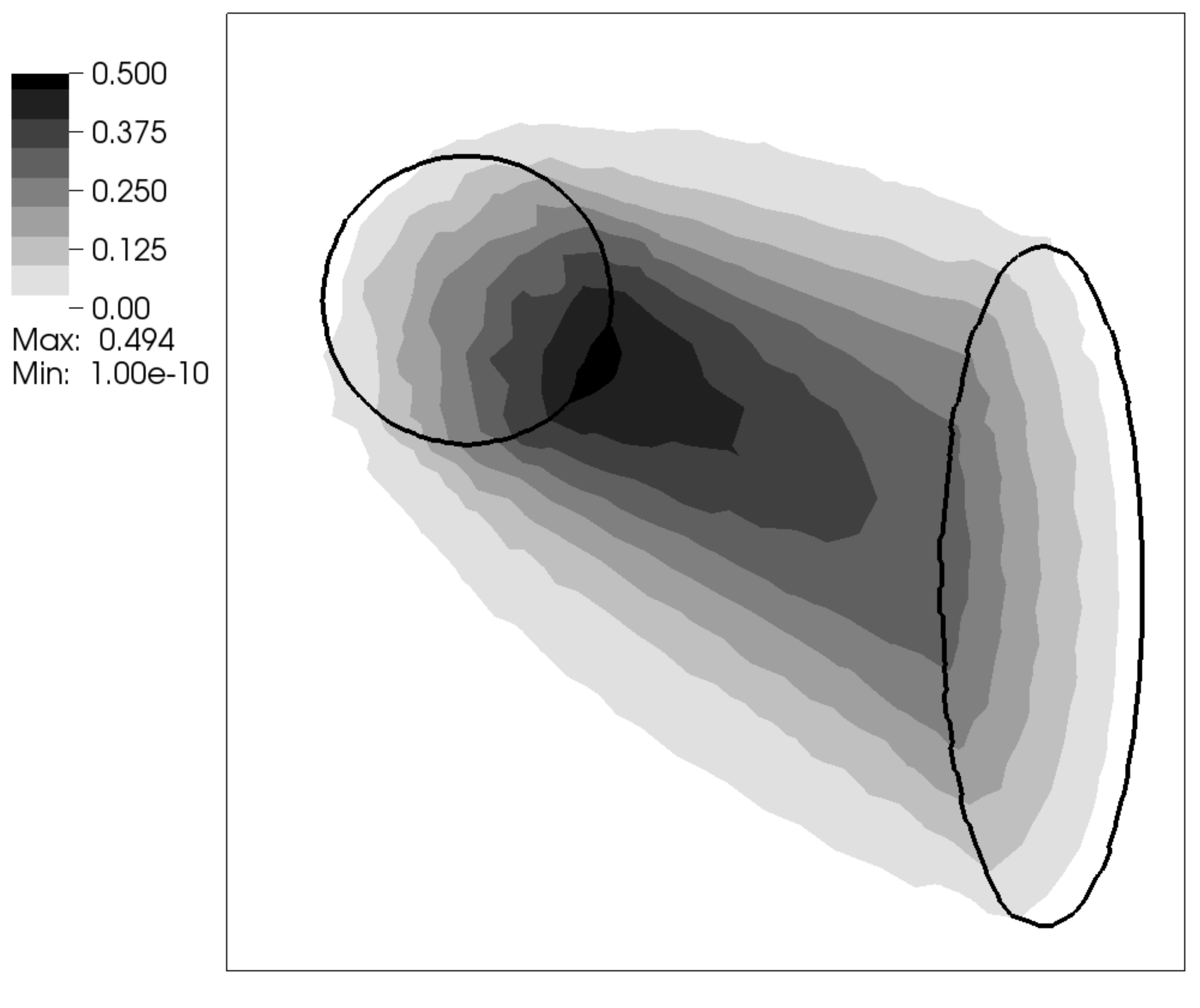}}
    \quad
    {\includegraphics[width=0.3\textwidth]{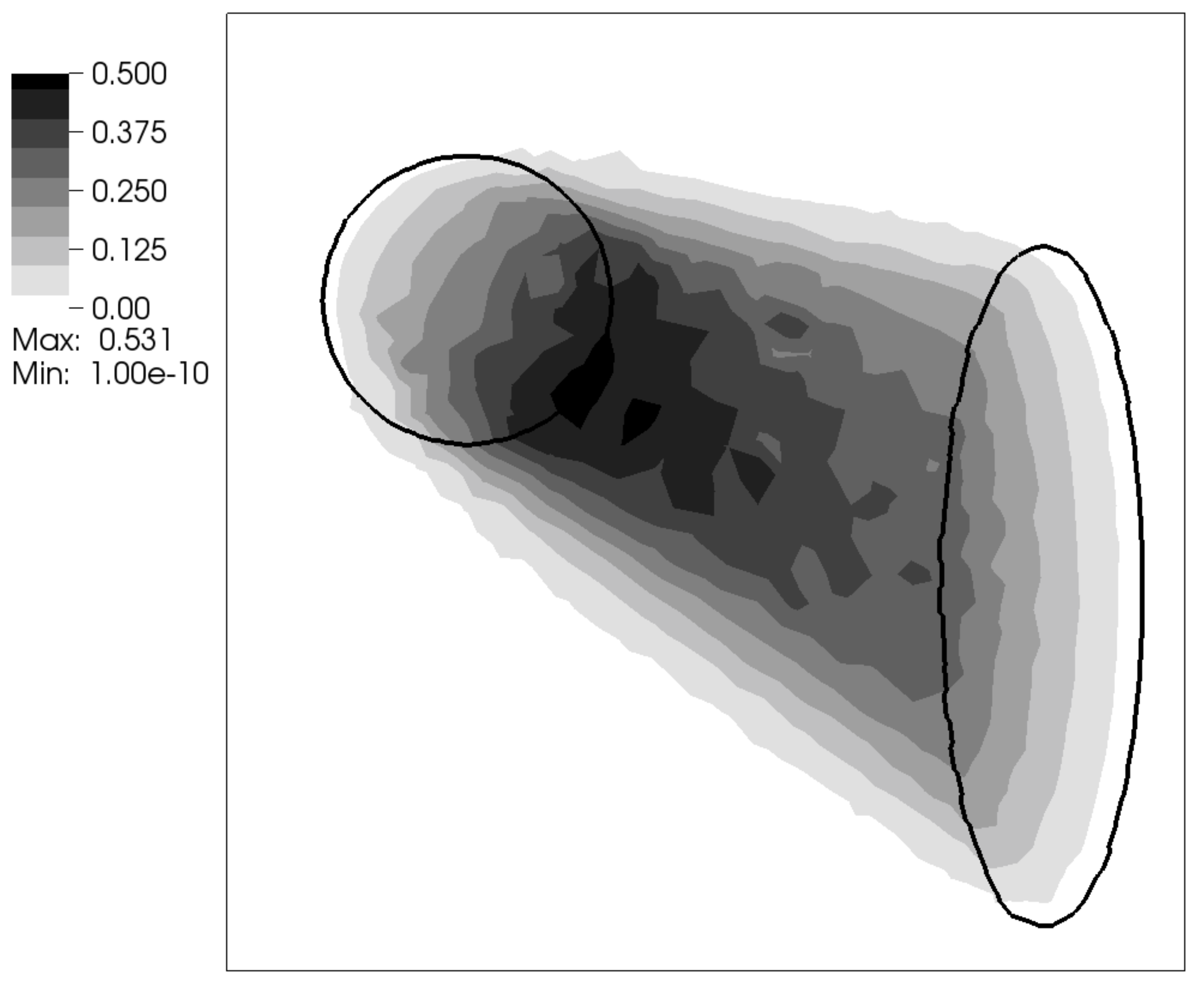}}
    \quad
    {\includegraphics[width=0.3\textwidth]{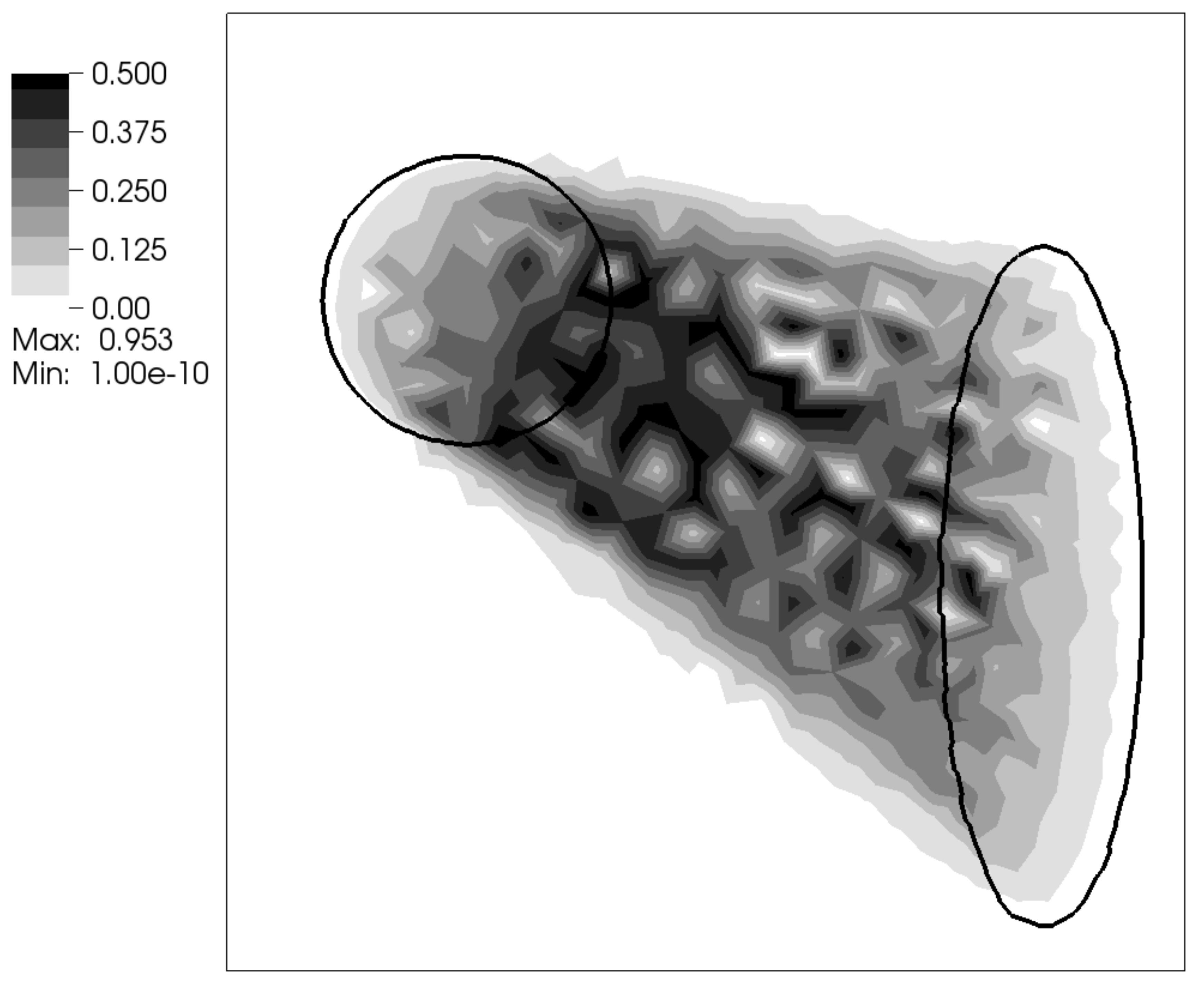}}
  }
  \centerline{
    {\includegraphics[width=0.3\textwidth]{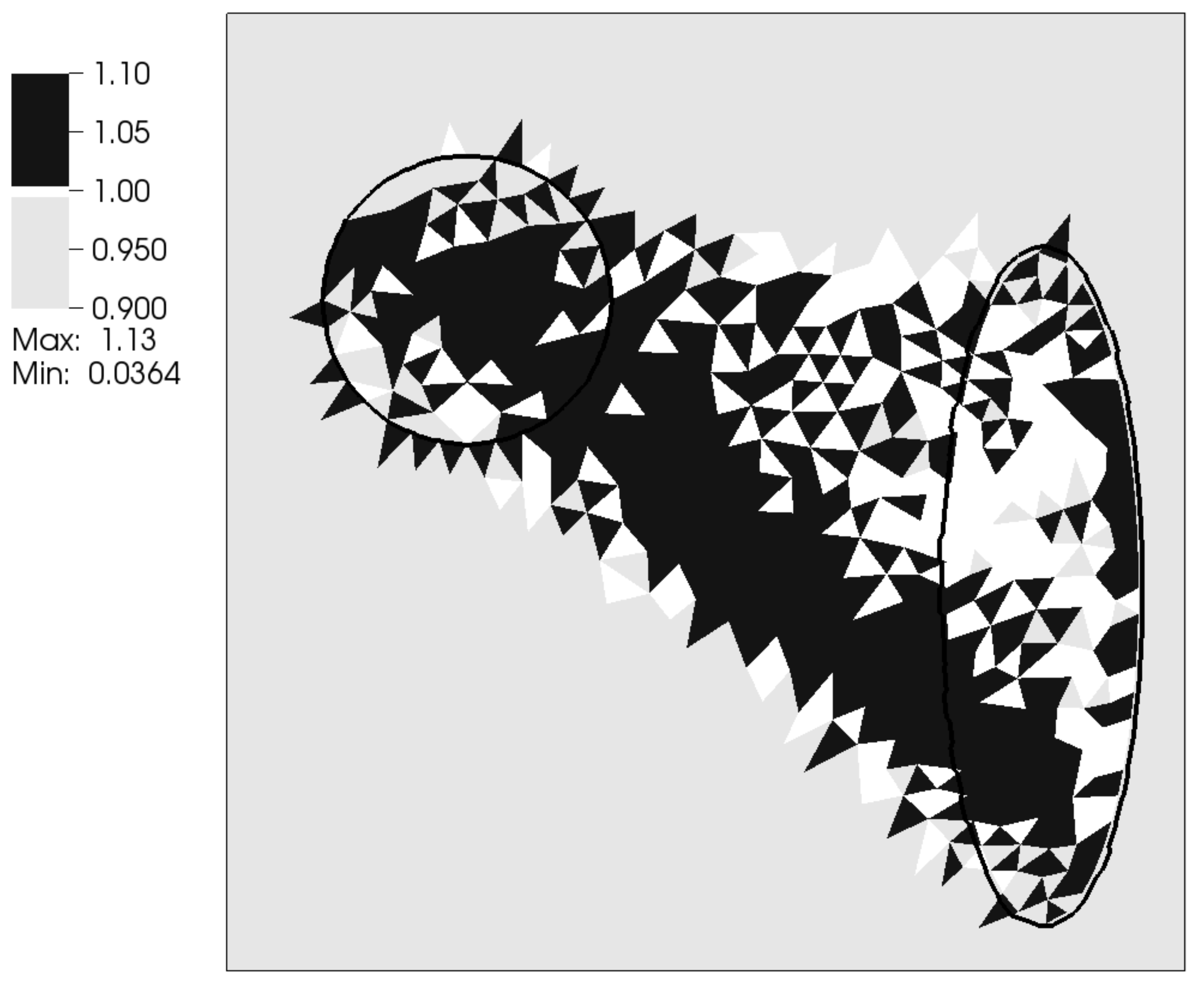}}
    \quad
    {\includegraphics[width=0.3\textwidth]{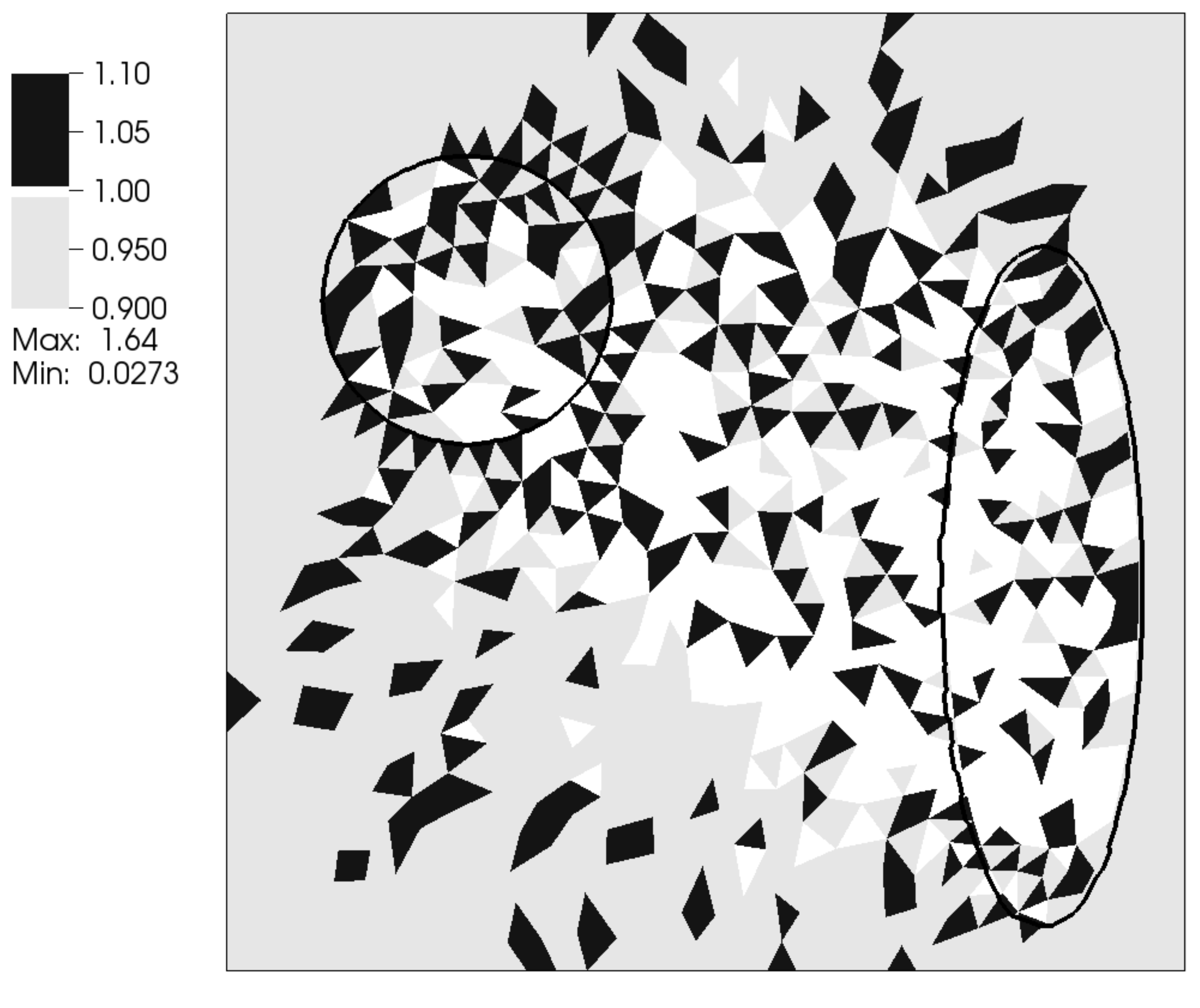}}
    \quad
    {\includegraphics[width=0.3\textwidth]{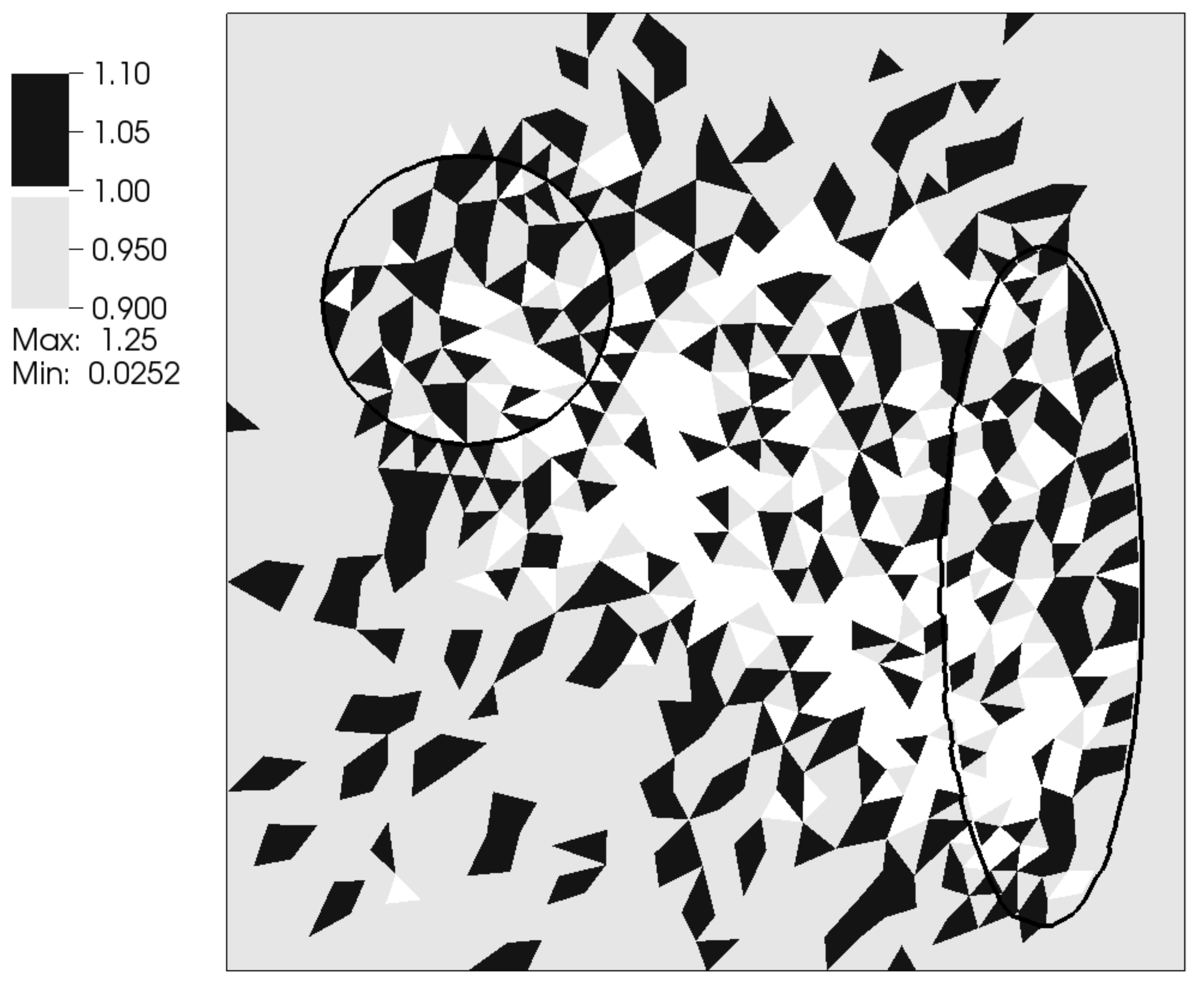}}
  }
  \caption{ Solution of Test~Case~2 at $t_1=6.82\time 10^1$ (left),
    $t_2=5.36\time 10^2$ (center), and $t_3=9.65time 10^3$ (left),
    using the $\PC{1,\MeshPar}-\PC{1,\MeshPar}$ approach.  Top row:
    spatial distribution of $\TdensH\in\PC{1,\MeshPar}$. Bottom row:
    spatial distribution of $|\Grad\PotH|\in\PC{0,\MeshPar}$
    calculated from $\PotH\in\PC{1,\MeshPar}$.}
  \label{fig:prig_p1}
\end{figure}

The situation does not improve by using higher order spaces for
$\TdensH$. \Cref{fig:prig_p1} shows the results obtained by using a
$\PC{1,\MeshPar}-\PC{1,\MeshPar}$ approach. We still observe
oscillations, albeit appearing at a later time and with a different
pattern.  Once oscillations in $|\Grad\PotH|$ around the unit value
start developing, the dynamic equation determines a decay of $\TdensH$
within the elements where $|\Grad\PotH|<1$ even if located within
$\SuppOT$.  This decay quickly reinforces in time leading to the
observed checkerboard pattern.  The behavior resembles the classical
lack of stability due to a violation of an inf-sup-like constraint,
but at this point we are not able to clearly identify this condition.

\begin{figure}
  \centerline{
    {\includegraphics[width=0.3\textwidth]{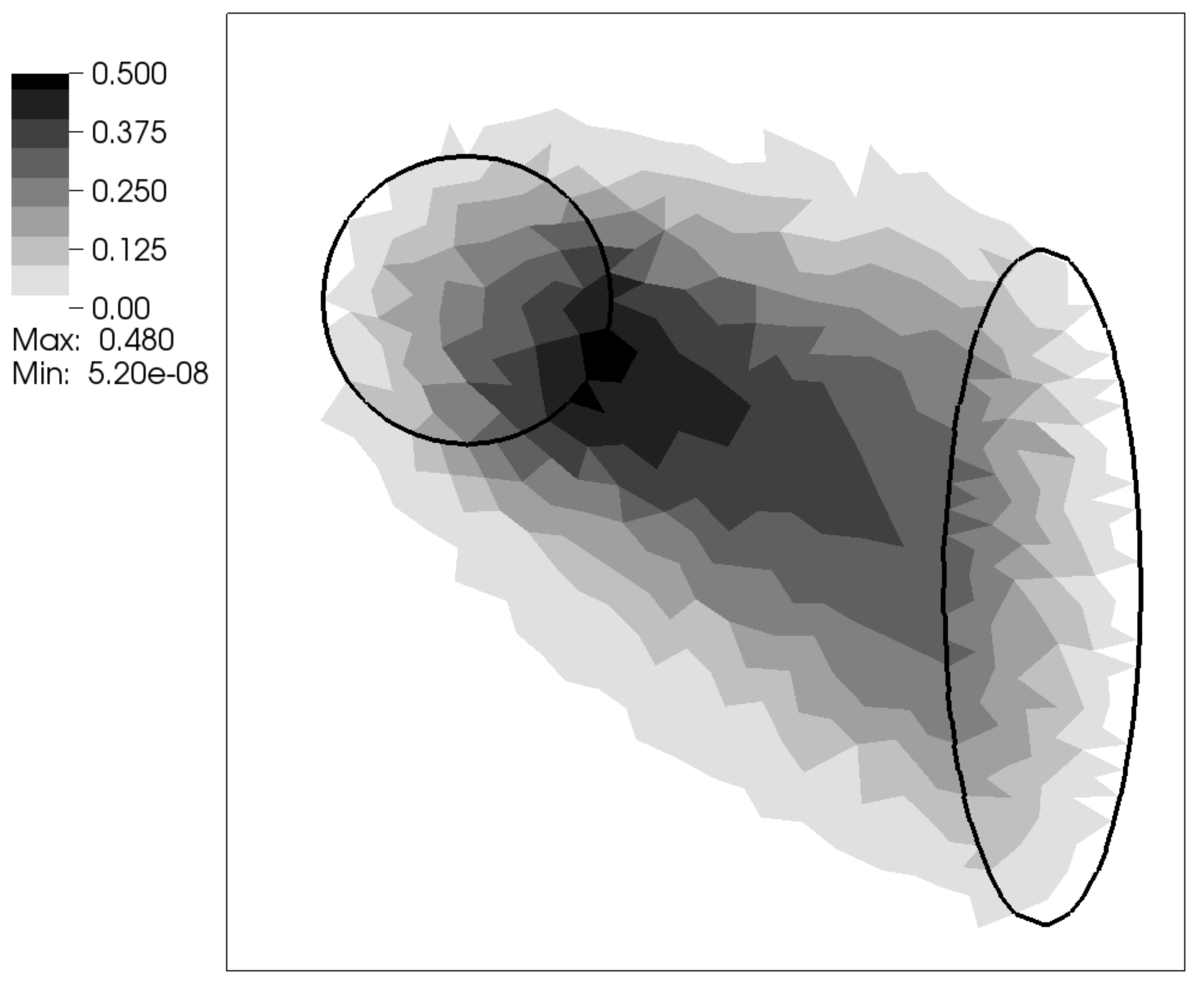}}
    \quad
    {\includegraphics[width=0.3\textwidth]{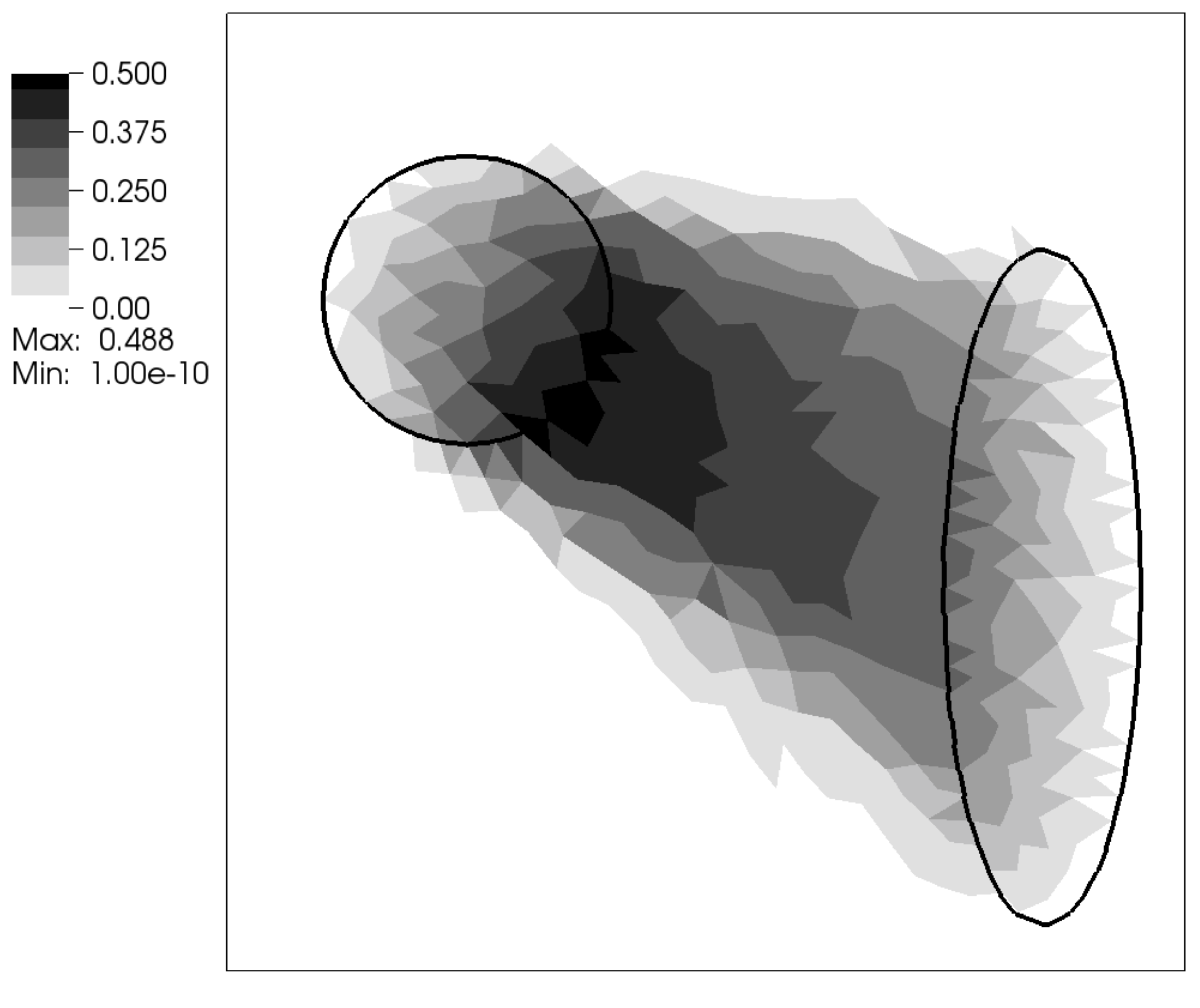}}
    \quad
    {\includegraphics[width=0.3\textwidth]{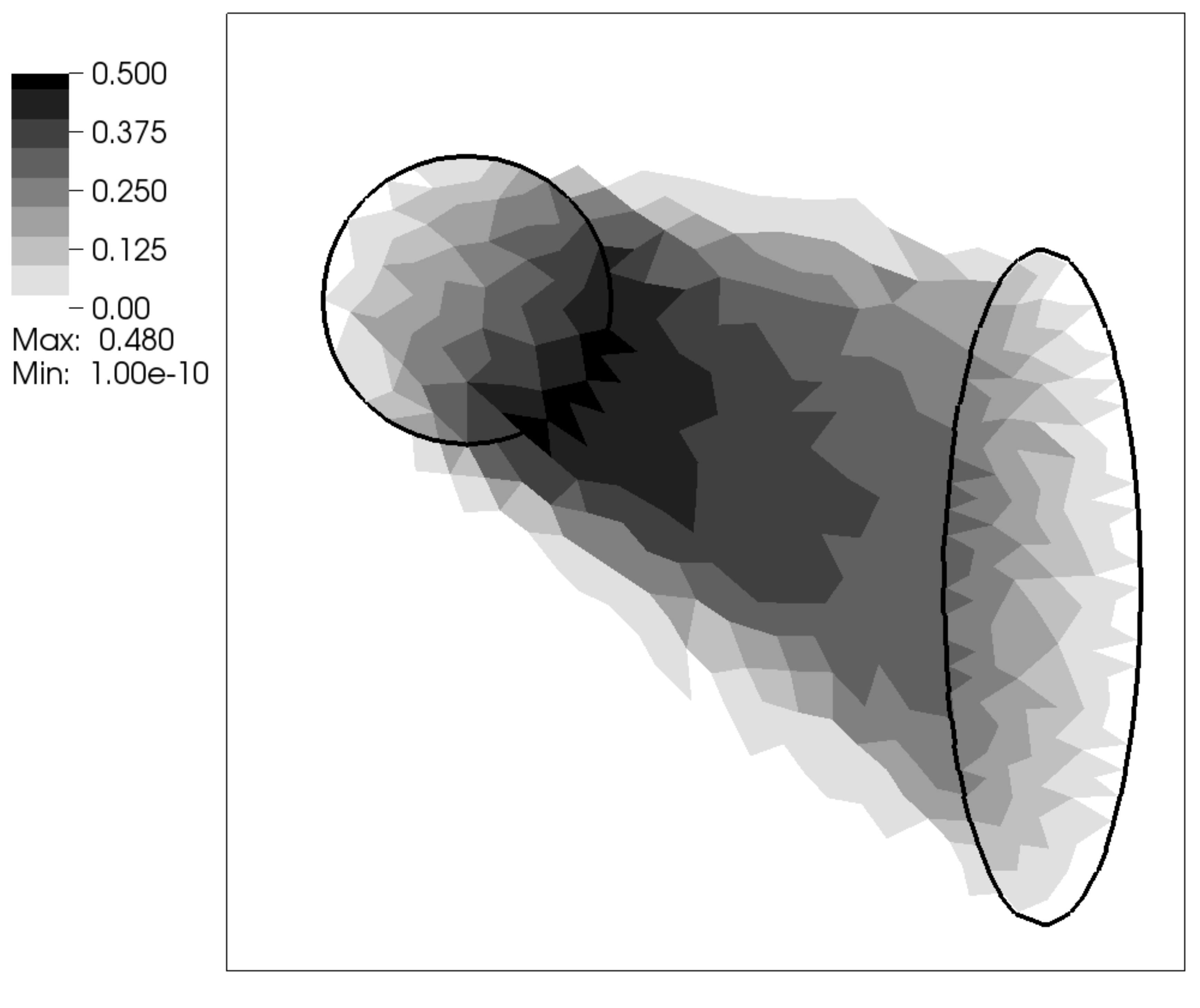}}
  }
  \centerline{
    {\includegraphics[width=0.3\textwidth]{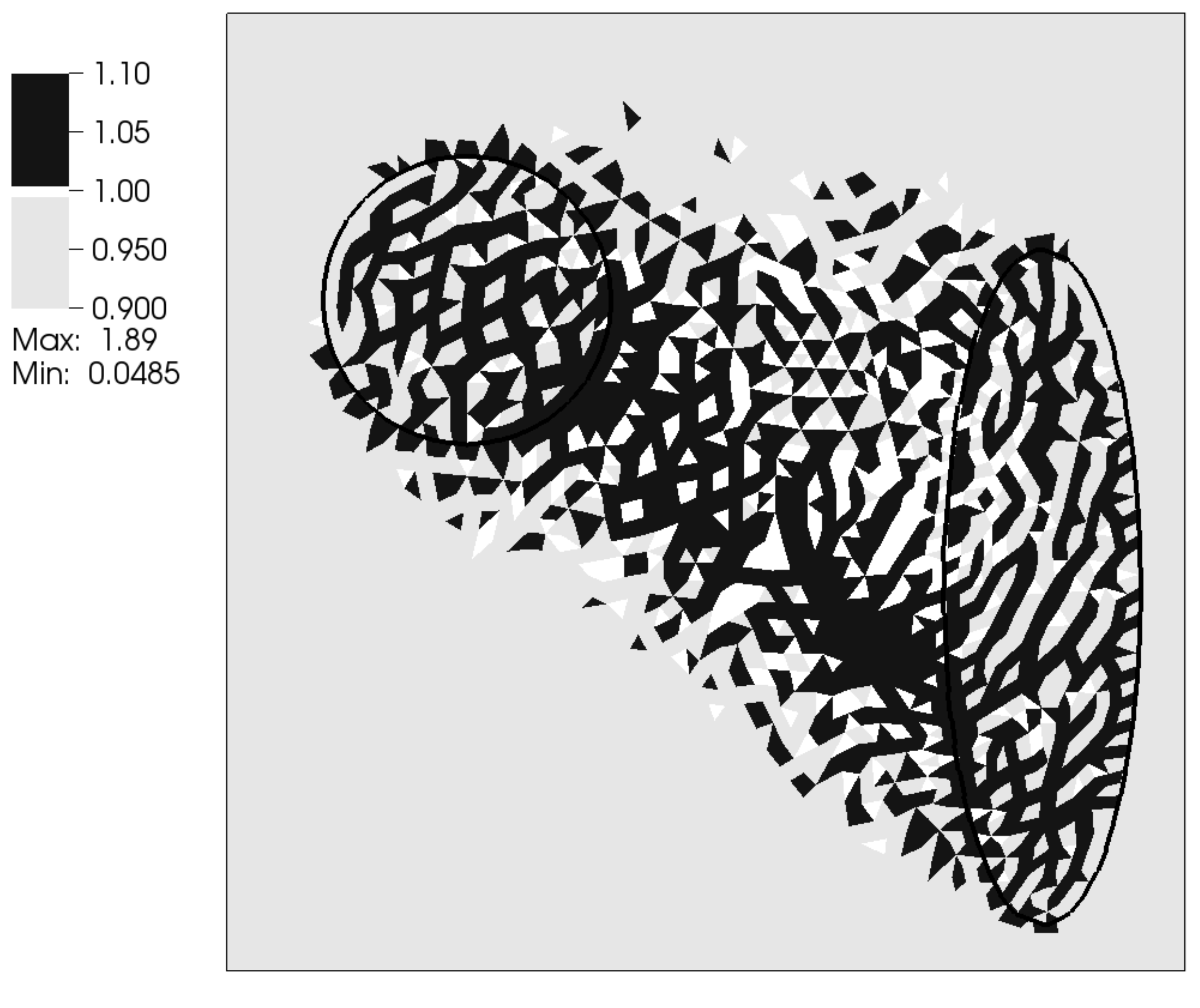}}
    \quad
    {\includegraphics[width=0.3\textwidth]{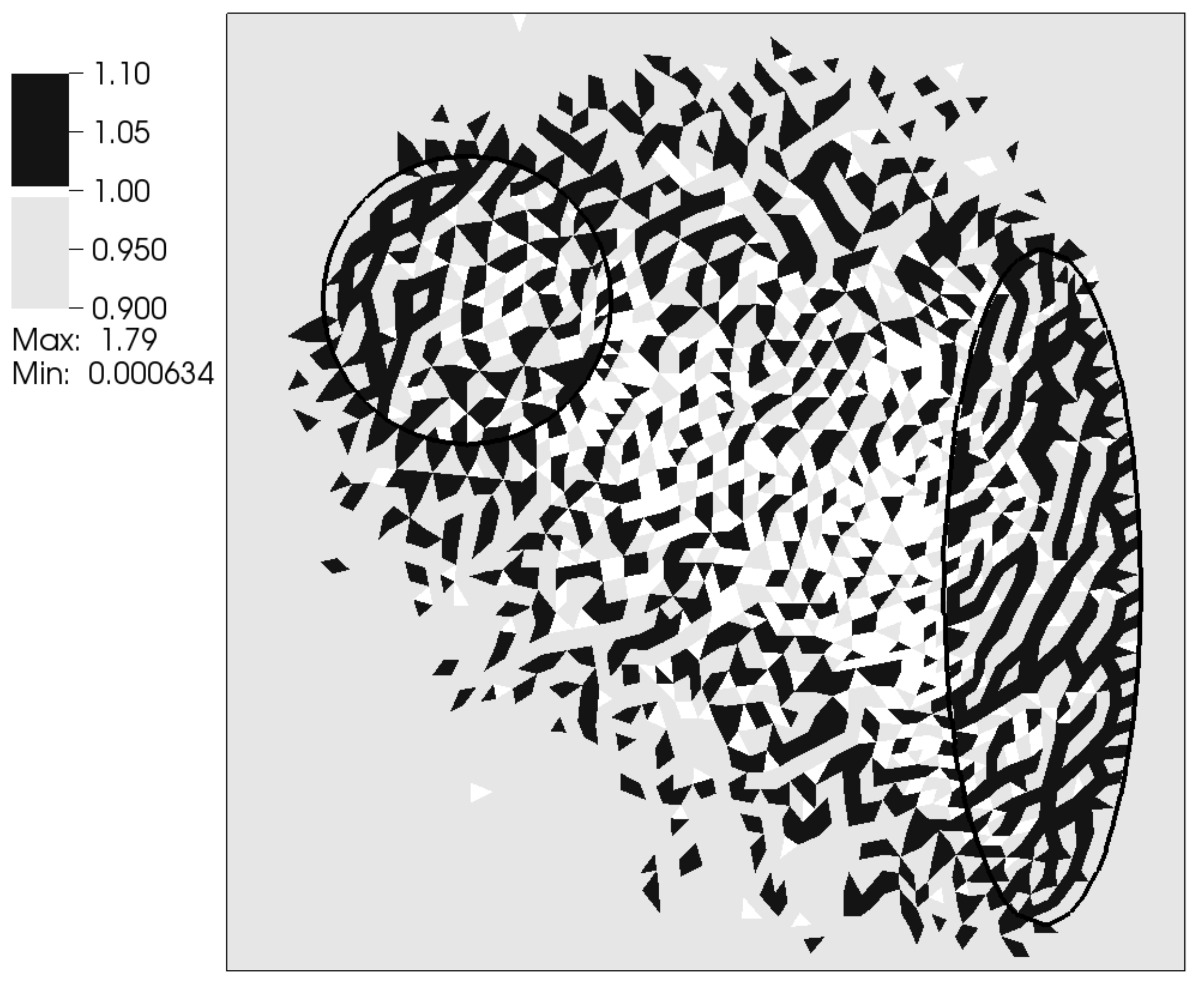}}
    \quad
    {\includegraphics[width=0.3\textwidth]{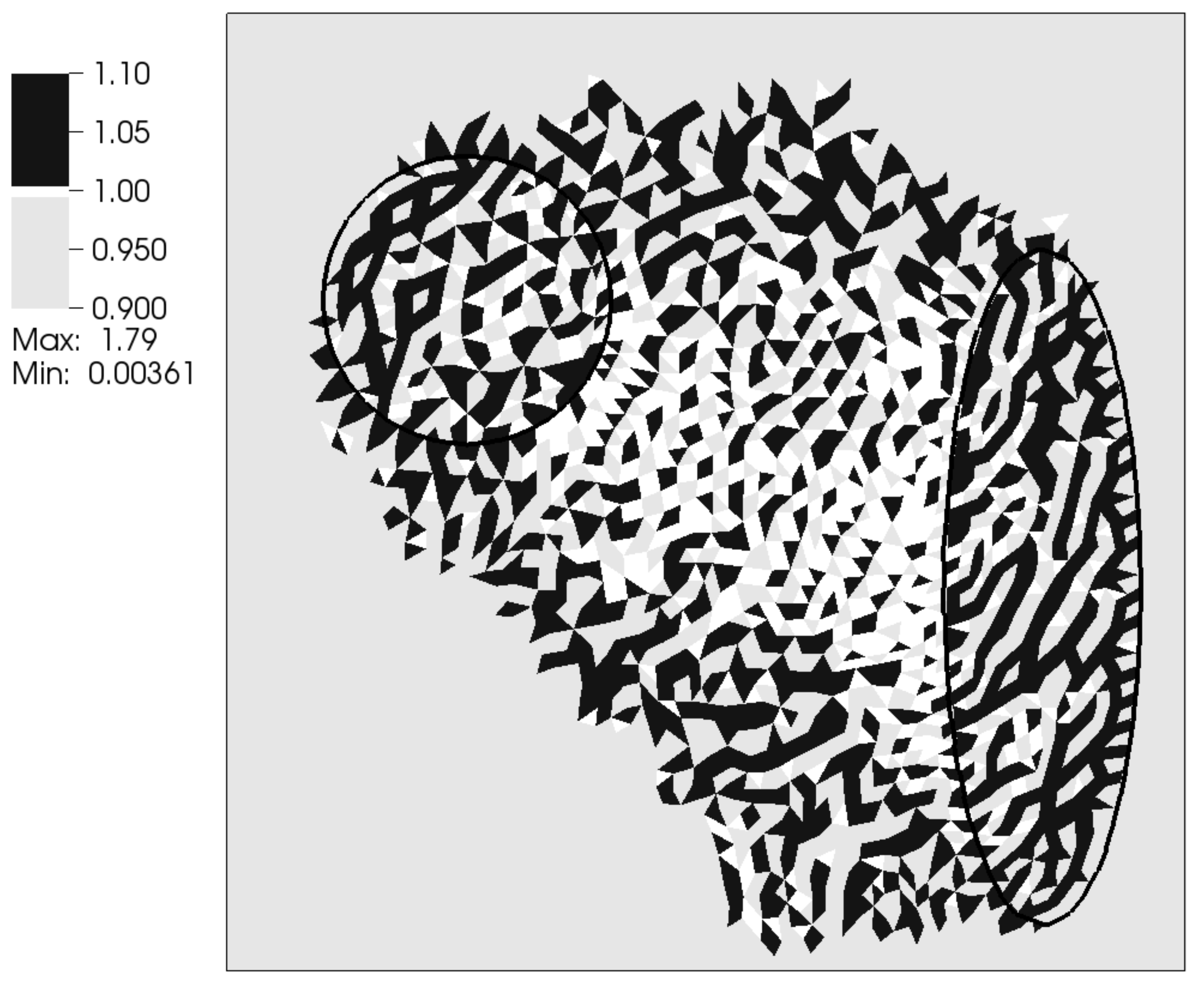}}
  }
  \centerline{
    {\includegraphics[width=0.3\textwidth]{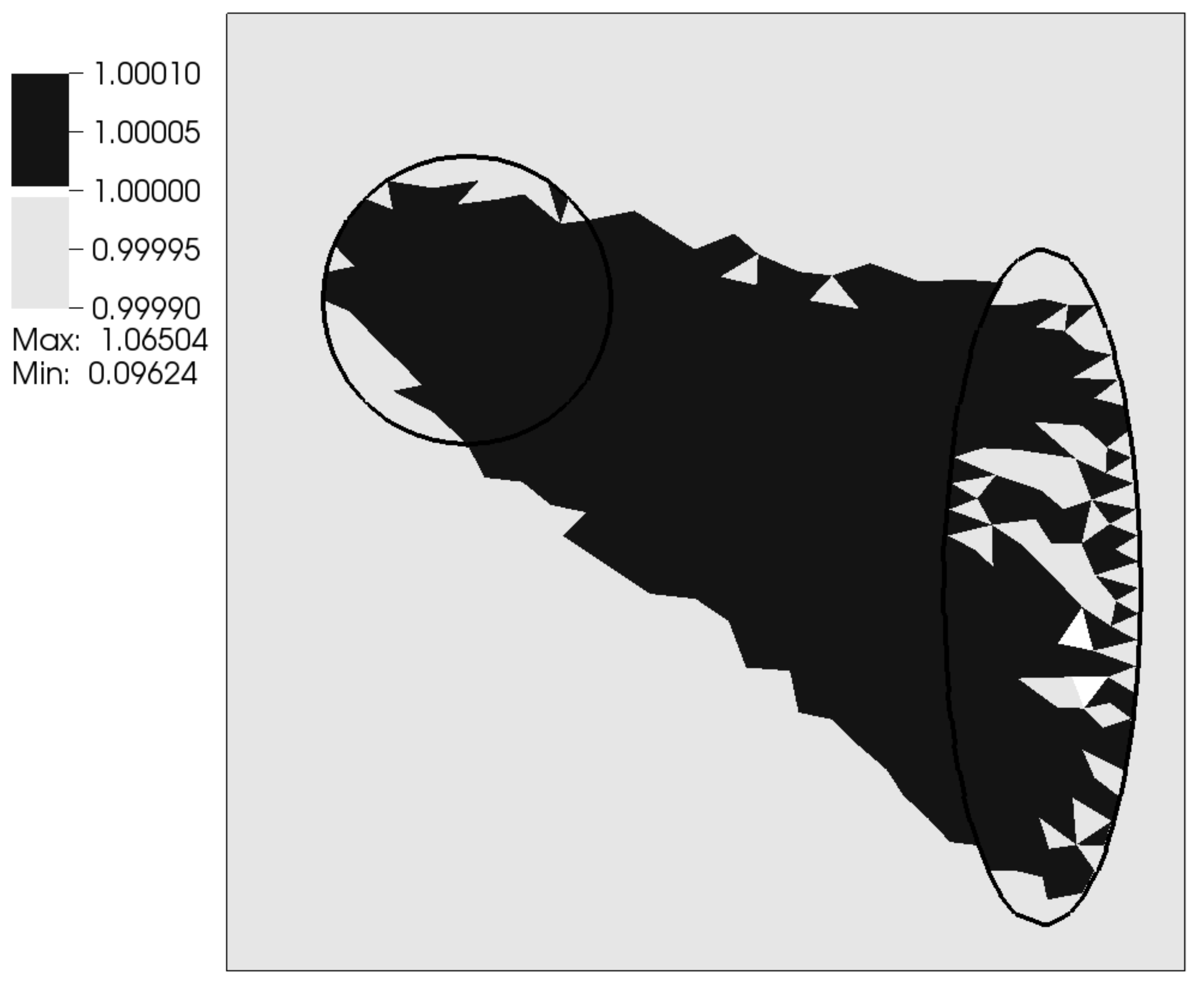}}
    \quad
    {\includegraphics[width=0.3\textwidth]{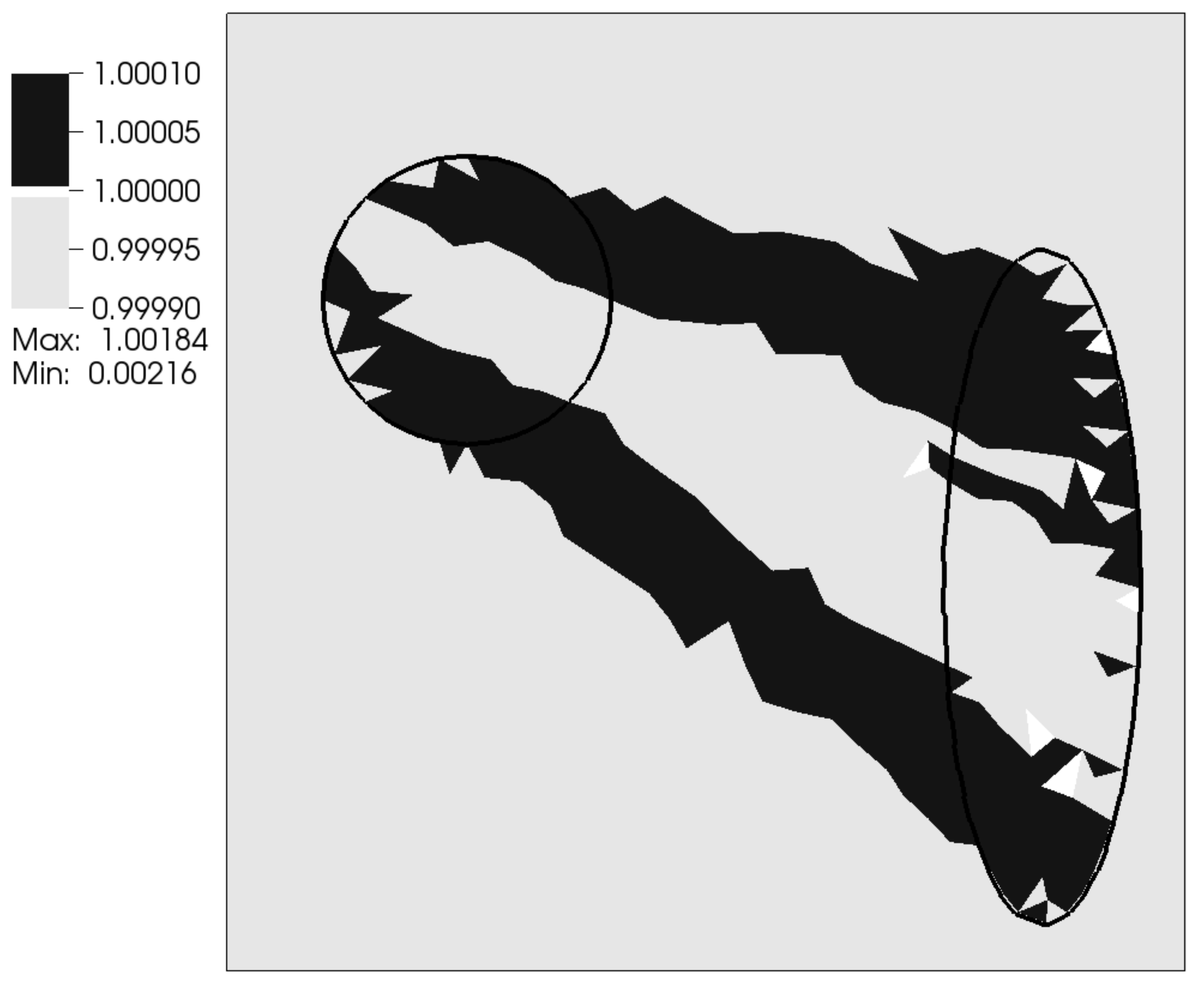}}
    \quad
    {\includegraphics[width=0.3\textwidth]{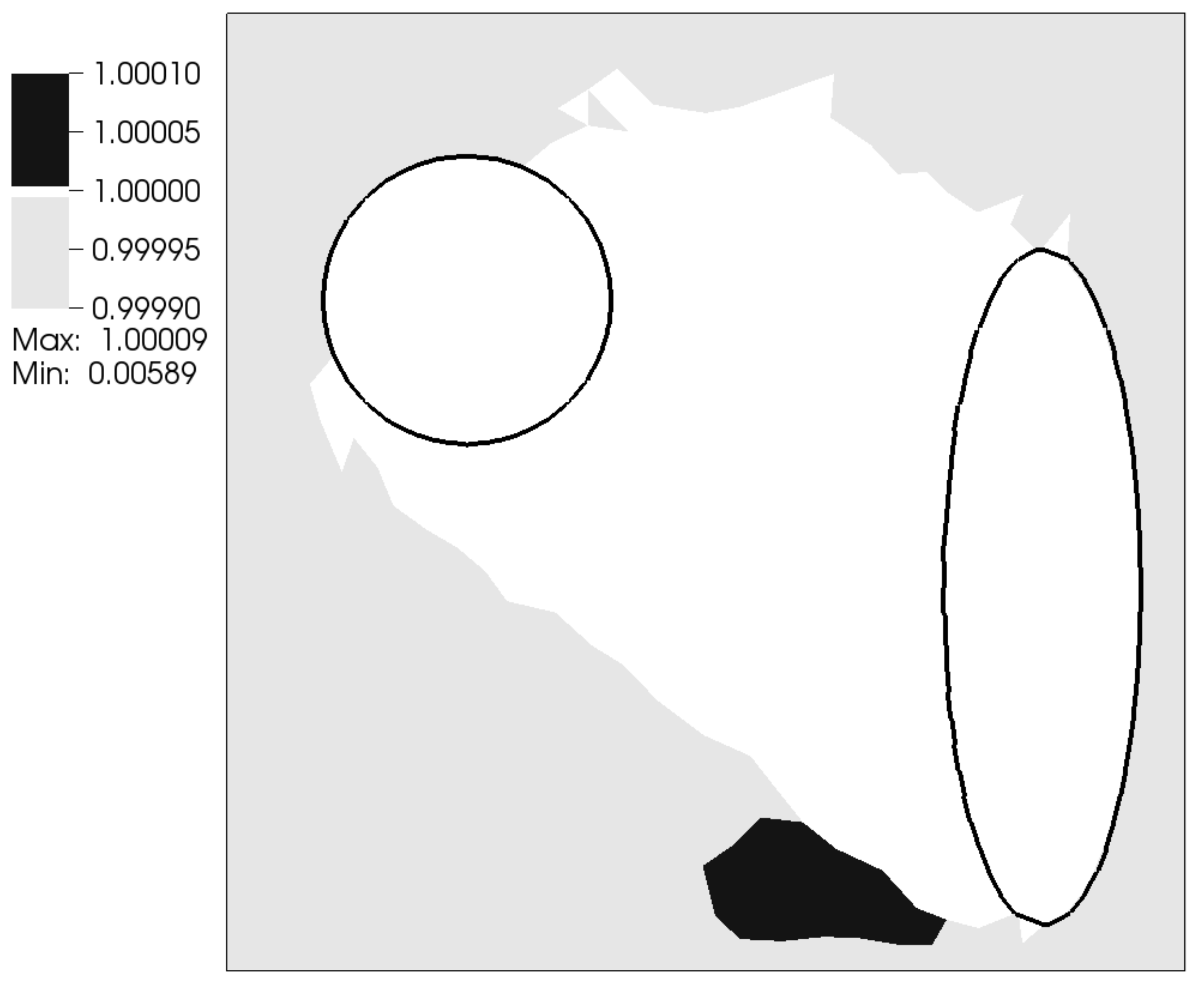}}
  }
  \caption{ Solution of Test~Case~2 at $t_1 = 6.75 \times 10^{1}$
    (left), $t_2 = 2.04 \times 10^{2}$ (center), and
    $t_3=1.54 \times 10^{3} $ (right) using the
    $\PC{1,\MeshPar/2}-\PC{0,\MeshPar}$ approach.  Top row: spatial
    distribution of $\TdensH\in\PC{0,\MeshPar}$. Middle row: spatial
    distribution of $|\Grad\PotH|\in\PC{0,\MeshPar/2}$ calculated from
    $\PotH\in\PC{1,\MeshPar/2}$. Bottom row: spatial distribution of
    $|\Grad\PotH|\in\PC{0,\MeshPar}$.}
  \label{fig:prig_p0_sub}
\end{figure}

On the other hand, oscillations completely disappear if we employ a
two-mesh approach. Looking at the checkerboard oscillations
displayed in \cref{fig:prig_p0}, it is intuitive to think that
averaging the gradient magnitude between neighboring triangles
should compensate the fluctuations. This observation led us to 
employ the $\PC{1,\MeshPar/2}-\PC{0,\MeshPar}$ discretization
described in \cref{sec:discr}. Indeed, with this approach the
gradients calculated from $\PotH\in\PONE(\Triang[\MeshPar/2])$ are
projected onto the space $\PZERO(\Triang[\MeshPar])$ for insertion
into the dynamic equation~\cref{eq:fem_ode}. This projection is
equivalent to averaging the piecewise constant gradients over the four
triangles of $\Triang[\MeshPar/2]$ that form one triangle of
$\Triang[\MeshPar]$. 
This results in a oscillation free $\TdensH$ field, as shown in
\cref{fig:prig_p0_sub}. It is evident that 
no $\TdensH$ oscillations form even at the coarsest mesh level used in
this test. Note that the spatial discretization of the elliptic
equation does not guarantee monotonicity~\citep{PuttiCordes98}. In
fact, the gradient magnitudes arising from $\PotH\in\PC{1,\MeshPar/2}$
still show the classical checkerboard fluctuations
(\cref{fig:prig_p0_sub}, middle row). However, the projection of
$|\Grad\PotH|$ onto $\PC{0,\MeshPar}$ (\cref{fig:prig_p0_sub},
bottom row) does not show oscillations, albeit small overshooting
occurs especially at the earlier times. We should emphasize 
that $|\Grad\PotH|$ is plotted here using an extremely narrow color scale
ranging within $[0.9999,1.0001]$.  

Looking at the final time-converged solution, the value
$|\Grad\OptPotH|$ within the support of $\OptTdensH$ and neighboring
regions is exactly unitary, and remains bounded by 1 almost
everywhere, in compliance with the constraint of the MK
equations. Only one small region with $|\Grad\OptPotH|>1$ develop with
a maximum value approaching 1.00009, considered consistent with the
tolerance used in the PCG linear solve. Indeed, oscillations of the
order of $10^{-5}$ in the gradient magnitude may be indistinguishable
by the linear solver of the $\PotH$ equation.  Similar considerations
can be done in the case $\TdensH\in\PC{1,\MeshPar}$ (not shown here)
but in this case some oscillations in $|\Grad\PotH|$ persist even when
$\PotH\in\PC{1,\MeshPar/2}$.  This reinforces the conjecture that some
sort of inf-sup stability condition exists that couples the
discretization spaces for $\PotH$ and $\TdensH$, and will be the
subject of further studies.

One final observation for this test case concerns the computational
cost of our approach. In comparison with the technique proposed
by~\citet{prigozhin}, our method seems to be computationally
advantageous. In fact, as already mentioned, the simulations reported
in~\citet{prigozhin} where obtained using a mixed FEM approach in
combination with adaptive mesh refinement, leading to nonlinear
systems of dimension approaching 60000. In our case, the dimensions
for the smallest test case are 1531 (number of triangles in
$\Triang[\MeshPar]$) and 3170 (number of nodes in
$\Triang[\MeshPar/2]$) for the diagonal dynamic algebraic system and
the elliptic system, respectively, leading to a total of 4701 degrees
of freedom.  Note that the finest solution of \cref{fig:prig-domain}
was obtained with a total of 73917 degrees of freedom.  Our confidence
that the approach we propose is superior to that~\citet{prigozhin} is
is reinforced by the observation that effective simulations can be
obtained at intermediate mesh levels. Moreover, time-convergence can
be considered achieved at much earlier times then the ones employed in
this work if we look at the stationarity of the \LCF.  Obviously,
adding simple adaptive mesh refinement strategies would greatly
enhance the performance of the studied methodology.

\subsection{Test case~3:  $\Lspace{1}$-Optimal Transport map}

In this section, we present a further experiment where the numerical
solution $(\OptTdensH,\OptPotH)$ from the \DMK\ approach is used to
compute approximate $\Lspace{1}$-OT Maps following the algorithm
suggested in~\citet{EvansGangbo99}.  Finally, these results are
compared with approximate maps obtained by linear programming or the
Sinkhorn algorithm with entropic
regularization~\citep{Perrot-et-al:2016}.

This experiment is shaped after~\citet{Li-et-al:2018}, and considers a
square domain $\Domain=[0,1]\times[0,1]$ with a source term $\Source$
supported in the ball of radius $0.35$ centered and
$(x_c,y_c)=(0.5,0.5)$ given by
$\Source(x,y)=max(0,1-\sqrt{(x-x_c)^2+(y-y_c)^2})$.  The sink term
$\Sink$ is the sum of four functions of the same form of $\Source$,
but located near the four corners of the domain. All the terms are
balanced to ensure zero mean of $\Fsource$.  The problem setting
together with the triangulation used in the \DMK\ solution is shown
in~\cref{fig:maps}.  Note that in this test case the OT map
will necessarily split $\Source$ into four different subsets that are
reallocated towards the four disjoint sinks $\Sink$. The resulting
singular distribution poses non trivial issues on the numerical
solution of the problem.

\paragraph{Calculation of the OT map via \DMK}

Given two Lipschitz continuous forcings $\Source$ and $\Sink$ with
disjoint supports, the OT map
$\OptMap:\Supp(\Source) \mapsto \Supp(\Sink)$ can be defined as
$\OptMap(x):=z(1,x)$, where $z(t,x)$ is the solution of the following
Cauchy Problem~\citep{EvansGangbo99}:
\begin{equation}
  \label{eq:optimal-field}
  \left\{
    \begin{aligned}
      z'(t)&=\Vectorfield(t,z(t))\\
      z(0)&=x \in \SuppPlus
    \end{aligned} 
  \right.
  \quad
  \Vectorfield(t,z)
  =
  \frac{-\OptTdens\Grad\OptPot}{(1-t)\Source(z)+t\Sink(z)}.
\end{equation}
Thus, we first compute $(\OptTdensH,\OptPotH)$ via the
$\PC{1,\MeshPar/2}-\PC{0,\MeshPar}$ approach combined, for simplicity,
with Explicit Euler time-stepping. We adopt the same parameters (time
step, linear solver tolerance, stop criteria, etc.)  used in the other
experimental tests in this paper. Then, we construct the approximate
OT map by evaluating the streamlines of the vector field
$\Vectorfield$ in~\cref{eq:optimal-field} emanating from the
barycenters of the triangles discretizing the support of $\Source$.
In order to avoid division by zero in~\cref{eq:optimal-field}, we
replace the term $(1-t)\Source(z)+t\Sink(z)$ with
$min[(1-t)\Source(z)+t\Sink(z)],10^{-5}]$.  Time-integration is
performed with a 4-th order Runge-Kutta method.  We denote with
$\OptMapH(\DMK)$ this approximate OT map.

\paragraph{Calculation of OT map via barycentric map}

The second method considered follows the approach detailed
in~\citet{Perrot-et-al:2016}, where an approximate OT map is built from
an approximate OT plan.  First, $\Source$ and $\Sink$ are discretized
with two atomic measures $\SourceH$ and $\SinkH$:
\begin{equation}
  \label{eq:atomic}
  \SourceH
  =
  \sum_{i=1}^{\DimSource} 
  \Vect[i]{\DiscreteSource}\Dirac(\bar{x}_i)
  \quad
  \SinkH
  =
  \sum_{j=1}^{\DimSink} 
  \Vect[j]{\DiscreteSink}  \Dirac(\bar{y}_j),
\end{equation}
where $(\bar{x}_{i})_{i=1,\ldots,\DimSource}$ and
$(\bar{y}_{j})_{i=1,\ldots,\DimSink}$ are sampling points in the
support of $\Source$ and $\Sink$.  We then denote with
$\Matr{\OptPlan}$ the solution of the Linear Programming problem
solving the classical $\Lspace{1}$-\OTP\ given $\SourceH,\SinkH$.  The
associated \emph{barycentric map} $\OptMap[\OptPlan]$ is defined as:
\begin{equation}
  \label{eq:bari-map}
  \OptMap[\Matr{\OptPlan}](\bar{x}_i)
  :=
  \Argmin_{y\in \Domain}
  \sum_{j=1}^{\DimSink}\Matr[i,j]{\OptPlan}\ |y-\bar{y}_{j}|
  \quad
  \forall i=1,\ldots,\DimSource.
\end{equation}
In order to build the approximate OT map that can then be compared
with $\OptMap(\DMK)$, the sampling points
$(\bar{x}_{i})_{i=1,\ldots,\DimSource}$ and
$(\bar{y}_{j})_{i=1,\ldots,\DimSink}$ are taken to be the barycenters
of the triangles discretizing the support of $\Source$ and $\Sink$,
respectively. The coefficients
$\Vect{\DiscreteSource}\in\REAL^{\DimSource}$ and
$\Vect{\DiscreteSink}\in\REAL^{\DimSink}$ in~\cref{eq:atomic} are then
computed as:
\begin{equation*}%  \label{eq:discrete-measures}
  \Vect[i]{\DiscreteSource}=\int_{\Cell[i]} \Source \dx
  \quad \forall \Cell[i]\in \Supp(\Source)
  \qquad
  \Vect[j]{\DiscreteSink}=\int_{\Cell[j]} \Sink \dx
  \quad\forall \Cell[j]\in \Supp(\Sink).
\end{equation*}
We use two algorithms contained in the POT
toolbox~\citep{flamary2017pot} to find the OT Plan $\Matr{\OptPlan}$
for the discrete \OTP.  The first algorithm is based on a classical LP
solver and we denote with $\Matr{\OptPlanH}$, this approximated OT
plan.  The second algorithm is based on the the Sinkhorn
regularization of discrete $\Lspace{1}$-\OTP described
in~\citet{Cuturi:2013}.  We denote with $\Matr{\SnkhOptPlanH}$ its
approximate solution, where $\SnkhPar$ indicates the Sinkhorn
relaxation parameter with value $\SnkhPar=8e-4$, which was
experimentally evaluated to avoid algorithm failure. The barycentric
maps of the two plan are given by:
\begin{equation*}% \label{eq:bar-maps}
  \OptMap(LP):=\OptMap[\Matr{\OptPlanH}]
  \quad
  \OptMap(S):=\OptMap[\Matr{\SnkhOptPlanH}].
\end{equation*}
The computation of $\OptMap(LP)$ and $\OptMap(S)$ requires the
solution of~\cref{eq:bari-map}, i.e., a weighted version of the
Fermat-Weber location problem, solved in our case with the algorithm
described in~\citet{Vardi2001}.

\begin{figure}
  \centerline{
    {\includegraphics[width=0.33\textwidth]{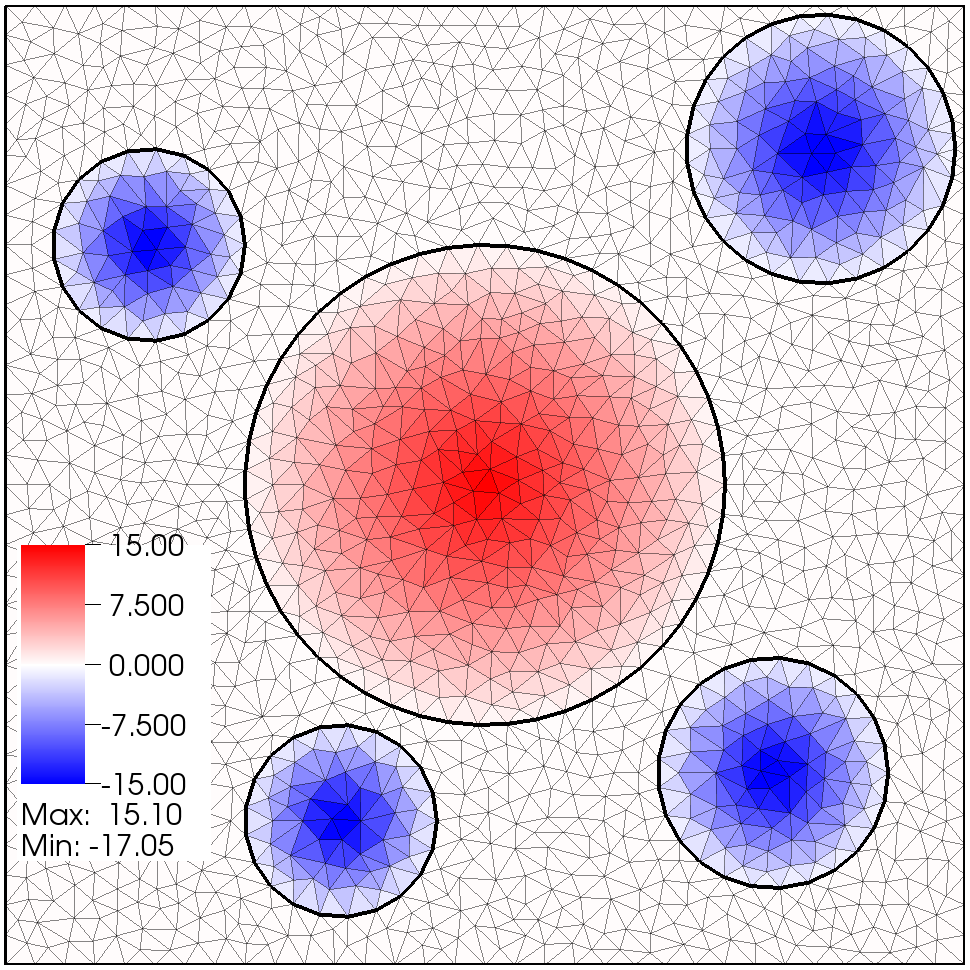}}
    {\includegraphics[width=0.33\textwidth]{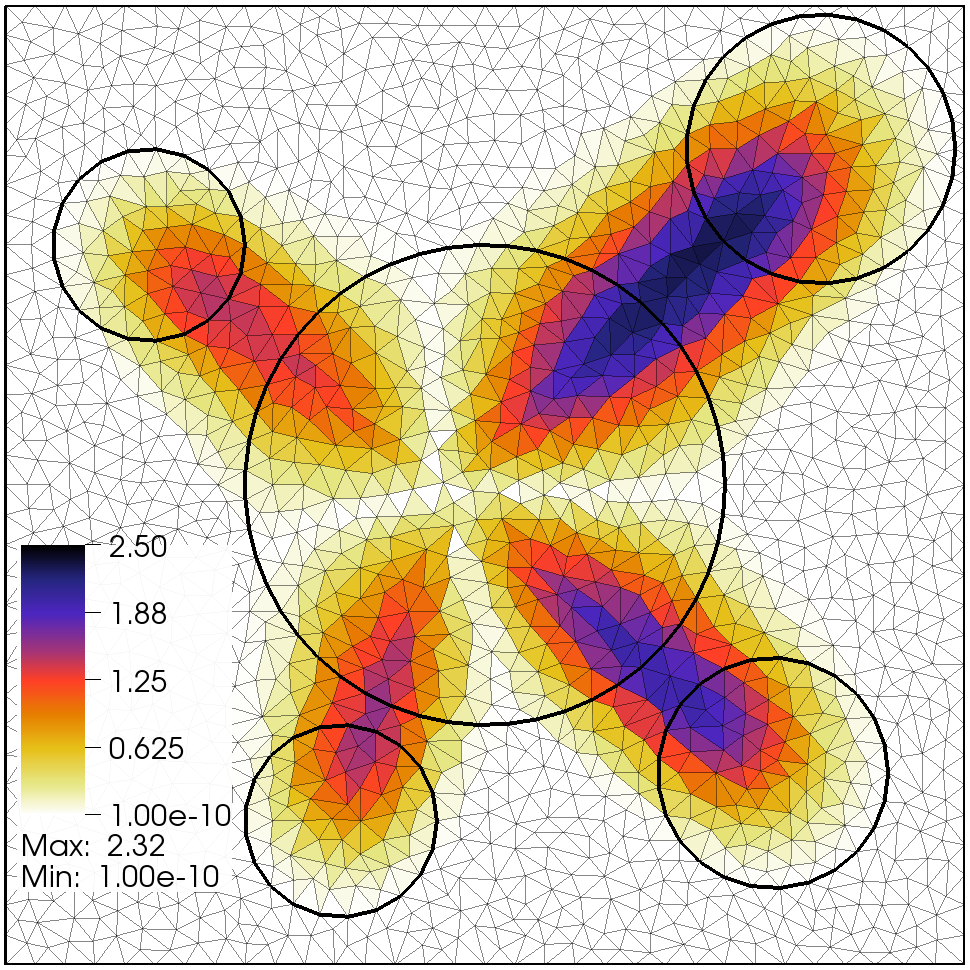}}
    {\includegraphics[width=0.33\textwidth]{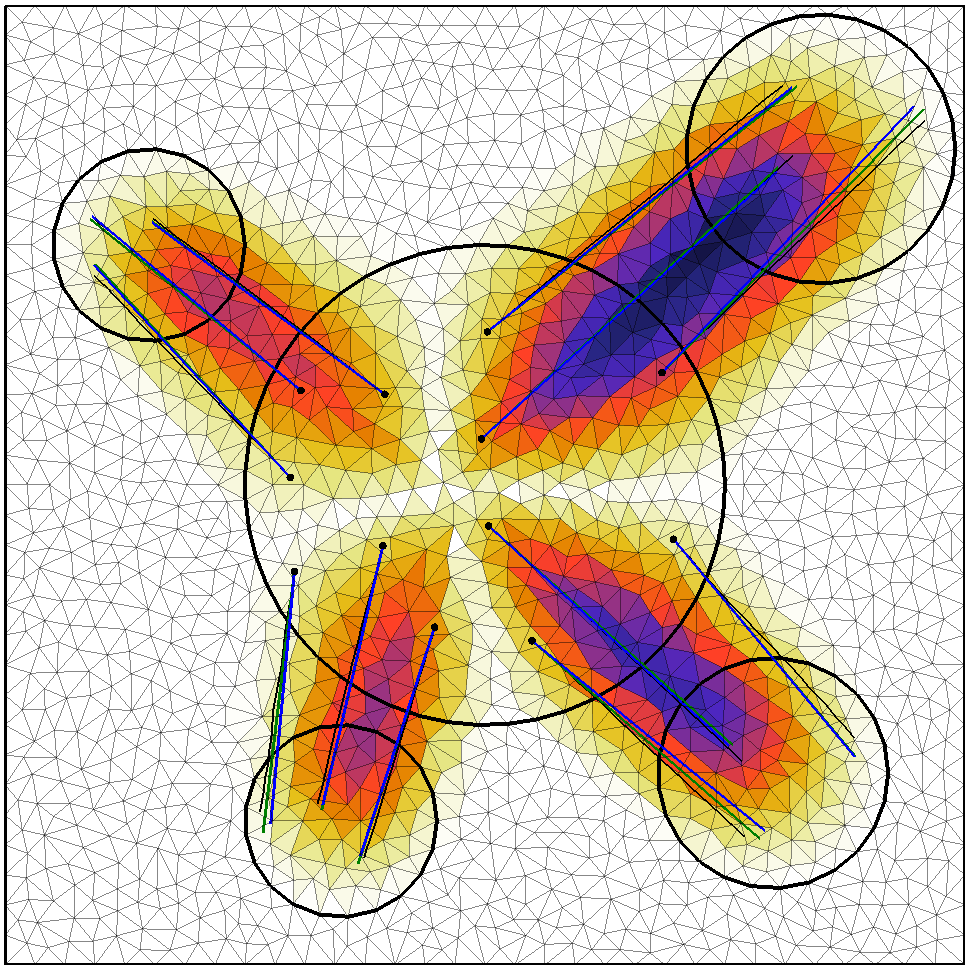}}
  }
   \centerline{
    {\includegraphics[width=0.33\textwidth]{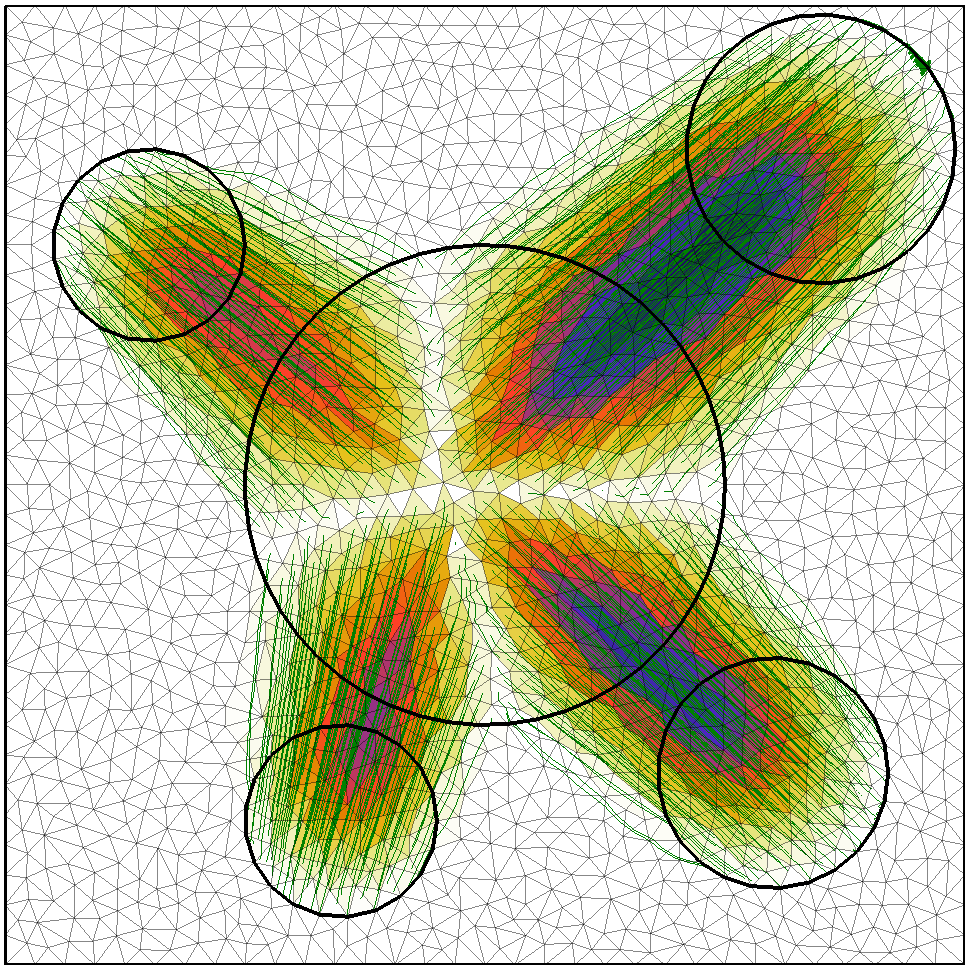}}
    {\includegraphics[width=0.33\textwidth]{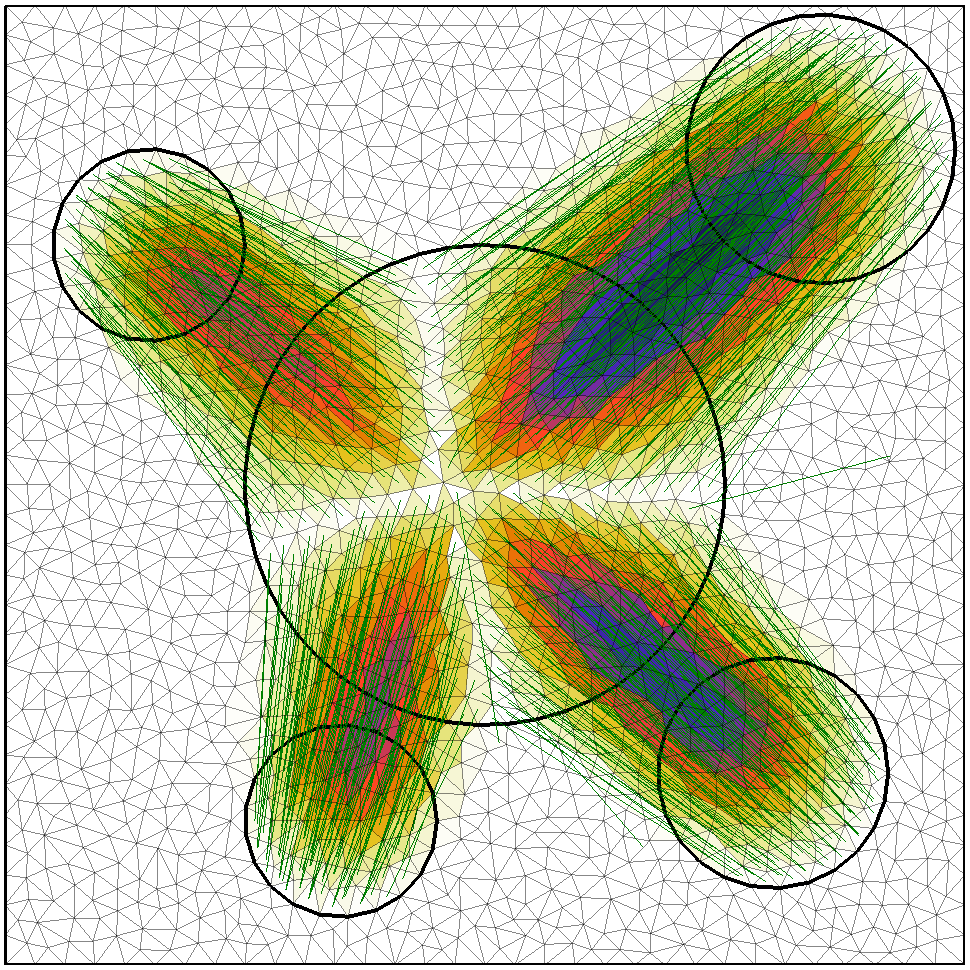}}
    {\includegraphics[width=0.33\textwidth]{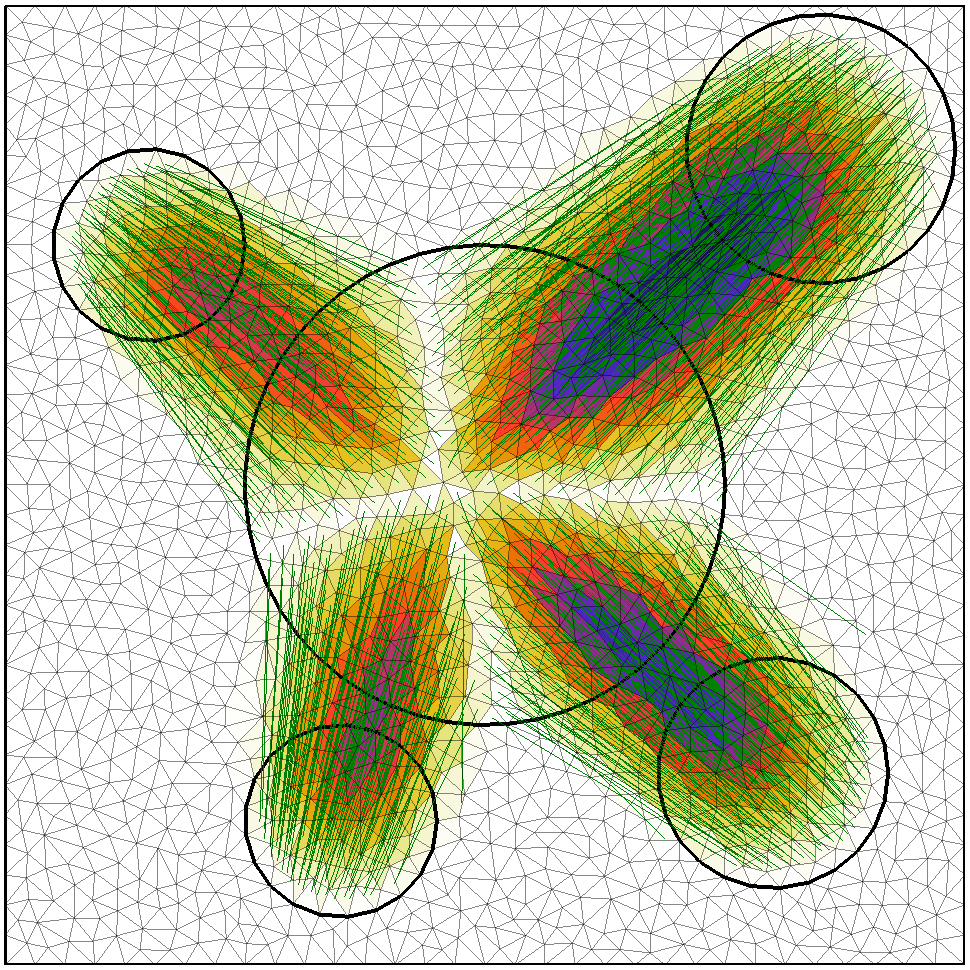}}
  }
  \caption{ Test case~3. The top row contains the problem definition
    with the triangulation $\Triang[\MeshPar]$ ($1699$ nodes and
    $3265$ triangles) used in the solution of \DMK, the source (red)
    and the sink(blue) terms (left); the approximate OT density
    $\OptTdensH$ distribution obtained via \DMK\ (middle) and the
    comparison between OT maps calculated with \DMK, LP, and Sinkhorn
    algorithms for few barycenters. The bottom row contains the full
    OT maps calculated starting from the triangle barycenters by the
    \DMK\ (left), the LP (middle), and the Sinkhorn (right)
    algorithms.}
  \label{fig:maps}
\end{figure}

The spatial distribution of $\OptTdensH$ as calculated by the \DMK\
approach is shown in the top central panel in~\cref{fig:maps}.  We
note the four regions into which the support of $\Source$ is divided.
Each region corresponds to the ``portion'' of $\Source$ that is sent
to each one of the four circles where $\Sink$ is supported. We call
$\Gamma$ the one-dimensional boundary dividing these regions.  In the
right top panel we compare the lines connecting 12 sampling point in
$\Supp(\Source)$ with their image through the maps $\OptMap(\DMK)$
(black), $\OptMap(LP)$ (green) and $\OptMap(S)$ (blue).  The three
approximated maps are qualitatively similar, suggesting that
$\OptTdensH$ can be effectively used to determine transport maps.  The
bottom panels of~\cref{fig:maps} show the full OT maps calculated for
triangle barycenters within $\Supp(\Source)$ with the three considered
methods.  We first note the almost prefect coincidence of the supports
of these lines with the support of $\OptTdensH$.  Almost all the
streamline computed for the map $\OptMapH(\DMK)$ are perfectly
straight, with the end-points on the lines covering $\Supp(\Sink)$.
The computation of the optimal destination of the barycenters close to
$\Gamma$ creates some numerical difficulties for all methods. Indeed,
the image of the $\OptMapH(LP)$ and $\OptMapH(S)$ with domain located
in barycenters of triangles located in $\Gamma$ fall outside
$\Supp(\Sink)$.  For $\OptMapH(\DMK)$, some of the streamlines
starting from these points are not exactly straight, while others
remain stack at the starting point.  Mesh refinement around $\Gamma$
should help to give a better characterization the partitions dividing
$\Supp(\Source)$.  In the \DMK\ model this mesh refinement strategy is
easily accomplished as triangles to be refined are easily identified
by the fact that $\OptTdensH$ tends to zero in the $\Supp(\Source)$.
Another advantage of the \DMK\ method is that, given a new sample
point $\bar{x}\in\Supp(\Source)$, we can compute $\OptMapH(\bar{x})$
without recomputing the pair $(\OptTdensH,\OptPotH)$. The
computational effort required is only the integration of the Cauchy
Problem~\cref{eq:optimal-field}.  The results of this test case show
that, for this case of $\Lspace{1}$-\OTP, our approach does not suffer
of the \emph{out-of samples problem} described
in~\citet{Perrot-et-al:2016}.

\section{Conclusions}

The performance of the proposed finite element method for the solution
of the dynamic Monge-Kantorovich equations has been thoroughly
analyzed experimentally on several test cases.  The results show that
the strategy for solving the \MKEQS\ by searching for the stationary
solution of the dynamic MK problem is highly promising.  These
experiments show that the resulting discrete system achieves optimal
convergence in space and time even when using simple successive
(Picard) linearization schemes.

The enhancement path for the proposed approach is clear. Some of the
issues currently under study include the development of a Newton
method for the solution of the nonlinear system in the case of
implicit Euler time-stepping, to completely exploit the geometric
convergence towards steady state only hinted at in the present paper.
Further improvements can be readily obtained by careful use of a
sequence of meshes with progressively finer resolution as time
increases, with adaptation to the support of the transport density
easily achievable.  The iterative nature of \DMK\ approach allows the
tight control of these computational savings both in the spatial and
in the time discretizations.

Theoretical work is needed to ascertain the formal convergence of the
proposed methods and to determine the exact relationships between the
spatial discretization spaces used for $\PotH$ and $\TdensH$ that
guarantee stability of the approach.  Future studies include the
handling of less regular forcing functions to address the more
interesting problems that can be studied by means of Optimal Transport
theory.  We believe that this work can be a useful starting
point to further developments of time-dependent transport systems.

\section*{Acknowledgments}
This work was partially funded by the the UniPD-SID-2016 project
“Approximation and discretization of PDEs on Manifolds for
Environmental Modeling” and by the EU-H2020 project
``GEOEssential-Essential Variables workflows for resource efficiency
and environmental management'', project of ``The European Network for
Observing our Changing Planet (ERA-PLANET)'', GA 689443.

\bibliographystyle{abbrvnat}
\bibliography{Strings,biblio}

\end{document}